\documentclass[12pt,usenames,dvipsnames]{amsart}

\usepackage{amssymb}
\usepackage{amsmath}
\usepackage{amsthm}
\usepackage{amscd,bm,bbm}
\usepackage{enumitem,graphicx,tikz, pst-node, pst-plot, pstricks}
\usepackage [margin=1.00in,marginparwidth=.75in]{geometry}
\usepackage{hyperref}
\usepackage{color}
\usepackage{xcolor}
\usepackage{cite}
\usepackage{longtable}
\usepackage{marginnote}
\usepackage[normalem]{ulem}
\usepackage{verbatim}

\theoremstyle{plain}
\newtheorem{theorem}{Theorem}
\newtheorem{proposition}{Proposition}[section] 
\newtheorem{corollary}[proposition]{Corollary} 
\newtheorem{lemma}[proposition]{Lemma}
\newtheorem*{blanktheorem}{Theorem}

\theoremstyle{definition}
\newtheorem{definition}[proposition]{Definition}
\newtheorem{notation}[proposition]{Notation}
\newtheorem{remark}[proposition]{Remark} 
 
\newtheorem{assumption}{Assumption}

\theoremstyle{remark}
\newtheorem{qRodrigo}{Question to Rodrigo} 
\newtheorem{qSasha}{Question to Sasha} 
\newtheorem{qPatricia}{Question to Patricia} 
\newtheorem{qMichael}{Question to Michael} 
\numberwithin{equation}{section}

\global\newcount\N \global\N=1
\global\newcount\NM \global\NM=1
 
\long\def\mS#1{ {\marginnote{{\color{OliveGreen}{ 
 [$mS\the\N$]{\global\advance\N by 1}: #1}}}} }

\long\def\mpatd#1{{\marginnote{\textcolor{TealBlue}{[mpatd$\the\N$]{\global\advance\N by 1}: #1}}} }
\long\def\mrem#1{{\marginnote{\textcolor{Mahogany}{[rem$\the\NM$]{\global\advance\NM by 1}: #1}}} }

\def\cphi{\circ \phi^{-1}}

\DeclareMathOperator{\supp}{supp}
\DeclareMathOperator{\diam}{diam}
\DeclareMathOperator{\clos}{clos}

\DeclareMathOperator{\im}{Im}

\DeclareMathOperator{\rank}{rank}

\newcommand{\mbbR}{{\mathbb R}}
\newcommand{\Rd}{\ensuremath{{\mathbb R^d}}}

\newcommand{\mbbP}{{\mathbb P}}
\newcommand{\B}{{\mathcal{B}}}
\newcommand{\D}{\mathcal{D}}
\newcommand{\E}{\mathcal{E}}
\newcommand{\mcH}{\mathcal{H}}

\newcommand{\orb}{\operatorname{orb}}

\newcommand{\OO}{\ensuremath{\Omega}}
\newcommand{\Koo}{U}
\newcommand{\OLe}{\ensuremath{O_{\Lambda,\varepsilon}}}
\newcommand{\CLe}{\ensuremath{\mathcal{C}_{\Lambda,\varepsilon}}}

\newcommand{\vx}{{\vec{x}}}
\newcommand{\vy}{{\vec{y}}}
\newcommand{\vt}{{\vec{\,t}}}
\newcommand{\vzer}{{\vec{\,0}}}
\newcommand{\vv}{{\vec{v}}}

\newcommand{\ve}{{\vec{e}}}
\newcommand{\vs}{{\vec{s}}}
\newcommand{\vW}{{\vec{W}}}
\newcommand{\bl}{{\ensuremath{\bar{\Lambda}}}}

\newcommand{\PHi}{{\ensuremath{\mbox{$\Phi$\hskip.017em\llap{$\Phi$}\hskip.017em\llap{$\overset{{\genfrac{}{}{0pt}{}{}{{\Large*}}}}{\Phi}$}\hskip.017em\llap{$\Phi$}\hskip.017em\llap{$ {\Phi}$}}}}}
\renewcommand{\PHi}{{\ensuremath{\mbox{$\Phi$\hskip.01em\llap{$\Phi$}\hskip.01em\llap{$\Phi$}\hskip.01em\llap{$\Phi$}\hskip.01em\llap{$\Phi$}}}}}

\newcommand{\PHii}{{\ensuremath{\PHi_i}}}

\begin{document}

\title{Canonical diffusions on   pattern spaces of aperiodic Delone sets}

\author{Patricia Alonso-Ruiz$^{1}$}

\address{$^{1,3}$ Department of Mathematics, University of Connecticut, Storrs, CT 06269, USA}

\email{\href{mailto:patricia.alonso-ruiz@uconn.edu}{patricia.alonso-ruiz@uconn.edu}}
\urladdr{\url{http://www.math.uconn.edu/~alonsoruiz/}}

\email{\href{mailto:alexander.teplyaev@uconn.edu}{alexander.teplyaev@uconn.edu}}
\urladdr{\url{http://www.math.uconn.edu/~teplyaev/}}

\author{Michael Hinz$^2$}

\address{$^2$ Fakult\"at f\"ur Mathematik, Universit\"at Bielefeld, 
 33501 Bielefeld, Germany}

\email{\href{mailto:mhinz@math.uni-bielefeld.de}{mhinz@math.uni-bielefeld.de}}
\urladdr{\url{https://www.math.uni-bielefeld.de/~mhinz/home.html}}

\author{Alexander Teplyaev$^3$}

\author{Rodrigo Trevi\~{n}o$^4$}

\address{$^4$ Department of Mathematics, 
 University of Maryland
 College Park, MD 20742, 
 USA}
\email{ \href{mailto:rodrigo@math.umd.edu}{rodrigo@math.umd.edu}\ }
\urladdr{\url{http://trevino.cat} \ }

\thanks{\mbox{Research supported in part by:} $(^1)$ the Feodor Lynen Fellowship, Alexander von Humboldt Foundation;
$(^2)$ the DFG IRTG 2235: `Searching for the regular in the irregular: Analysis of singular and random systems' and by the `Fractal Geometry and Dynamics' program, Institut Mittag-Leffler, Stockholm, 2017; $(^{1,3})$ the NSF 
 (DMS-1613025, 
 DMS-1262929, 
 DMS-1700187); 
$(^{2,3})$ the DFG CRC 701: `Spectral Structures and Topological Methods in Mathematics'; $(^4)$ NSF grant DMS-1665100.}

\date{\today}

\begin{abstract}
We consider pattern spaces of aperiodic and repetitive Delone sets of finite local complexity. These spaces are compact metric spaces and constitute a special class of foliated spaces. We define new Sobolev spaces with respect to the unique invariant measure and prove the existence of the unitary Schr\"odinger semigroup, which   in physics terms describe the evolution of  phasons. We define and study natural leafwise diffusion processes on these pattern spaces. These processes have Feller, but lack  strong Feller and hypercontractivity properties, and  heat kernels do not exist. The associated Dirichlet forms are regular, strongly local, irreducible and recurrent, but not strictly local.  For harmonic functions we prove new Liouville and  Helmholtz-Hodge type theorems.


\end{abstract}
\maketitle


{\ }\newpage\thispagestyle{empty}
\setcounter{page}{1}
\tableofcontents\newpage
\thispagestyle{empty}\setcounter{page}{1}

 \begin{table}[h] \caption{Summary of notation used throughout the paper. 
 }
 \begin{center}
 \begin{longtable}{|p{0.14\textwidth}|p{0.60\textwidth}|p{0.17\textwidth}|}
 \hline
 $\Omega_{\Lambda_0}=\Omega$ & the pattern space of a Delone set $\Lambda_0$ & Definition~\ref{D:Omega}\\\hline
 $\varrho$ & metric on $\Omega$ & Definition~\ref{D:rho}\\\hline
 $\orb(\Lambda)$ & orbit of $\Lambda$ & Remark~\ref{R:orbits}\\\hline
 $\varrho_{\orb}$ & orbit metric & Definition~\ref{D:orbit metric}\\\hline
 $h_{\Lambda}$ & orbit homeomorphism & \eqref{E:orbithomeo}\\\hline
 $\mathcal{C}_{\Lambda,\varepsilon}$ & a transversal $\varepsilon$-cylinder set & Definition \ref{D:Cilinder}\\\hline
 $\OLe$ & translated $\varepsilon$-cylinder set & \eqref{E:transcylsets}\\\hline 
 $\mu$& unique $\mathbb{R}^d$-invariant probability measure on $\Omega$&\eqref{E:ergodic}\\ \hline
 $p_{\mbbR^d}(t,\vs,)$& Gaussian heat kernel on $\mbbR^d$& Definition~\ref{D:Gaussian}\\\hline
 $(T_t)_{t\geq 0}$ & Feller semigroup of $X_t$& Definition~\ref{D:sg}\\\hline
 $X_t$ & canonical diffusion process on $\Omega$& Theorem~\ref{T:main_Feller} \\\hline
 $\lambda^d_{\Omega}$& not $\sigma$-finite pushforward of the Lebesgue measure~$\lambda^d$&
 \eqref{e-ld}\\\hline
 $p_\Omega(t,\cdot,\cdot)$& 
 heat kernel 
 with respect to the 
 measure $\lambda^d_\Omega$&
 \eqref{e-lp}\\\hline
 $C^k(\Omega)$ & $C^k$-functions on $\Omega$ & Definition~\ref{D:Ck} \\\hline
 $C^k_{\orb}(\Omega)$ & $\varrho_{\orb}$-continuous functions that are orbit-wise $C^k$ & Definition \ref{D:Ckorb} \\\hline
 $C^k_{tlc}(\Omega)$ & transversally locally constant $C^k$-functions on $\Omega$ & Definition~\ref{D:tlc} \\\hline
 $C^k_{\Lambda}(\mathbb{R}^d)$ & $\Lambda$-equivariant $C^k$-functions on $\mathbb{R}^d$ & Definition~\ref{D:eqvar}\\\hline
 $\nabla$ & canonical gradient operator on $\Omega$ & \eqref{E:gradient}\\\hline
 $\Delta$ & canonical Laplacian on $\Omega$ & \eqref{E:Laplace}\\\hline
 $\mathcal{L}_{C(\Omega)}$ & Feller generator of $X_t$ & \eqref{E:Fellergen}\\\hline
 $(P_t)_{t>0}$ & extension of $(T_t)_{t>0}$ to a semigroup on $L^2(\Omega,\mu)$ & Notation \ref{nPt}\\\hline
 $\mathcal{L}$& Generator of $(P_t)_{t>0}$ & \eqref{E:L2gen}\\\hline
 $\mathcal{E}$& Dirichlet form of $(P_t)_{t>0}$ on $\Omega$&\eqref{E:Dform}\\\hline
 $\mathcal{B}(\Omega),\mathcal{B}(\mathbb{R}^d)$ & Borel measurable functions on $\Omega$, $\mathbb{R}^d$ & 
 \\\hline
 $bS$ & subspace of bounded elements of a function space $S$ & \\\hline 
 $\mathcal{W}^{k,2}(\Omega,\mu)$ & canonical Sobolev spaces on $\Omega$ & Definition~\ref{def-sobolev}\\\hline
 \end{longtable}
\end{center}\end{table}


\newpage

	
\section{Introduction}

Pattern spaces of aperiodic and repetitive Delone sets of finite local complexity 
are foliated spaces that are important in mathematics and mathematical physics, 
see for instance \cite{BG:book1,kellendonk2015mathematics,MSch06, sadun:book}. 
Our work is the first to define and study the natural 
symmetric diffusion processes on such spaces, and  
to define the unique 
self-adjoint extension $\mathcal{L}$ of the Laplacian $\Delta$. This allows to define   
the unitary Schr\"odinger semigroup $e^{i\mathcal{L}}$ that, under natural physical assumptions, describes the evolution of \emph{phasons}, quasiparticles existing in quasicrystals. In physical literature this is a well studied object, see \cite{A5,A8,divincenzo1999quasicrystals} and references therein, but previously there have not been a mathematical proof of existence of such particles. Besides the  functional analysis, our  probabilistic construction allows to use the tools of stochastic analysis to obtain the Feynman-Kac formulas for  phasons, including the Feynman-Kac formulas with a magnetic field \cite{HT13m,H15m}. 
In some sense our study is dual to the study of aperiodic discrete  
Schr\"odinger operators, see \cite[and references therein]{DamanikGorodetskiYessen,AvilaDamanikZhang,JitomirskayaLiu}.
Technically our work is difficult because classical methods do not apply, as we comment throughout the article. 

There exist many studies of various properties of pattern spaces, specific examples, and their physical relevance. However, at present, the 
literature dealing with the relation between diffusion processes and their semigroups, and the related functional analysis, is not well developed. 
Our work aims to bridge this gap, and has three goals. The first is to begin with a simple definition and, based on it, to give rather straightforward arguments to prove some basic 
important facts: the presence of the Feller property and the absence of the strong Feller property. 
The second goal is to investigate the situation when the existence and uniqueness of an invariant ergodic probability measure is known. In this case, the diffusion is symmetric with respect to this measure and exploiting this fact, we can prove new results about Sobolev spaces, quadratic (Dirichlet) forms, self-adjoint Laplacians, 
and harmonic functions. Third, in the one dimensional case of Delone sets on the real line, we obtain the Helmholtz-Hodge decomposition for $L^2$ vector fields. 
This Helmholtz-Hodge decomposition is new and substantially different from the more well established results dealing with the de Rham type cohomology. 

Delone sets arise from various models such as tilings or quasicrystals, see for example 
\cite{A1,A2,A3,
	A5,
	BG:book1, BaakeLenz17, BBG06, Gardner77, kellendonk2015mathematics, KP00, LP03, SBGC82,ACCDP:survey,Bellissard00,Bellissard15}. One can introduce a certain metric on all Delone sets (Definition \ref{D:rho}), and consider the closure of the $\mathbb{R}^d$-orbit of a given Delone set with respect to this metric. In this way, a Delone set of finite local complexity produces a compact metric space, called the pattern space or continuous hull. Delone sets that in addition are aperiodic and repetitive lead to pattern spaces with an interesting topological structure, as they are not manifolds or foliated manifolds, but examples of more complicated 
foliated spaces, \cite{BG12,CanCon00, MSch06}. More specifically, they are locally homeomorphic to products of Cantor sets and Euclidean balls. Therefore, although they are neither a manifold nor a fractal, such pattern spaces have in some sense aspects of both.

Diffusions on foliated manifolds were introduced in \cite{GarnettL} and defined using (truncated) heat kernels of Brownian motion on the leaves of the foliation. The existence of invariant measures for the diffusion was also verified. Further properties of such invariant measures can be found in~\cite{Kaimanovich89}. The paper \cite{Candel} studied diffusions in the more general setup of foliated spaces, see \cite[Definition 2.1]{Candel} or \cite{CanCon00, MSch06}, which covers the pattern spaces considered here. These papers dealt with elliptic operators on smooth functions on foliated spaces and used the Hille-Yosida theorem to verify the existence of associated Feller semigroups. Correcting a gap in \cite{GarnettL}, the paper \cite{Candel} also proved the existence of invariant measures for these semigroups and used Kolmogorov's theorem to actually construct related Feller diffusion processes on the foliated spaces (see also \cite[Chapter 2]{CanCon03}). A complementary approach to diffusions on foliated spaces using stochastic differential equations was recently provided in \cite{Su15}, where the Feller property of the diffusions was verified. 

Pseudodifferential operators acting in the transversal (`Cantor-set') direction of pattern spaces have been investigated in \cite{PB09} using tools from noncommutative geometry. Physical applications are discussed in \cite{kellendonk2015mathematics}.
 
In the present paper we give a different, probably the most robust, definition of diffusions on pattern spaces in comparison to \cite{Candel} and \cite{Su15}, although all natural approaches are related to the canonical differentiation along the orbits considered in \cite{Kellendonk:PEC, KP:RS}, see Definitions \ref{D:sg} and \ref{D:Ck} and formula (\ref{E:process}). 
Starting with our definition of the diffusion, we first give a short alternative proof of the Feller property, now based on the metric on the pattern space. We verify in Theorem~\ref{T:main_Feller} that the diffusion does not possess the strong Feller property, a simple result that we could not find in the existing literature. Afterwards, we assume that the action of $\mathbb{R}^d$ on the pattern space is uniquely ergodic, a situation that covers many typical examples of pattern spaces, see Remark \ref{R:uniqueergodic}. This adds a new perspective and leads to new results not previously covered in \cite{GarnettL, Candel, CanCon00, CanCon03}, because in this case the unique invariant and ergodic probability measure $\mu$ provides a volume measure on the pattern space with respect to which the diffusion is easily seen to be symmetric. Moreover, $\mu$ is invariant with respect to the Feller semigroup and convergence to equilibrium holds in the weak sense. As a consequence of the symmetry, we can consider the $L^2$-generator of the diffusion, which is a self-adjoint extension of the natural Laplacian defined on smooth functions, Theorem~\ref{T:main_L2sg}. We observe that the $L^2$-semigroup does not admit a heat kernel with respect to $\mu$. The corresponding Dirichlet form is regular strongly local, irreducible and recurrent, but the regularity in this case is not enough to consider the natural regularization of the generalized eigenfunctions, see \cite{LT16} and references therein.

It is important to note that 
the spectrum of the Laplacian is connected with the spectrum of the Koopman operators associated with the action of $\mathbb{R}^d$ on the pattern space, see Theorem~\ref{T:main_L2sg} (\ref{T:main_Koopman}) and Corollary~\ref{c:S}. This provides a connection to recent results on spectral notions of aperiodic order \cite{BaakeLenz17, BaakeLenzvanEnter15}. A Liouville theorem for measurable harmonic functions is proved in 
Theorem~\ref{T:main_Liouville}. For pattern spaces of Delone sets in dimension one, we provide in Theorem \ref{T:main_Hodge} a Helmholtz-Hodge decomposition and verify that the orthogonal complement of gradient fields consists only of constants. This result can be rephrased by saying that the space of harmonic vector fields is one-dimensional, in contrast to topological results such as the fact that for the Fibonacci-tiling in dimension one, the first \v{C}ech cohomology is two-dimensional see \cite{sadun:book, Sadun2015} and Remark~\ref{R:contrast}.  

In Section~\ref{S:FellSmg} we consider the properties of the Feller semigroup and its 
generator. Section~\ref{S:L2} deals with the $L^2$-semigroup with respect to the unique ergodic invariant measure, the generator of this semigroup and associated quadratic forms. In Section~\ref{S:harmonic} we discuss harmonic functions and prove the Liouville theorem. Section~\ref{S:Hodge} deals with the Helmholtz-Hodge decomposition in the one-dimensional case. Basic facts on differentiable functions and the Hodge star operators for Dirichlet forms 
are collected in the appendix.
This is connected to the 
tangential, dynamical, weakly
patterned equivariant cohomologies, but 
different  
from the Cech, strongly patterned equivariant, transversally locally
constant dynamical cohomologies (see Theorems 20 and 23 in 
\cite{KP:RS} and also \cite{MSch06}).



We would like to mention two aspects which make analysis on a pattern space difficult. One aspect is that we have to consider two metrics, $\varrho$ from Definition~\ref{D:rho} and $\varrho_{\orb}$ from Definition~\ref{D:orbit metric}. The pattern space is compact in the non-geodesic metric 
$\varrho$, but the topology induced by the geodesic metric $\varrho_{\orb}$ is non-compact and non-separable. Moreover, in most of the paper, we have to deal with two measures: a unique ergodic probability measure 	$\mu$ from~\eqref{E:ergodic}, and a non $\sigma$-finite measure $\lambda^d_{\Omega}$, which is the pushforward of the $d$-dimensional Lebesgue measure~$\lambda^d$ 
under the orbit homeomorphisms~\eqref{E:orbithomeo}. Thus, even though $\mu \ll \lambda^d_{\Omega} $, the Radon-Nikodym theorem does not apply.  Note also that, by continuity, the natural diffusion process almost surely has to be confined to the  orbit of its starting point, 
implying that the diffusion is not strong Feller and there is no heat kernel with respect to $\mu$. 
As a deep consequence of this aspect of stochastic analysis, the classical theory of 
Carlen-Kusuoka-Stroock \cite{CKS87} does not apply in our case. The Dirichlet form is not strictly local because the intrinsic metric does not 
generate the topology of the space, unlike in more well studied geometric analysis situations \cite{Sturm1,Sturm2,Cheeger,HKST,KSZ14,KZ12} (see \cite{BCR11,Gross67,Gross67jfa,Gross75,Hansen,BG12} for related functional analysis and probabilistic discussion). 

\subsection*{Acknowledgements} 
We thank 
Eric Akkermans, 
Jean Bellissard, 
Lucian Beznea, 
David Damanik, 
Jozef Dodziuk, 
Dmitry Dolgopyat, 
Maria Gordina, 
Daniel Lenz, 
Luke Rogers, 
and 
Claude Schochet 
for 
helpful and inspiring discussions.

\section{Definitions and Notation}

\subsection{Definitions of Delone sets and pattern spaces}
A subset $\Lambda\subset \mathbb{R}^d$ is called \emph{uniformly discrete} if there exists a number $\varepsilon>0$ such that for any two distinct $\vx,\vy\in\Lambda$ we have that $|\vx-\vy|>\varepsilon$. It is called \emph{relatively dense} if there exists $R>0$ such that $\Lambda\cap B_R(\vx) \neq \varnothing$ for any $\vx\in\mathbb{R}^d$. The set $\Lambda$ is a \emph{Delone set} if it is both relatively dense and uniformly discrete.

A finite subset $P\subset \Lambda$ is called a \emph{cluster} of $\Lambda$. A Delone set has \emph{finite local complexity} if for every $R>0$ there exist finitely many clusters $P_1,\dots, P_{n_R^{\vphantom{A}}}$ such that for any $\vx\in\mathbb{R}^d$ there is an $i$ such that the set $B_R(\vx)\cap \Lambda$ is translation-equivalent to $P_i$. A Delone set $\Lambda$ is \emph{aperiodic} if $\Lambda-\vt=\Lambda$ implies $\vt=\vec{0}$. It is \emph{repetitive} if for any cluster $P\subset \Lambda$ there exists $R_P>0$ such that for any $\vx\in\mathbb{R}^d$ the cluster $B_{R_P}(\vx)\cap\Lambda$ contains a cluster which is translation-equivalent to $P$. 

\begin{definition}\label{D:rho}
Let $\Lambda_0\subset\mathbb{R}^d$ be a Delone set and denote by $\varphi_\vt\:(\Lambda_0) = \Lambda_0-\vt$ its translation by the vector $\vt\in\mathbb{R}^d$. For any two translates $\Lambda_1$ and $\Lambda_2$ of $\Lambda_0$ define
\[
\varrho(\Lambda_1,\Lambda_2) = \min \{\bar{\varrho}(\Lambda_1,\Lambda_2), 2^{-1/2}\},
\]
where
\[
\bar{\varrho}(\Lambda_1,\Lambda_2) = \inf\{ \varepsilon>0~\colon~\exists~\vs,\vt\in B_\varepsilon(\vec{0})~\text{such that}~B_{\frac{1}{\varepsilon}}(\vec{0})\cap \varphi_\vs(\Lambda_1) = B_{\frac{1}{\varepsilon}}(\vec{0})\cap \varphi_\vt(\Lambda_2) \}.
\]
\end{definition}
A proof that $\varrho$ is a metric on the set of Delone sets in $\mathbb{R}^d$ can be found in \cite{LMS}. Replacing $2^{-1/2}$ by
any positive number less than $2^{-1/2}$ the function $\varrho$ would still be a metric and so would the pullback of the Hausdorff metric on the set of closed subsets of $S^d$ under the stereographic projection $\Pi_d:\mathbb{R}^d\rightarrow S^d$. Different choices of this number lead to different metrics, but they all generate the same topology. For convenience we stick to the formulation in Definition \ref{D:rho}.


\begin{definition}\label{D:Omega}
Let $\Lambda_0\subset\mathbb{R}^d$ be a Delone set. The \emph{pattern space (hull)} of $\Lambda_0$ is the closure of the set of translates of $\Lambda_0$ with respect to the metric $\varrho$, i.e.
\[
\OO=\Omega_{\Lambda_0}=\overline{\left\lbrace \varphi_\vt\:(\Lambda_0): \vt\in\mathbb{R}^d\right\rbrace}.
\]
\end{definition}

\begin{remark}
 Pattern spaces are examples of \emph{inverse limit spaces}. Examples of these types of spaces are solenoids, although pattern spaces are not solenoids. We refer the reader to \cite[\S 2]{sadun:book} for further details. Connections to Dirichlet structures (see \cite[\S 2.3]{BH91}) will be discussed in a follow-up paper.
\end{remark}

\begin{remark}\label{R:orbits}
Given an aperiodic and repetitive Delone set $\Lambda_0$ of finite local complexity, every $\Lambda\in\Omega_{\Lambda_0}$ is also an aperiodic and repetitive Delone set of finite local complexity. Moreover, the path connected component of a point $\Lambda\in\Omega_{\Lambda_0}$ coincides with its \textit{orbit}
\[
\orb(\Lambda)=\{\varphi_{\vt\,}(\Lambda)~\colon~\vt\in\mbbR^d\},
\]
that is homeomorphic to $\mathbb{R}^d$ via
\begin{align}
h_\Lambda\colon&\mbbR^d\longrightarrow\orb(\Lambda) \label{E:orbithomeo}\\
&\;\vt\;\;\longmapsto\varphi_{\vt\,}(\Lambda).\nonumber
\end{align}
In fact, $\orb(\Lambda)$ is naturally isometric to $\mathbb{R}^d$ in a certain metric $\varrho_{\orb}$, see Definition \ref{e-rhoh} and Lemma \ref{L:Horizontal} below. 
For any $f\colon\Omega\to\mathbb{R}$ and $\Lambda\in \Omega$ we consider the function $h_{\Lambda}^\ast f$ defined on $\mathbb{R}^d$ by 
\begin{equation}\label{E:algiso}
h_{\Lambda}^\ast f(\vt\,): = f\circ h_{\Lambda}(\vt\,), \qquad \vt\in\mathbb{R}^d.
\end{equation}
\end{remark}

A \emph{period} of $\Lambda\subset \mathbb{R}^d$ is any vector $\vec{v}\in\mathbb{R}^d$ such that $\varphi_{\vec{v}}(\Lambda) = \Lambda$. The set of periods $\mathrm{Per}(\Lambda)$ forms a discrete subgroup of $\mathbb{R}^d$. As such, being aperiodic is therefore equivalent to $\mathrm{Per}(\Lambda) = \{\vec{0}\}$. If the rank of the group is maximal (i.e. $d$), then $\mathrm{Per}(\Lambda)$ is a complete lattice in $\mathbb{R}^d$ and the pattern space is the flat torus $\mathbb{R}^d/\mathrm{Per}(\Lambda)$. 

From now on, we will work under the following assumption.
\begin{assumption}\label{A:flcapr}
The underlying Delone set $\Lambda_0\subset \mathbb{R}^d$ has finite local complexity, is aperiodic and repetitive. 
\end{assumption}

Under this assumption, $(\Omega_{\Lambda_0},\varrho)$ is a compact metric space that is connected but not path connected. Moreover, repetitivity implies that 
\begin{equation}\label{E:minimal}
\text{for any }\Lambda\in\Omega_{\Lambda_0}\text{ we have }\Omega_{\Lambda}=\Omega_{\Lambda_0}.
\end{equation}
Since the set $\Lambda_0$ is fixed, to simplify notation we write 
\[\Omega:=\Omega_{\Lambda_0}\] 
and only refer explicitly to the underlying Delone set when confusion may occur. 

\subsection{Local product structure of pattern spaces}\label{subsec:LocalStru} A more detailed look at the geometry of $\Omega$ allows to consider a useful a local product structure.

\begin{definition}\label{D:Cilinder}
For each $\Lambda\in \Omega$ and $\varepsilon>0$ a set of the form
\begin{equation*}
\mathcal{C}_{\Lambda,\varepsilon}:=\left\lbrace \Lambda'\in\Omega: B_{\frac{1}{\varepsilon}}(\vec{0})\cap \Lambda'=B_{\frac{1}{\varepsilon}}(\vec{0})\cap \Lambda\right\rbrace
\end{equation*}
is called a transversal \emph{$\varepsilon$-cylinder set} or \emph{$\varepsilon$-transversal at $\Lambda$}.
\end{definition}
The \emph{canonical transversal} of $\Omega$ is defined as $\mho:=\{ \Lambda \in\Omega: \vec{0}\in\Lambda\}$. For any $\Lambda\in\mho$ the set $\mathcal{C}_{\Lambda,\varepsilon}$ is a clopen subset of $\mho$, and the topology of $\mho$ is generated by clopen sets of that type. Moreover, the canonical transversal $\mho$ and the cylinder sets $\{\mathcal{C}_{\Lambda,\varepsilon}\colon~\Lambda\in\Omega\}$ are Cantor sets, see \cite{KP00}. 

Due to the aperiodicity of $\Lambda_0$, the pattern space $(\Omega,\varrho)$ is not a manifold, yet it has a useful local (product) structure. The following notation will be important throughout the paper (see \cite{Kellendonk:PEC,KP00} for the background). 

\begin{notation}
For any $\Lambda\in\Omega$ and sufficiently small $\varepsilon>0$ the translations of cylinder sets
\begin{equation}\label{E:transcylsets}
O_{\Lambda,\varepsilon}=\{\varphi_\vt\:(\Lambda')~\colon~\Lambda'\in\CLe,\vt\in B_\varepsilon(\vec{0})\}
\end{equation}
form a base of the topology of $\Omega$, see e.g. \cite[Theorem 8]{FHK:CCPT}. In addition, for any $\Lambda\in\Omega$, 
the translated cylinder set $\OLe\subset\Omega$ is the homeomorphic image under $$(\Lambda',t)\mapsto \varphi_\vt\,(\Lambda')$$ of the product $\CLe\times B_\varepsilon(\vec{0})$ of a Cantor set $\CLe$ as in Definition~\ref{D:Cilinder} and a ball $B_\varepsilon(\vec{0})\subset\mbbR^d$. We denote the inverse of this homeomorphism by 
\begin{align}\label{e-phi} 
\phi_{\Lambda,\varepsilon}\colon&\quad \OLe\quad\longrightarrow\quad \CLe\times B_\varepsilon(\vec{0})\nonumber\\
& \bl=\varphi_{\vt\,}(\Lambda') \longmapsto \quad(\Lambda',\vt\:),
\end{align}
which may be considered a foliated chart map \cite{MSch06}. 

The local product structure of $(\Omega,\rho)$ allows to identify locally any function $f\colon\Omega\to\mbbR$ with a function of two variables. Let $\OLe$ be a translated cylinder set as in~\eqref{E:transcylsets} with $\Lambda\in\Omega$ that is mapped homeomorphically onto the product $\mathcal{C}\times B$, where $\mathcal{C}=\CLe$ is given by Definition~\ref{D:Cilinder} and $B=B_\varepsilon(\vec{0})$. 
Thus, for any $f\colon\Omega\to\mbbR$ we define $f\cphi \colon\mathcal{C}\times B\to\mbbR$ as
\begin{equation}\label{E:fO}
f\cphi (\Lambda',\vt\:):=f(\bl),\qquad\quad\bl=\varphi_{\vt}\:(\Lambda')\in\OLe.
\end{equation}
Note that sometimes the notation $f^\OLe=f\cphi =(\phi^{-1})^\ast f$ on $\mathcal{C}\times B$, where $(\phi^{-1})^\ast f:= f\circ \phi^{-1}$, also can be used.

Given an open set $\OLe=\phi^{-1}(\mathcal{C}\times B)$ as in~\eqref{E:transcylsets} we write $(\phi^{-1})^\ast (b\mathcal{B}(\mathcal{C})\otimes b\mathcal{B}(B))$ for the space of finite linear combinations of functions defined as products
\begin{equation}\label{e-prod}
f(\bar\Lambda):=f_0(\Lambda')F_0(\vt\,),\quad (\Lambda',\vt\,)=\phi(\overline{\Lambda}),
\end{equation}
with $f_0\in b\mathcal{B}(\mathcal{C})$ and $b\mathcal{B}(B)$. If in this product $F_0$ is compactly supported in $B$, then $f$ is compactly supported in $\OLe$. We can extend it by zero to a bounded Borel function on all of $\Omega$ and denote this extension by the same symbol $f$. 
\end{notation}

\subsection{Unique ergodicity and invariant measures}
The translative action of $\mathbb{R}^d$ makes $\Omega$ into a dynamical system $(\Omega,\mathbb{R}^d)$. Most of our results will be formulated under the assumption that this action is \emph{uniquely ergodic}.

One characterization of unique ergodicity is that there exists a unique $\mathbb{R}^d$-invariant probability measure $\mu$ on $\Omega$ such that for any F\o lner sequence $(A_n)_n$ and any continuous function $f\in C(\Omega)$ we have 
\begin{equation}\label{E:ergodic}
\lim_{n\to\infty}\frac{1}{\lambda^d(A_n)}\int_{A_n} f\circ \varphi_\vt\:(\Lambda)d\vt=\int_{\Omega}f\:d\mu,
\end{equation}
uniformly for every $\Lambda\in \Omega$. Here, $\lambda^d$ denotes the $d$-dimensional Lebesgue measure. The convergence in (\ref{E:ergodic}) even holds in $L^1(\Omega,\mu)$ for F\o lner sequences and pointwise for $L^1$-functions as long as the F\o lner sets are Euclidean balls, see~\cite[\S 8.5-6]{EW:book}.

\begin{remark}\label{R:uniqueergodic}
Unique ergodicity is a natural assumption since all well-known examples of aperiodic, repetitive Delone sets of finite local complexity define a uniquely ergodic action of $\mathbb{R}^d$ on their pattern space. In fact, it is hard to be aperiodic, repetitive, have finite local complexity and \emph{not} be uniquely ergodic. There are several criteria which will guarantee unique ergodicity: if the Delone set is linearly repetitive \cite[ \S 4.2]{ACCDP:survey}; if the aperiodic Delone set comes from a self-affine tiling (e.g. Penrose tiling\cite{solomyak:SS,Bandt,levine1984quasicrystals,duneau1985quasiperiodic,elser1985crystal,Penrose}); generic cut and project sets also define uniquely ergodic systems. Thus, unique ergodicity is in a sense a typical property. 
 
However, there are known examples of repetitive Delone sets of finite local complexity which do not define uniquely ergodic actions: there is a construction in \cite{CortezNavas} which for any $d>1$ and Choquet simplex $\mathcal{K}$, will produce a repetitive Delone set of finite local complexity for which the $\mathbb{R}^d$ action on its pattern space has a set of invariant probability measures isomorphic to $\mathcal{K}$.
\end{remark}

Unique ergodicity has a description in terms of clusters of patterns, which we now describe. For a Delone set $\Lambda$ of finite local complexity, let $P\subset \Lambda$ be a cluster and $A\subset \mathbb{R}^d$. We denote by
\[
[P:A] = \# \{\vt\in\mathbb{R}^d: \varphi_\vt(P)\subset A\cap \Lambda \},
\]
i.e., the number of translates of $P$ contained in $A$.

A Delone set has \emph{uniform cluster frequencies} relative to a F\o lner sequence $\{A_n\}_n$ if for any non-empty cluster $P$ the limit
\[
\mathrm{freq}(P,\Lambda) = \lim_{n\rightarrow \infty} \frac{[P:\varphi_\vt(A_n)]}{\mathrm{Vol}(A_n)}
\]
exists uniformly in $\vt\in\mathbb{R}^d$. 

\begin{blanktheorem}[{\cite[Theorem 2.7]{LMS}}]
A Delone set $\Lambda$ has uniform cluster frequencies if and only if the action of $\mathbb{R}^d$ on its pattern space of $\Lambda$ is uniquely ergodic.
\end{blanktheorem}

Moreover, the unique $\mathbb{R}^d$-invariant measure $\mu$ on the pattern space of a Delone set with uniform cluster frequencies has a local product structure described as follows. Recall that $O_{\Lambda',\varepsilon}$ denotes a translated cylinder set, the types of which form the basis of the topology of $\Omega$ (see (\ref{E:transcylsets})). Denote by $\Lambda_\varepsilon := \Lambda\cap B_{\varepsilon^{-1}}(\vec{0})$. The following comes from \cite[Corollary 2.8]{LMS}.
\begin{corollary}
 Let $\Lambda$ be a Delone set of finite local complexity and uniform cluster frequencies. For any $\Lambda'\in\Omega$ and $\varepsilon>0$ small enough,
 $$\mu(O_{\Lambda',\varepsilon}) = \mathrm{freq}(\Lambda'_\varepsilon,\Lambda)\cdot \mathrm{Vol}(B_\varepsilon(\vec{0})) .$$
\end{corollary}
As such, in the local foliated charts from~\eqref{E:transcylsets}, the push forward $\mu\circ \phi^{-1}$ under the chart map $\phi$ of the unique $\mathbb{R}^d$-invariant probability measure $\mu$ equals the product measure
\begin{equation}\label{E:localprodmeas}
\mu \circ \phi^{-1}= \nu_\mathcal{C}\times \lambda^d|_B,
\end{equation}
where $\nu_\mathcal{C}$ is the frequency measure on the Cantor set $\mho$ defined through $\nu(\mathcal{C}_{\Lambda, \varepsilon}) = \mathrm{freq}(\Lambda_\varepsilon, \Lambda)$. 


\medskip

\begin{remark}\label{R:pushforward}
On each orbit $\mathcal{O}_\Lambda:=\orb(\Lambda)$, the homeomorphism~\eqref{E:orbithomeo} also induces a measure given by the pushforward of the Lebesgue measure $\lambda^d$. As a result, another measure on $\Omega$ can be defined as
\begin{equation}
\label{e-ld}
\lambda^d_{\Omega}(A):=\sum_{\mathcal{O}_\Lambda\text{ orbit of }\Omega}\lambda^d(h_{\Lambda}^{-1}(A\cap\mathcal{O}_\Lambda)),\qquad A\in\B(\Omega).
\end{equation}
Notice that $\mu$ is absolutely continuous with respect to $\lambda^d_{\Omega}$. However, $\lambda^d_\Omega$ is not $\sigma$-finite and thus the Radon-Nikodym theorem does not apply. This fact will become specially relevant when discussing the existence of heat kernels, see Remark~\ref{R:HKs}.
\end{remark}

\subsection{Orbit-wise metric}
As mentioned in Remark \ref{R:orbits}, the pattern space $\Omega$ can also be considered with a different topology than the one induced by the metric $\varrho$ given in Definition~\ref{D:rho}. This different topology is induced by the following metric.

\begin{definition}\label{D:orbit metric}
The orbit-wise metric $\varrho_{\orb}$ on $\Omega$ is defined as
\begin{equation*}\label{e-rhoh}
 \varrho_{\orb}(\Lambda_1,\Lambda_2) =\begin{cases}
\inf\{~\|\vt\,\|~\colon~\Lambda_1{-}\vt=\Lambda_2\} &\text{if }\Lambda_1,\Lambda_2\in \orb(\Lambda) \text{ for some }
 \Lambda\in\Omega, \\
+\infty &\text{otherwise}.
 \end{cases}
\end{equation*} 
\end{definition}
Since the underlying lattice $\Lambda_0$ is aperiodic, for each $\Lambda\in\Omega=\Omega_{\Lambda_0}$ its orbit $(\orb(\Lambda),\varrho_{\orb})$ is naturally isometric to $\mbbR^d$. The next lemma is obvious.

\begin{lemma}\label{L:Horizontal}\mbox{}
For any $\Lambda\in \Omega$, the space $(\orb(\Lambda),\varrho_{\orb})$ is the image of $\mbbR^d$ under the isometry (with respect to $\varrho_{\orb}$) given by $h_\Lambda$ in~\eqref{E:orbithomeo}.
\end{lemma}

The orbit-wise topology is drastically different from the one induced by the original metric $\varrho$: The space $(\Omega, \varrho_{\orb})$ is not compact, not connected, and not separable. It is easy to see that 
\[
\varrho(\Lambda_1,\Lambda_2)\leqslant2{\varrho_{\orb}}(\Lambda_1,\Lambda_2),
\]
which implies that any continuous function on $\Omega$ is orbit-wise continuous with respect to the topology generated by $\varrho_{\orb}$. In particular, if $f\in C(\Omega)$, then for any $\Lambda\in\Omega$ we have $f\circ h_{\Lambda}\in C(\mathbb{R}^d)$. An orbit-wise continuous function is not necessarily continuous on $\Omega$. Similarly, any Borel measurable function on $\Omega$ is orbit-wise Borel measurable. In particular, if $f\in \mathcal{B}(\Omega)$, then for any $\Lambda\in\Omega$ we have $f\circ h_{\Lambda}\in C(\mathbb{R}^d)$. An orbit-wise Borel measurable function is not necessarily Borel measurable on $\Omega$.

\section{Main results} 
In this section we state the main results of this paper, namely the existence of a natural diffusion, heat semigroup, Sobolev spaces, Dirichlet form on the pattern space $(\Omega,\rho)$ along with their most important properties: Feller but not strong Feller properties, 
commutativity with the Koopman operators, and the identification of the domains of the generators with analogs of the Sobolev spaces. 
 
For a locally compact metric space $E$, a semigroup of linear operators $(T_t)_{t>0}$ acting on the space $b\B(E)$ is called a \emph{Markov semigroup} if all $T_t$ are positive and contractive with respect to the supremum norm. Let $C_0(E)$ denote the space of continuous functions vanishing at infinity, i.e. the space of all $f\in C(E)$ such that for any $\varepsilon>0$ there exists some compact $K\subset E$ with $\sup_{x\in K^c} |f(x)|<\varepsilon$. A Markov semigroup $(T_t)_{t>0}$ is said to be a \emph{Feller semigroup} (or to have the \emph{Feller property}) if it defines a strongly continuous contraction semigroup on $C_0(E)$. It is said to have the \emph{strong Feller property} if each $T_t$ is bounded from $b\B(E)$ into $bC(E)$. It is said to be \emph{conservative} if $T_t\mathbbm{1}=\mathbbm{1}$ for any $t>0$.
For a Markov process $(X^x_t)_{t\geq 0}$ over some probability space $(S,\Sigma,\mathbb{P})$ with state space $E$ and starting point $x\in E$, the \emph{Markov transition semigroup} is $T_tf(x)=\mathbb{E}[f(X_t^x)]$, $f\in b\mathcal{B}(E)$, $t\geqslant0$. If its transition semigroup is Feller and in addition its paths are $\mathbb{P}$-a.s.\ continuous, then the process $(X^x_t)_{t\geq 0}$ is called a \emph{Feller diffusion}. 

\begin{definition}\label{D:Gaussian}
The standard Gaussian density on $\mathbb{R}^d$ is given by
\[
p_{\mathbb{R}^d}(t,\vs\,):=\frac{1}{(2\pi t)^{d/2}}\exp\left\lbrace -\frac{\|\vs\,\|^2}{2t}\right\rbrace, \qquad t>0, \;\vs\,\in\mathbb{R}^d.
\]
\end{definition}

We introduce next the linear operators that will form the transition semigroup of the diffusion. Our definition is similar to the one used in \cite{GarnettL} for foliated manifolds.
 
\begin{definition}\label{D:sg}
For any $t>0$ and $f\in b\B(\Omega)$ define
\begin{equation}\label{E:defsg}
T_t f(\Lambda)=\int_{\mathbb{R}^d} p_{\mathbb{R}^d}(t,\vs\,) f(\varphi_{\vs\,}(\Lambda))\,d\vs, \qquad \Lambda\in \Omega.
\end{equation}
\end{definition}

By the discussion of Borel measurability in the preceding section there are no measurability issues in (\ref{E:defsg}), so that for any $f\in b\mathcal{B}(\Omega)$ and $t>0$ we have $T_tf\in b\mathcal{B}(\Omega)$.
The semigroup property $T_tT_sf=T_{t+s}f$ for $b\mathcal{B}(\Omega)$ is immediate from formula~\eqref{E:defsg} together with the group property $\varphi_{\vec{t}}\circ\varphi_{\vec{s}}=\varphi_{\vs+\vt}$ of the action of $\mathbb{R}^d$ and the Chapman-Kolmogorov equations for the Gaussian density $p_{\mathbb{R}^d}(t,\vec{s})$. 
It is also obvious from (\ref{E:defsg}) that $(T_t)_{t>0}$ is a Markov semigroup and that $T_t\mathbbm{1}=\mathbbm{1}$ for any $t>0$.

The proof of the following theorem and more details are given in Section \ref{S:FellSmg}.

\begin{theorem}\label{T:main_Feller}
Under Assumption \ref{A:flcapr} the following statements hold. 
\begin{enumerate}
\item Let $\vW=(\vW_t)_{t\geq 0}$ be a standard Brownian motion on $\mbbR^d$ over a probability space $(S,\Sigma,\mathbb{P})$, started at zero. For any $\Lambda\in \Omega$, the process 
\begin{equation}\label{E:process}
X^\Lambda_t:=\varphi_{\vW_t}(\Lambda)=\Lambda-\vW_t,\qquad t\geq 0,~\Lambda\in\Omega,
\end{equation}
is a Feller diffusion in $(\Omega,\varrho)$ with transition semigroup $(T_t)_{t>0}$. 
\item The semigroup $(T_t)_{t>0}$ is conservative and the Koopman operators $\Koo_{\vt\,}$ defined on $C(\Omega)$ by 
\begin{equation}\label{E:KoopmanT-def}
\Koo_{\vt\,} f = f\circ \varphi_{\vt\,},\quad \vt\in\mathbb{R}^d, 
\end{equation}
commute with the semigroup $(T_t)_{t>0}$, i.e.
\begin{equation}
\label{E:KoopmanT}
\Koo_{\vt\,} T_{t\vphantom{\vt}} = T_{t\vphantom{\vt}} \Koo_{\vt\,},\quad \vt\in\mathbb{R}^d, t>0,
\end{equation} 
and hence commute with its generator.

\item For any $f\in C^k(\Omega)$, $T_tf\in C^k(\Omega)$. The infinitesimal generator $\mathcal{L}_{C(\Omega)}$ of $(T_t)_{t> 0}$ is a local operator whose domain $\mathcal{D}(\mathcal{L}_{C(\Omega)})$ contains $C^2(\Omega)$. Moreover, for $f\in C^2(\Omega)$ we have 
\[
\mathcal{L}_{C(\Omega)}f=\frac{1}{2}\Delta f
\]
and the space $C_{tlc}^\infty(\Omega)$ defined in Definition~\ref{D:tlc} is a core for $\mathcal{L}_{C(\Omega)}$. In particular, $\mathcal{L}_{C(\Omega)}$ is an extension of $\Delta$, defined in~\eqref{E:Laplace}. 

\item The semigroup $(T_t)_{t>0}$ is not strong Feller. It does admit a symmetric heat kernel 
\[
p_\Omega\colon (0,\infty)\times\Omega\times\Omega\to\mbbR
\]
with respect to the not $\sigma$-finite pushforward measure $\lambda_\Omega$, 
\begin{equation}\label{e-lp}
p_\Omega(t,\Lambda_1,\Lambda_2)=\left\{\begin{array}{ll}
p_{\mbbR^d}(t,h_{\Lambda_1}^{-1}(\Lambda_2))&\text{ if }\Lambda_2\in\orb(\Lambda_1),\\
0&\text{ otherwise.}
\end{array}\right.
\end{equation}
\end{enumerate}
\end{theorem}

The derivatives $D^\alpha$ that we consider in this paper are in some sense (i.e. under the action of the Koopman operators $\Koo_{\vt\,}$) isomorphic to the usual $ \mathbb R^d$ derivatives, and it is natural to call them ``horizontal derivatives''. In particular, we can refer to $\Delta$ as the 
 ``horizontal Laplacian''.

\begin{remark}\label{R:Feller}\mbox{}
\begin{enumerate}
\item[(i)] As mentioned in the introduction, $(\Omega,\varrho)$ is an example of a foliated topological space. Therefore, the existence statement of a Feller diffusion in Theorem \ref{T:main_Feller} is a special case of \cite[Theorem 4.14]{Candel}, also obtained in \cite{Su15} by means of stochastic differential equations. Formula~\eqref{E:defsg} was stated as a result in \cite[Proposition 4.16]{Candel}.
\item[(ii)] Our proof that $(T_t)_{t>0}$ is a Feller semigroup uses only Definition \ref{D:rho} and \eqref{E:defsg}. 
\item[(iii)] The existence of a Feller diffusion with transition semigroup $(T_t)_{t>0}$ on some underlying probability space (a space of $\Omega$-valued paths) follows from Kolmogorov's extension theorem. However, Theorem \ref{T:main_Feller} allows to start from a given Euclidean Brownian motion over a given probability space and yields a Markov process associated with $(T_t)_{t>0}$ defined by the simple formula (\ref{E:process}).
\item[(vi)] The diffusion $(X^\Lambda_t)_{t\geq 0}$ started at $\Lambda$ on $\Omega$ is not Gaussian nor a semimartingale, but its definition (\ref{E:process}) permits to use some of the structural properties of the standard Brownian motion on $\mathbb{R}^d$. For instance, it satisfies a classical It\^o-formula: For any $f\in C^2(\Omega)$ we have $\mathbb{P}$-a.s. that 
\[f(X^\Lambda)=f(\Lambda)+\sum_{i=1}^d\int_0^t \frac{\partial f}{\partial \vec{e}_i}(X_s^\Lambda)dW^i_s+\frac12\int_0^t \Delta f(X_s^\Lambda)ds, \quad 0\leq t<+\infty,\]
where $\vec{W}_t=(W^1_t,...,W^d_t)$. This follows from an application of the usual It\^o formula to $h_\Lambda^\ast f$. It is possible to consider strong solutions to stochastic differential equations and stochastic flows.
\item[(v)] In a similar way one can observe the Feller property of Markov processes on $\Omega$ obtained by (\ref{E:process}) with a more general L\'evy process in place of the Brownian motion. 
\item[(vi)] Although we will not use it, we point out that by~\eqref{E:algiso} it is trivial to see that~\ref{E:defsg}) also defines a semigroup on the bounded Borel functions if $\Omega$ is equipped with the metric $\varrho_{\orb}$. With respect to this topology, the semigroup is both Feller and strong Feller but the metric space $(\Omega,\varrho_{\orb})$ is not compact and not even separable.
\end{enumerate}
\end{remark}
%

The next results are valid under the following standing assumption. 
\begin{assumption}\label{A:uniqueergodic}
The action of $\mathbb{R}^d$ on $\Omega$ is \emph{uniquely ergodic}. If this assumption is satisfied, we denote by $\mu$ the unique ergodic probability measure in~(\ref{E:ergodic}). 
\end{assumption}

A consequence of Assumptions \ref{A:flcapr} and \ref{A:uniqueergodic} is the following Liouville-type theorem, which is proved in Section~\ref{S:harmonic} together with related regularity results, a discussion of finite energy harmonic functions, and the irreducibility of the Dirichlet form. 
The following definition is an adaptation of the classical one for an open subset $O\subset\OO$. 

\begin{definition}\label{def-harm} 
A Borel measurable function $f:\OO\to\mathbb R$ is called \emph{harmonic in} $O\subset\Omega$ if 
for each $\Lambda\in O$, the function $\vt\mapsto h_{\Lambda}^\ast f(\vt\,)$ is harmonic in the $\mathbb R^d $-sense on an open neighborhood of $\vzer$. 
\end{definition}

\begin{theorem}\label{T:main_Liouville}
Under Assumptions \ref{A:flcapr} and \ref{A:uniqueergodic}, let $f:\OO\to\mathbb R$ be a harmonic function. If either $d=1$, or $\inf f>-\infty$, or $\sup f<\infty$, or $f\in L^1(\OO,\mu)$, then $f$ is $\mu$-a.e.\ constant. 
\end{theorem}
 
 
Recall that $\mu$ is said to be an \emph{invariant measure} for $(T_t)_{t>0}$ if $\int_\Omega T_tfd\mu=\int_\Omega fd\mu$ for any $f\in C(\Omega)$ and $t>0$. We use the notation $T_t(\Lambda,A):=T_t\mathbbm{1}_A(\Lambda)$ for $t>0$, $\Lambda\in \Omega$ and $A\subset \Omega$ Borel (see notation in \cite{FOT94,ChFu12}). Note that 
\[
T_t(\Lambda,A)=T_t^*\delta_\Lambda(A),
\]
where $T_t^*\delta_\Lambda$ denotes the adjoint semigroup acting on finite measures, applied to the delta measure $\delta_\Lambda$.

\begin{proposition}
Under Assumptions \ref{A:flcapr} and \ref{A:uniqueergodic} the measure $\mu$ is invariant for $(T_t)_{t>0}$, and for any $\Lambda\in \Omega$ we have 
\begin{equation}\label{E:convtoeq}
\lim_{t\to\infty} T_t(\Lambda,\cdot)=\lim_{t\to\infty} T_t^*\delta_\Lambda=\mu
\end{equation}
in the sense of weak convergence of measures (see Lemma~\ref{L:symmetric}). 
\end{proposition}

\begin{remark}
 The invariance of $\mu$ also follows from \cite[Proposition 5.2]{Candel}, which generalizes \cite[Theorem 1(b)]{GarnettL}. There are many works in dynamical systems connected with this particular proposition and our work in general (see, for instance, \cite{SchmiedingTrevino2018,DolgopyatFernando,Breuillard,KLObook,KV1,
 	G1,G2,DG1,GP1,G3,GW1}) and we hope that these connections will be explored in the future research.  
\end{remark}

The introduction of the measure $\mu$ allows us to consider natural Sobolev spaces on \OO\ and develop the main elements of the general theory of such spaces on \OO.
Note that particularly important in relation to Dirichlet forms are 
\begin{itemize}
 \item $\mathcal{W}^{1,2}(\Omega,\mu)$, the space of all $f\in L^2(\Omega,\mu)$ with $|\nabla f|\in L^2(\Omega,\mu)$ in the distributional sense, and 
 \item $\mathcal{W}^{2,2}(\Omega,\mu)$, the space of all $f\in\mathcal{W}^{1,2}(\Omega,\mu)$ such that $\Delta f$, defined also in distributional sense, is in $L^2(\Omega,\mu)$, 
\end{itemize}see Lemma~\ref{lem-w22} and Section~\ref{S:L2}. By this reason we restrict our attention to spaces $\mathcal{W}^{k,2}$ in the present paper. 
Another reason why we restrict our attention to spaces 
 $\mathcal{W}^{k,2}$ is that typical Sobolev embeddings involving  Sobolev (Banach) spaces $\mathcal{W}^{k,p}$  with $p\neq 2$ do not hold in our situation. 
 
\begin{definition}\label{def-sobolev}
The Sobolev (Hilbert) spaces $ \mathcal{W}^{k,2}(\Omega,\mu)$ are defined as 
\[
\mathcal{W}^{k,2}(\Omega,\mu):=\left\lbrace f\in L^2(\Omega,\mu): D^\alpha f\in L^2(\Omega,\mu)\quad\text{for all }|\alpha|\leq k\right\rbrace,
\]
where $D^\alpha f$ is the distributional derivative from Definition \ref{D:distributionalsense}, with the Sobolev norm defined by 
\begin{equation}\label{e-sobnorm}
 \left\|f\right\|_{\mathcal{W}^{k,2}(\Omega,\mu)}:=
 \left(
 \sum_{|\alpha|\leq k} \left\|D^\alpha f\right\|_{L^2(\Omega,\mu)}^2
 \right)^{1/2}.
\end{equation}
\end{definition}
Note that one can localize Definitions \ref{def-sobolev}, \ref{D:distributionalsense}, and\ref{D:EvansDerivative} for any open $O\subset\OO$ and define the Sobolev spaces $ \mathcal{W}^{k,2}(O,\mu)$. 

Given a measurable function $f:\Omega\to\mathbb{R}$, we write $f\in \mathcal{W}^{k,2}(\Omega,\mu)$ if the $\mu$-class of $f$ is an element of $\mathcal{W}^{k,2}(\Omega,\mu)$. If $f$ is a measurable function with $f\in \mathcal{W}^{k,2}(\Omega,\mu)$ then, as a consequence of this notational agreement and Definition \ref{D:distributionalsense}, we have $f=0$ $\mu$-a.e. in $\Omega$ if and only if $f=0$ as an element of $\mathcal{W}^{k,2}(\Omega)$.

The proof of the following theorem and more details are given in Section \ref{S:S}.

\begin{theorem}\label{T:main_Sobolev}
Under Assumptions \ref{A:flcapr} and \ref{A:uniqueergodic} and for each $k=0,1,2\ldots$, \begin{enumerate}
 \item\label{sob1} the Sobolev space $\mathcal{W}^{k,2}(\Omega,\mu)$ is a Hilbert space equipped with the norm \eqref{e-sobnorm};
\item\label{sob2} for any $|\alpha|\leq k$, the distributional derivative $D^\alpha$ is a bounded local linear operator 
$$D^\alpha:\mathcal{W}^{k,2}(\Omega,\mu)\to\mathcal{W}^{k-|\alpha|,2}(\Omega,\mu)$$ which can be intertwined with the orbit homeomorphisms $h_{\Lambda}$ as follows: 
for every $f\in\mathcal{W}^{k,2}(\Omega,\mu)$ 
and $\mu$-almost every $\Lambda\in\OO$,
\begin{equation}
\label{e-intertwine}
\left( h_{\Lambda}^*D^\alpha f \right)(\vt)
=
\left( D^\alpha_\Rd h_{\Lambda}^*f \right) (\vt)\qquad\text{for $\lambda^d$-a.e. }\vt\in\Rd
\end{equation}
in the sense of distributional derivatives in both sides of \eqref{e-intertwine};
\item\label{sob3} the space $C_{tlc}^\infty(\Omega)$ is dense in $\mathcal{W}^{k,2}(\Omega,\mu)$ with respect to the norm \eqref{e-sobnorm}. 
\end{enumerate}
\end{theorem} 

\begin{notation}\label{nPt}
By $(P_t)_{t>0}$ we denote the unique extension of $(T_t)_{t>0}$ to $L^2(\Omega,\mu)$. 
\end{notation}

Recall that a strongly continuous semigroup $(P_t)_{t>0}$ of bounded linear operators on $L^2(\Omega,\mu)$ is called a \emph{Markov semigroup} if $0\leq f\leq 1$ $\mu$-a.e. implies $0\leq P_tf \leq 1$ $\mu$-a.e. It is called \emph{conservative} if for any $t>0$ we have $P_t\mathbbm{1}=\mathbbm{1}$ $\mu$-a.e. on $\Omega$. 

The proof of the following theorem and more details are given in Section \ref{S:L2}.

\begin{theorem}\label{T:main_L2sg}
Under Assumptions \ref{A:flcapr} and \ref{A:uniqueergodic} the following results hold. 
\begin{enumerate}[leftmargin=2em]
\item The semigroup $(P_t)_{t>0}$ is a conservative Markov semigroup of self-adjoint operators on $L^2(\Omega,\mu)$. 
Its generator $\mathcal{L}$ is a self-adjoint extension of the Laplacian $\frac12\Delta$ on smooth functions and its domain is $\mathcal{W}^{2,2}(\Omega,\mu)$. 
\item \label{T:main_Koopman}
The unitary Koopman operators $\Koo_{\vt\,}$, defined on $L^2(\Omega,\mu)$ by \eqref{E:KoopmanT-def}, commute with the semigroup $(P_t)_{t>0}$, i.e.
\begin{equation}\label{E:KoopmanP}
\Koo_{\vt\,} P_{t\vphantom{\vt}} = P_{t\vphantom{\vt}} \Koo_{\vt\,},\quad \vt\in\mathbb{R}^d, t>0,
\end{equation}
and hence commute with the Laplacian $\mathcal{L}$. 
\item The associated Dirichlet form of $\mathcal{L}$ is 
\[
\mathcal{E}(f,g)=\int_{\Omega}\left\langle \nabla f,\nabla g\right\rangle d\mu,\quad f,g\in \mathcal{W}^{1,2}(\Omega,\mu).
\]
The Dirichlet form $\mathcal{E}$ is regular, strongly local, irreducible, recurrent, and has pointwise Kusuoka-Hino index $d$ $\mu$-a.e.
\item   The semigroup $(P_t)_{t>0}$ does not admit a heat kernel with respect to $\mu$. Moreover, it does not improve integrability: There exist no $2<q\leq +\infty$ and $t>0$ such that $P_t$ is a bounded operator from $L^2(\Omega,\mu)$ into $L^q(\Omega,\mu)$.  
\end{enumerate}
\end{theorem}

%

\begin{remark}\label{R:HKs}\mbox{}
\begin{enumerate}
\item[(i)]
As already mentioned Remark~\ref{R:pushforward}, there is no Radon-Nikodym derivative of $\mu$
with respect to the non-$\sigma$-finite measure $\lambda^d_\Omega$, which explains why the existence of the heat kernel $p_\Omega(t,\Lambda_1,\Lambda_2)$ does not provide a kernel with respect to~$\mu$.
\item[(ii)]   Recall that a Markov semigroup $(P_t)_{t>0}$ of self-adjoint operators on $L^2(\Omega,\mu)$ 
always induces positivity preserving and contractive semigroups on the spaces $L^p(\Omega,\mu)$, $1\leq p\leq +\infty$, strongly continuous for $1\leq p<+\infty$, see for instance \cite[formula (1.1)]{CKS87} or \cite[Theorem 1.4.1]{Davies89}. A Markov semigroup $(P_t)_{t>0}$ on $\Omega$ is called \emph{hypercontractive} if there exists some $t>0$ such that $P_t$ is bounded from $L^2(\Omega,\mu)$ into $L^4(\Omega,\mu)$, see \cite[Section 2.1]{Davies89} and also \cite{Gross75, Nelson66}. It is called \emph{ultracontractive} if for any $t>0$ the operator $P_t$ is bounded from $L^2(\Omega,\mu)$ into $L^\infty(\Omega,\mu)$, see \cite[Section 2.1]{Davies89} and also \cite{CKS87, Coulhon96, DaviesSimon84}. Obviously ultracontractivity implies hypercontractivity. According to Theorem \ref{T:main_L2sg} the semigroup $(P_t)_{t>0}$ is not hypercontractive and in particular, not ultracontractive. The absence of ultracontractivity also follows from the non-existence of a heat kernel, because an extension to a Markov semigroup of self-adjoint operators of a Feller semigroup which is ultracontractive is known to admit (a preliminary form of) heat kernel, see \cite[Section 3]{CKS87}. For further details see Lemma \ref{L:nothyper} and Corollary \ref{C:noNash}.  
\end{enumerate}
\end{remark}

\begin{remark}
The semigroup $(P_t)_{t>0}$ satisfies the $L^2$-form of the mixing property~\eqref{E:convtoeq}: For any $f\in L^2(\Omega,\mu)$ we have $\lim_{t\to \infty} P_tf=\int_{\Omega} fd\mu$ in $L^2(\Omega,\mu)$. Using the contractivity and conservativity of $(P_t)_{t>0}$ together with $\mu(X)=1$ this is straightforward from \ref{E:convtoeq}. It can also be deduced from \cite[Proposition 3.1.13]{BGL14} together with the irreducibility and recurrence of $(\mathcal{E},\mathcal{D}(\mathcal{E}))$, see Corollary \ref{C:Fukushima} below.
\end{remark}

\begin{remark}\label{R:K}
Theorem \ref{T:main_L2sg}(\ref{T:main_Koopman}) implies that the spectral theory of the semigroup $(P_t)_{t>0}$ and the Laplacian $\mathcal{L}$ follow from the spectral theory of the dynamical system $(\Omega,\mathbb{R}^d)$ coded in the operators $\Koo_{\vt\,}$. It is known that there exists an orthonormal basis of eigenfunctions for the operators $\Koo_{\vt\,}$ (and therefore for $\mathcal{L}$) if and only if the ergodic dynamical system $(\Omega,\mathbb{R}^d,\mu)$ has pure point diffraction, see for instance \cite{BaakeLenz17, BaakeLenzvanEnter15}. Recall that two self-adjoint operators commute if and only if their spectral projections commute for every Borel set. Consequently, all spectral operators, as for instance the Schr\"odinger Hamiltonian, also commute with the operators $\Koo_{\vt\,}$. The standard spectral theory (see e.g.~\cite{RS80,RudinFA}) implies the following corollary.

\end{remark}

\begin{corollary}\label{c:S}Under Assumptions \ref{A:flcapr} and \ref{A:uniqueergodic}:
 \begin{enumerate}
 \item zero is a simple eigenvalue (with the constant eigenfunction) of the Laplacian $\mathcal L$ and 
one is a simple eigenvalue (with the constant eigenfunction) of the heat semigroup $P_{t\vphantom{\vt}}$;
 \item if the spectrum of the family of the Koopman operators $\Koo_{\vt\,}$ consists of eigenvalues of finite multiplicity, then the spectra of $P_{t\vphantom{\vt}}$ and $\mathcal L$ are pure point; 
 \item if the family of the Koopman operators $\Koo_{\vt\,}$ has no nonzero eigenvalues (see 
 Appendix~\ref{SS:DsIs} for the definition and more details), then 
 $\mathcal L$ does not have nonzero eigenvalues of finite multiplicity. 
 \end{enumerate}
\end{corollary}

%
%
%

Finally, in the one dimensional case, 
we obtain a new Helmholtz-Hodge decomposition, which is 
stated in more detail and 
proved in Section~\ref{S:Hodge}.
\begin{theorem}\label{T:main_Hodge}
Under Assumptions \ref{A:flcapr} and \ref{A:uniqueergodic}, and if $d=1$, the space $L^2(\Omega,\mu)$ of square integrable vector fields admits the Helmholtz-Hodge decomposition
\begin{equation}\label{E:Hodgedecomp}
L^2(\Omega,\mu)=\im\nabla\oplus(\im\nabla)^\bot\simeq\im\nabla\oplus\mbbR,
\end{equation}that is, $(\im \nabla)^\bot$ is one-dimensional.
More precisely,
\[
(\im \nabla)^\bot=\star \{\text{constant functions on }\Omega\},
\]
where $\star$ is the Hodge star operator defined in Appendix~\ref{A:H}.
\end{theorem}

\begin{remark}\label{R:contrast}\mbox{}\nopagebreak
\begin{enumerate}
\item[(i)] On a one-dimensional compact Riemannian manifold $M$ the identity (\ref{E:Hodgedecomp}), up to the duality between $L^2$-vector fields and $L^2$-differential $1$-forms, is referred to as Hodge decomposition and $(\im \nabla)^\bot$ corresponds to the space of harmonic $1$-forms. By Hodge's theorem, there exists an isomorphism between the space of harmonic $1$-forms and the first deRham cohomology of $M$, connecting elliptic and smooth theory, see for instance \cite{Jost02}. In \cite{HT15}, a decomposition of type~\eqref{E:Hodgedecomp} was obtained for certain (non-smooth) compact topologically one-dimensional spaces endowed with a local regular Dirichlet form. It was also proved that locally harmonic forms are dense in the space of harmonic forms.
A theory of differential forms associated with Dirichlet forms is a generalization of elliptic theory. In general it cannot be expected to produce topological information but is rather linked to $L^2$-cohomology, see e.g. \cite{Carron02, Dodziuk77, Dodziuk82}.
\item[(ii)] For the Fibonacci tiling in dimension one it is known that the first \v{C}ech cohomology with integer coefficients of its pattern space has rank two, see \cite[p.\ 39]{sadun:book} or \cite[Example 1]{Sadun2015}. Consequently, its first \v{C}ech cohomology with real coefficients is two-dimensional. However, it is also isomorphic to the first pattern equivariant cohomology, see \cite{KP00} or \cite[Theorem 5.1]{sadun:book}, which may be thought of as a replacement of the first deRham cohomology. More generally, this holds for any proper, primitive and aperiodic one-dimensional substitution with substitution matrix of rank greater than one \cite[\S 6.1]{sadun:book}.  
By Theorem \ref{T:main_Hodge}, $(\im \nabla)^\bot$ is one-dimensional so that a Hodge theorem as for compact Riemannian manifolds cannot hold.
\end{enumerate}
\end{remark}

\section{Proof of Theorem \ref{T:main_Feller}: Feller semigroups and generators}\label{S:FellSmg}

\subsection{Existence and Feller property}
 
In this subsection we show that $(T_t)_{t>0}$ is a Feller semigropup on $(\Omega,\varrho)$. It suffices to prove the strong continuity and the Feller property, that is $T_t f$ is continuous for any continuous function $f$. We start by showing that
\begin{equation}\label{E:contatzero}
\lim_{t\to 0}\left\|T_tf-f\right\|_{\infty}=0,\quad\ f\in C(\Omega).
\end{equation}
First, note that for any $\vs\,\in\mathbb{R}^d$ and any $\Lambda\in\Omega$, 
\begin{equation}\label{E:Euclaction}
\varrho(\Lambda,\varphi_{\vs}\,(\Lambda))\leq 2\varrho_{\orb}(\Lambda,\varphi_{\vs}\,(\Lambda))\leq 2|\vs|.
\end{equation}
Given a function $f\in C(\Omega)$ and $\varepsilon>0$, by the uniform continuity of $f$ on $(\Omega,\varrho)$ we can find $s_{\varepsilon}>0$ such that $\sup_{\Lambda\in\Omega}|f(\varphi_{\vec{s}}\,(\Lambda))-f(\Lambda)|<\frac{\varepsilon}{2}$, provided $|\vec{s}\,|<s_{\varepsilon}$. For $t>0$ sufficiently small we have
\begin{equation}\label{E:smalltails}
\int_{|\vs\,|\geq s_{\varepsilon}}p_{\mathbb{R}^d}(t,\vs\,)d\vs<\frac{\varepsilon}{4\left\|f\right\|_{\infty}} 
\end{equation}
and consequently
\begin{multline*}
|T_tf(\Lambda)-f(\Lambda)|\leq \int_{\mathbb{R}^d}p_{\mathbb{R}^d}(t,\vec{s}\,)|f(\varphi_{\vec{s}}\,(\Lambda))-f(\Lambda)|d\vec{s}\notag\\
\leq \frac{\varepsilon}{2}\int_{|\vs\,|<s_{\varepsilon}}p_{\mathbb{R}^d}(t,\vec{s}\,)d\vec{s}+2\left\|f\right\|_{\infty}\int_{|\vs\,|\geq s_{\varepsilon}}p_{\mathbb{R}^d}(t,\vs\,)d\vs\notag\leq \varepsilon\notag
\end{multline*}
uniformly in $\Lambda\in \Omega$, proving (\ref{E:contatzero}). The next lemma will be used in order to show that continuous functions are continuous under $T_t$.
\begin{lemma}\label{L:prep}
Let $\delta>0$. For all $\Lambda\in\Omega$ and all $\vs\in\mathbb{R}^d$ with $|\vs\,|<\frac{1}{2\delta}$ we have 
\[
\sup_{\overline{\Lambda}\in O_{\Lambda,\delta}} \varrho(\varphi_{\vs}\,(\overline{\Lambda}),\varphi_{\vs}\,(\Lambda))<4\delta.
\]
\end{lemma}
\begin{proof}
We first show that
\begin{equation}\label{E:claim}
\varrho(\varphi_{\vs}\,(\Lambda'),\varphi_{\vs}\,(\Lambda))<2\delta
\end{equation} 
for any $\Lambda'\in \mathcal{C}_{\Lambda,\delta}$. Since $B_{\frac{1}{\delta}}(\vec{0}\,)\cap \Lambda'=B_{\frac{1}{\delta}}(\vec{0}\,)\cap \Lambda$ implies $B_{\frac{1}{2\delta}}(\vs\,)\cap \Lambda'=B_{\frac{1}{2\delta}}(\vs\,)\cap \Lambda$, translating by $-\vs$ we obtain $\varphi_{\vs}\,(\Lambda')\in\mathcal{C}_{\varphi_{\vs}\,(\Lambda),2\delta}$ and in particular (\ref{E:claim}). 

Consider now the translation $\overline{\Lambda}=\varphi_{\vt}\,(\Lambda')$ of a given point $\Lambda'$ with $|\vt\,|<\delta$. By~\eqref{E:Euclaction} we have that 
\[
\varrho(\varphi_{\vs}\,(\overline{\Lambda}),\varphi_{\vs}\,(\Lambda'))=\varrho(\varphi_{\vt}\,(\varphi_{\vs}\,(\Lambda')), \varphi_{\vs}\,(\Lambda'))\leq 2\delta.
\]
According to~\eqref{e-phi}, an arbitrary point $\overline{\Lambda}\in O_{\Lambda,\delta}$ can be uniquely written as such a translation of some $\Lambda'\in\mathcal{C}_{\Lambda,\delta}$ and therefore
\[
\varrho(\varphi_{\vs}\,(\overline{\Lambda}),\varphi_{\vs}\,(\Lambda))\leq \varrho(\varphi_{\vs}\,(\overline{\Lambda}),\varphi_{\vs}\,(\Lambda'))+\varrho(\varphi_{\vs}\,(\Lambda'),\varphi_{\vs}\,(\Lambda))\leq 4\delta.
\]
\end{proof}
Finally, in order to prove that each $T_t$ maps $C(\Omega)$ into itself, we consider $f\in C(\Omega)$, $t>0$ and $\Lambda\in \Omega$. Given $\varepsilon>0$, let $s_{\varepsilon}>0$ be large enough so that~\eqref{E:smalltails} is satisfied. Applying Lemma \ref{L:prep} and the uniform continuity of $f$ on $\Omega$, we can find $0<\delta<\frac{1}{2s_\varepsilon}$ such that for any $\overline{\Lambda}\in O_{\Lambda,\delta}$ and any $\vs\in\mathbb{R}^d$ with $|\vs\,|<s_{\varepsilon}$ we have $|f(\varphi_{\vs}\,(\overline{\Lambda}))-f(\varphi_{\vs}\,(\Lambda))|<\frac{\varepsilon}{2}$. Consequently,
\begin{align}\label{E:transest}
|T_tf(\overline{\Lambda})-T_tf(\Lambda)|&\leq \int_{\mathbb{R}^d}p_{\mathbb{R}^d}(t,\vs\,)|f(\varphi_{\vs}\,(\overline{\Lambda}))-f(\varphi_{\vs}\,(\Lambda))|d\vs\nonumber\\
&\leq 2\left\|f\right\|_{\infty}\int_{|\vs\,|\geq s_\varepsilon}p_{\mathbb{R}^d}(t,\vs\,)\,d\vs+\frac{\varepsilon}{2}\int_{|\vs\,|<s_\varepsilon} p_{\mathbb{R}^d}(t,\vs)d\vs<\varepsilon
\end{align}
and hence $T_tf$ is continuous at each point $\Lambda\in \Omega$, as desired.


Mass conservation and commutativity with Koopman operators are direct consequences of Definition~\ref{D:sg}.

\subsection{Smoothness preservation and infinitesimal generator}

In this paragraph we analyze the infinitesimal generator of the Feller semigroup $(T_t)_{t>0}$ on $C(\Omega)$ which we denote by $(\mathcal{L}_{C(\Omega)},\mathcal{D}(\mathcal{L}_{C(\Omega)}))$, where
\begin{equation}\label{E:Fellerdomain}
\mathcal{D}(\mathcal{L}_{C(\Omega)}):=\left\lbrace f\in C(\Omega):~\lim_{t\to 0} \frac{1}{t}(T_tf-f)\text{ exists strongly in }C(\Omega)\right\rbrace
\end{equation}
and 
\begin{equation}\label{E:Fellergen}
\mathcal{L}_{C(\Omega)}f:=\lim_{t\to 0} \frac{1}{t}(T_tf-f),\ \ f\in\mathcal{D}(\mathcal{L}_{C(\Omega)}).
\end{equation}

As we will see in Lemma~\ref{T:TisSg}, on a suitable subspace of its domain, $\mathcal{L}_{C(\Omega)}$ coincides with the Laplacian $\Delta$ defined in~\eqref{E:Laplace}.
\begin{proposition}\label{P:Cinf}
For any $1\leq k\leq\infty$ and $t>0$ it holds that 
\[
T_tf\in C^k(\Omega)\qquad\text{ for all }~f\in C^k(\Omega).
\]
\end{proposition}
\begin{proof}
Let $f\in C^k(\Omega)$ and $\vec{v}\in\mathbb{R}^d$. By dominated convergence,
\[
\frac{\partial}{\partial\vec{v}}\,T_tf(\Lambda)=\int_{\mathbb{R}^d}p_{\mathbb{R}^d}(t,\vs\,)\frac{\partial f}{\partial\vec{v}\,}(\varphi_{\vs\,}(\Lambda))\,d\vs \qquad \forall~\Lambda\in \Omega,
\] 
hence $\frac{\partial}{\partial\vec{v}}\,T_tf\in C^{k-1}(\Omega)$. Iterating, we obtain the desired fact.
\end{proof}
 Recall that a subspace of the domain of a closed operator is called a \emph{core} if the operator equals the closure of its restriction to that subspace.
\begin{lemma}\label{T:TisSg}
We have that $C^2(\Omega)\subseteq\mathcal{D}(\mathcal{L}_{C(\Omega)})$ and 
\begin{equation}\label{E:LequalsDelta}
\mathcal{L}_{C(\Omega)}f=\frac{1}{2}\Delta f,\qquad f\in C^2(\Omega).
\end{equation}
Moreover, the space $C_{tlc}^\infty(\Omega)$ is a core for $\mathcal{L}_{C(\Omega)}$ and $\mathcal{L}_{C(\Omega)}$ is a local operator.
\end{lemma}

\begin{proof}
By Lemma~\ref{L:tlcdense}, $C^2_{tlc}(\Omega)$ is dense in $C^2(\Omega)$, hence it suffices to prove~\eqref{E:LequalsDelta} for $f\in C_{tlc}^2(\Omega)$, in which case we have $\Delta f\in C(\Omega)$. Applying Lemma \ref{L:basicprops}, Corollary \ref{C:connecttoRd} and integration by parts on $\mathbb{R}^d$ yields
\begin{equation}\label{E:Help1}
T_t\Big(\frac12\Delta f\Big)(\Lambda)=\int_{\mathbb{R}^d}p_{\mathbb{R}^d}(t,\vs\,)\Big(\frac12\Delta f\Big)(\varphi_{\vs}\,(\Lambda))d\vs=\int_{\mathbb{R}^d} \frac12\Delta_{\mathbb{R}^d} p_{\mathbb{R}^d}(t,\vs\,)h_\Lambda^\ast f(\vs)d\vs
\end{equation}
for any $t>0$ and all $\Lambda\in \Omega$. By the heat equation on $\mathbb{R}^d$ and dominated convergence the latter equals
\begin{equation}\label{E:Help2}
\int_{\mathbb{R}^d} \frac{\partial}{\partial t} p_{\mathbb{R}^d}(t,\vs\,)f(\varphi_{\vs}\,(\Lambda))d\vs=\frac{d}{dt} T_tf(\Lambda).
\end{equation}
Integrating the left hand side of~\eqref{E:Help1} and the right hand side of~\eqref{E:Help2} we obtain 
\[
\frac{1}{t}(T_tf-f)=\frac{1}{t}\int_0^t T_s(\frac12\Delta f)ds,
\]
as an equality in $C(\Omega)$. Letting $t\to 0$ we get $f\in \mathcal{D}(\mathcal{L}_{C(\Omega)})$ with $\mathcal{L}_{C(\Omega)}f=\frac12\Delta f$. 

To prove that $C_{tlc}^\infty(\Omega)$ is a core, it suffices by Lemma \ref{L:tlcdense} to show that $C^\infty(\Omega)$ is a core. 
This follows from Proposition~\ref{P:Cinf} and the standard Hille-Yosida theory (see for instance \cite[Chapter 1, Proposition 3.3]{EthierKurtz86}). 

Finally, by virtue of Corollary~\ref{C:local}, for any open $O\subset\Omega$ and $f\in \mathcal{D}(\mathcal{L}_{C(\Omega)})\cap C^\infty(\Omega)$ with $\supp f\subset O$ we also have $\supp \mathcal{L}_{C(\Omega)} f\subset O$. Due to the density of $C^\infty(\Omega)$ in $\mathcal{D}(\mathcal{L}_{C(\Omega)})$, the same is true for general $f\in \mathcal{D}(\mathcal{L}_{C(\Omega)})$.
\end{proof}

\begin{remark}\label{R:Fellerresolvent} 
The Hille-Yosida theory (see e.g.~\cite[Chapter 1]{EthierKurtz86} for a concise introduction), connects resolvents and semigroups. Let $R_1:C(\Omega)\to C(\Omega)$ denote the $1$-resolvent operator associated with $(T_t)_{t>0}$, that is, $R_1f(x)=\int_0^\infty e^{-t}T_tf(x)dt$, $f\in C(\Omega)$. It follows from the latter proof 
that $R_1f\in C^\infty(\Omega)$ for any $f\in C^\infty(\Omega)$. 
\end{remark}

As a consequence of the preceding and the next lemma, Corollary \ref{C:bumps} in the Appendix implies that the domain $\mathcal{D}(\mathcal{L}_{C(\Omega)})$ contains smooth partitions of unity.

\begin{lemma}\label{L:partofunity}
For any finite open cover of $\Omega$ by open sets $O_1,\ldots,O_n$ of type (\ref{E:transcylsets}) there exists a subordinate partition of unity $\chi_1,...,\chi_n$, where $\chi_i\in C_{tlc}^\infty(\Omega)$.
\end{lemma}

\begin{remark}
For the gradient operator $\nabla$ defined in~\eqref{E:gradient}, we have that $\nabla T_t f=T_t\nabla f$ for any $t>0$ and $f\in C_{tlc}^1(\Omega)$.
\end{remark}

\subsection{Absence of the strong Feller property}\label{subsec-no-strong-Feller}

If $f$ is the indicator function of an orbit, then $T_tf=f$, which gives the shortest proof of the absence of the strong Feller property stated in Theorem \ref{T:main_Feller}. The local product structure of $\Omega$ discussed in section~\ref{subsec:LocalStru}, and especially \eqref{e-prod}, is important in order to understand how the semigroup acts on other functions. 
On $\OLe=\phi^{-1}(\mathcal{C}\times B)$ as in~\eqref{E:transcylsets} we can consider a function $f(\bar\Lambda):=f_0(\Lambda')F_0(\vt\,)$, 
$ (\Lambda',\vt\,)=\phi(\overline{\Lambda})$, as in \eqref{e-prod}, defined as follows. 
We assume that $F_0=\mathbbm 1_{B'}$ is an indicator of an open nonempty ball whose closure is contained in the ball $B$, and $f_0=\mathbbm{1}_E$, 
where $E$ is subset of $\mathcal{C}$ such that $\partial_\mathcal{C}E\neq \varnothing$. 
Here, $\partial_\mathcal{C}E$ means the boundary of $E$ in the induced (i.e. intrinsic) topology of $\mathcal{C}$. Then it is easy to see that for any $\delta>0$ there is a small enough $t>0$ such that $T_t f(\Lambda_1)>1/3$ if $\Lambda_1\in E\times B'$ and $T_t f( \Lambda_2 )<1/9$ if $\varrho_{\orb}(\Lambda_2 , E\times B' )>\delta $. For small enough $\delta$ this implies that $T_t f( \Lambda_2 )<1/9$ if $ \Lambda_2 \in E^c\times B'$, which implies that $T_t f$ is not continuous. In fact, using more delicate arguments one can show that $T_tf$ is discontinuous is $f$ is supported in \OLe\ and has the structure $f(\bar\Lambda):=f_0(\Lambda')F_0(\vt\,)$ 
where $F_0>0$ on a set of positive $\mathbb R^d$-Lebesgue measure, and $f_0$ is discontinuous in the induced (i.e. intrinsic) topology of $\partial_\mathcal{C}E$. 

The fact that $(T_t)_{t>0}$ admits a heat kernel with respect to the measure $\lambda_\Omega$ is immediate from~\eqref{E:defsg}.

\begin{notation}\label{D:Ckorb}
We write $f\in C_{\orb}(O)$ if $f$ is $\varrho_{\orb}$-continuous in $O\subset \Omega$. If $O\subset \Omega$ is $\varrho_{\orb}$-open, then we write $f\in C_{\orb}^k(O)$ if $f$ is Borel measurable and $k$-times differentiable, in the sense of Appendix \ref{A:S}, and $D^\alpha f\in C_{\orb}(O)$ for all $|\alpha|\leq k$.
\end{notation}

Note that $C^k(O)\subsetneqq C^k_{\orb}(O)$ if $O$ is open and non-empty.

\begin{remark}
The corresponding Feller properties of $(T_t)_{t>0}$ on the metric space $(\OO,\varrho_{orb})$ are satisfied. In view of Definition~\ref{D:sg}, the process $X^\Lambda$ is also a diffusion along the orbit $\orb(\Lambda)$, and by the isomorphism in Lemma \ref{L:Horizontal}, it can naturally be identified with a Brownian motion on $\mathbb{R}^d$ by formula (\ref{E:process}). It makes sense to call the family of processes $X^\Lambda$ as in (\ref{E:process}) ``orbit-wise Brownian motion'' on $\OO$. As already noted in Remark \ref{R:Feller} (iii), Definition~\ref{D:sg} and formula (\ref{E:process}) show that the family $(T_t)_{t\geq 0}$ is a Feller semigroup with respect to $\varrho_{\orb}$. It also follows that $(T_t)_{t>0}$ is a strongly continuous contraction semigroup on $bC_{\orb}(\OO)$, the domain of its generator contains $bC_{\orb}^2(\OO)$, and on this space it agrees with $\frac{1}{2}\Delta$. Finally, the corresponding result on $\mathbb{R}^d$ imply that $(T_t)_{t>0}$ is strong Feller with respect to $\varrho_{\orb}$, and $X^\Lambda$ is a strong Feller process with respect to $\varrho_{\orb}$ started at $\Lambda$. Note also that if the lattice $\Lambda_0$ would be periodic so that $ \OO$ would be a flat torus, we would just recover the usual diffusion process on this flat torus. 
\end{remark}

\section{Proof of Theorem \ref{T:main_Sobolev}: Sobolev spaces}\label{S:S}
\subsection{Weak derivatives and Sobolev spaces}\label{SS:Sobo}
In this section we introduce weak derivatives of functions on $\Omega$ and characterize the Sobolev spaces $\mathcal{W}^{k,2}(\Omega,\mu)$ defined in Theorem~\ref{T:main_Sobolev}. We refer to Appendix~\ref{A:S} for further specific notation, definitions and technical results. 

\begin{definition}\label{D:distributionalsense}
For any $f\in L^1(\Omega,\mu)$ and any multiindex $\alpha$, the  distributional (generalized) derivative $D^\alpha f$ is defined as the element of the topological dual $(C^\infty(\Omega))^\ast$ of $C^\infty(\Omega)$ given by 
\begin{equation}\label{E:weak}
D^\alpha f(\varphi):=(-1)^{|\alpha|}\int_{\Omega} f D^\alpha \varphi \:d\mu,\quad \varphi\in C^\infty(\Omega).
\end{equation}
\end{definition}

\begin{definition}\label{D:EvansDerivative}
If $f\in L^2(\Omega,\mu)$, $D^\alpha f\in L^2(\Omega,\mu)$, then   $D^\alpha$ is called a weak $L^2(\Omega,\mu)$ derivative and 
\begin{equation}
\langle D^\alpha f,\varphi\rangle_{L^2(\Omega,\mu)}=(-1)^{|\alpha|}\int_{\Omega} f D^\alpha \varphi \:d\mu,\quad \forall~\varphi\in C^\infty(\Omega).
\end{equation}
\end{definition}
\begin{remark}\label{R:EvansDerivative}
 Definitions \ref{D:distributionalsense} and \ref{D:EvansDerivative} 
   can be localized  
   for any open $O\subset\OO$.  
\end{remark}

The following integration by parts formula for functions in the space $C^k_{tlc}(\Omega)$, c.f. Definition~\ref{D:tlc}, will be heavily used in many of the subsequent proofs.
\begin{lemma}\label{L:ibP}
For any $f,g\in C^k_{tlc}(\Omega)$ and $|\alpha|\leq k$,
\[
\left\langle D^\alpha f,g\right\rangle_{L^2(\Omega,\mu)}=(-1)^{|\alpha|}\left\langle f, D^\alpha g\right\rangle_{L^2(\Omega,\mu)}.
\]
\end{lemma}

\begin{proof}
By Lemma~\ref{L:partofunity} there exists a smooth partition of unity $\{\chi_i\}_{i=1}^n$, hence we may assume that $g$ is supported in an open set $O=\phi^{-1}(\mathcal{C}\times B)$ as in~\eqref{E:transcylsets}. In view of~\eqref{E:fO} and~\eqref{E:DalphafO}, the result follows from its Euclidean counterpart since
\begin{multline*}
-\int_{\mathcal{C}}\int_B D^\alpha_{\mathbb{R}^d} (f\cphi) (\Lambda,\vt\:) g\cphi (\Lambda,\vt\,)\,d\vt\:\nu_{\mathcal{C}}(d\Lambda)\\
=(-1)^{|\alpha|}\int_{\mathcal{C}}\int_B f\cphi (\Lambda,\vt\:) D^\alpha_{\mathbb{R}^d}(g\cphi) (\Lambda,\vt\:) d\vt\:\nu_{\mathcal{C}}(d\Lambda),
\end{multline*}
where the differential operator $D_{\mathbb{R}^d}^\alpha$ is understood in the usual Euclidean sense.
\end{proof}

\begin{remark}
This result and Lemma~\ref{L:tlcdense} imply that for $f\in C^k(\Omega)$, the weak derivative $D^\alpha f$ from~\eqref{E:weak} coincides with the classical derivative $D^\alpha f$ in~\eqref{E:Dalpha}. 
\end{remark}

The main properties of weak derivatives in the classical case, see e.g.~\cite[Theorem 1, Section 5.2.3] {Evans98}, 
can be obtained in the same fashion for $D^\alpha f$ with $f\in\mathcal{W}^{k,2}(\Omega,\mu)$. In particular, the Leibniz rule applies and for any $\chi\in C^\infty(\Omega)$ and $f\in\mathcal{W}^{k,2}(\Omega,\mu)$, also $\chi f\in \mathcal{W}^{k,2}(\Omega,\mu)$. 

With these properties on hand, the next lemma characterizes the function space $\mathcal{W}^{k,2}(\Omega,\mu)$. 
The proof is straightforward after observing the crucial fact that $\Omega$ admits a finite open covering and a subordinated smooth partition of unity, see Lemma~\ref{L:partofunity}. For any $O\subset\Omega$ open, the space $\mathcal{W}^{k,2}(O,\mu)$ is defined analogously as for $\Omega$. Abusing notation, $\mu$ denotes in this case its restriction to $O$.

\begin{lemma}\label{L:Sobolev_Charact}
For any $k\geq 0$ we have that 
 $\mathcal{W}^{k,2}(O,\mu)$ is a Hilbert space. Moreover, $f\in \mathcal{W}^{k,2}(O,\mu)$ if and only if 
\begin{equation}\label{E:Sobolev_Charact}
f=\sum_{i=1}^n\chi_if\qquad\text{and}\qquad\chi_if\in\mathcal{W}^{k,2}(O_i\cap O,\mu)~\forall\,i\geq 0,
\end{equation} 
where $\{\chi_i\}_{i=1}^n\subseteq C^\infty_{tlc}(\Omega)$ is a smooth partition of unity subordinated to a cover $\{O_i\}_{i=1}^n\subseteq\Omega$ of translated cylinder sets~\eqref{E:transcylsets}.
\end{lemma}

\begin{proof}
 The proof of this lemma is the same as in the classical case, see e.g.~\cite[Theorem 2, Section 5.2.3, p.\ 249]{Evans98}. 
\end{proof}
This completes the proof of Theorem~\ref{T:main_Sobolev}(\ref{sob1}).
 
\subsection{Localization in Sobolev spaces}\label{subsec:sob-loc}

In the special case when $O=\OLe\subset\Omega$ is a transversal $\varepsilon$-cylinder set, the product structure of this set and the definitions yield the following useful result for any $k\geqslant1$. Its proof is an application of the standard measure theory, especially the analysis on product spaces and the basic properties of distributional derivatives and Sobolev spaces, seee.g.~\cite[Chapter 8]{RudinRCA} and \cite[Chapter 5]{Evans98}.

\begin{lemma}\label{L:W2kOLe} 
If $f\in \mathcal{W}^{k,2} (\OLe,\mu)$, then $h_{\Lambda'}^\ast f\in W^{k,2} (B)$ for $\nu_{\mathcal{C}}$-a.e.\ $\Lambda'\in\mathcal{C}$, and 
\begin{equation}\label{e-key-sob-OLe}
\int_{\mathcal{C}}\Big\| h^\ast_{\Lambda'}f(\vt\,)\Big\|_{W^{k,2}_\Rd (B)}^2 \nu_{\mathcal{C}}(d\Lambda')
=\left\|f\right\|_{\mathcal{W}^{k,2}(\OLe,\mu)}^2.
\end{equation} 
Here, $h_{\Lambda'}^\ast$ is the orbit homeomorphism defined in~\eqref{E:algiso}, and $\OLe=\phi^{-1}(\mathcal{C}\times B)$ is of type~\eqref{E:transcylsets}. 

Moreover, if $f\in L^{2} (\OLe,\mu)$, then $h_{\Lambda'}^\ast f\in W^{k,2} (B)$ for $\nu_{\mathcal{C}}$-a.e.\ $\Lambda'\in\mathcal{C}$, and if the left hand side in \eqref{e-key-sob-OLe} is finite, then $f\in \mathcal{W}^{k,2} (\OLe,\mu)$. 

If $f,g\in \mathcal{W}^{k,2} (\OLe,\mu)$ then 
\begin{equation}\label{e-key-sob-OLe-inn}
\big\langle f,g\big\rangle_{\mathcal{W}^{k,2}(\OLe,\mu)} 
=
\int_{\mathcal{C}}
\Big\langle h^\ast_{\Lambda'}f(\vt\,),h^\ast_{\Lambda'}g(\vt\,)\Big\rangle_{\mathcal{W}^{k,2}(\OLe,\mu)} \nu_{\mathcal{C}}(d\Lambda')
.
\end{equation} 
\end{lemma}


On $\mathbb{R}^d$ we consider the spaces
\[
W^{k,2}_{loc}(\mathbb{R}^d):=\left\lbrace f\in L^2_{loc}(\mathbb{R}^d): D^\alpha f\in L^2_{loc}(\mathbb{R}^d)\quad\text{ for all $|\alpha|\leq k$}\right\rbrace.
\]
The next lemma follows from Lemma~\ref{L:Sobolev_Charact} and is important for understanding the local structure of Sobolev functions on \OO\. It completes the proof of Theorem~\ref{T:main_Sobolev}(\ref{sob2}). 
\begin{lemma}\label{L:W22loc}
If a measurable function $f$ on $\Omega$ is in $\mathcal{W}^{k,2} (\Omega,\mu)$ then $h_{\Lambda}^\ast f\in W^{k,2}_{loc}(\mathbb{R}^d)$ for $\mu$-a.e. $\Lambda\in \Omega$. Here $h_{\Lambda}^\ast$ is the orbit homeomorphism defined in \eqref{E:algiso}. 
If in addition $\supp f\subset O=\OLe=\phi^{-1}(\mathcal{C}\times B)$ of type (\ref{E:transcylsets}), then 
 \eqref{e-key-sob-OLe} holds. 
\end{lemma}
\begin{proof} 
Suppose $f\in W^{k,2}(\Omega,\mu)$. Taking a smooth partition of unity we may again assume that $\supp f\subset\OLe$. 
With the notation from Subsection~\ref{subsec:LocalStru} and following the results in Subsection~\ref{SS:Sobo}, it is easy to see that for $|\alpha|\leq k$ and $\mu$-a.e. $\bl \in\mathcal{C}\times B$ we have
$$ 
\phi_{\Lambda,\varepsilon}(\bl)
=
\phi_{\Lambda,\varepsilon}(\varphi_{\vt\,}(\Lambda'))
=
(\Lambda',\vt\:),
$$ 
and
\begin{equation}
\label{e-key-sob0}
 D^\alpha f \cphi (\Lambda',\vt\:)= D^\alpha_\Rd h^\ast_{\Lambda'}f(\vt\,).
\end{equation}
Therefore,
\begin{equation}\label{e-key-sob2}
 \big\|D^\alpha f \big\|_{L^{2}(\Omega,\mu)}^2=
 \int_{\mathcal{C}}\int_{B} \Big| D^\alpha_\Rd
h^\ast_{\Lambda'}f(\vt\,)\Big|^2d\vt\, \nu_{\mathcal{C}}(d\Lambda')
<+\infty
\end{equation}
and by Fubini's theorem there exists a Borel set $\mathcal{N}\subset\mathcal{C}$ of $\nu_{\mathcal{C}}$-measure zero such that for any $\Lambda'\in\mathcal{C}\setminus\mathcal{N}$ we have 
\begin{equation}\label{E:Sobofinite}
\max_{|\alpha|\leq k} \int_{B} \Big| D^\alpha_\Rd
h^\ast_{\Lambda'}f(\vt\,)\Big|^2d\vt\, 
<+\infty.
\end{equation}
Since each $\phi$ is continuous from $O$ onto $\mathcal{C}\times B_{\varepsilon}(\vec{0}\,)$, the set
\[\bigcup_{\vt\in B_{\varepsilon}(\vec{0}\,)} \varphi_{\vt\,}(\mathcal{N})=\phi^{-1}(\mathcal{N}_i\times B_{\varepsilon}(\vec{0}\,))\]
is a measurable subset of $\Omega$ and so are its translates $\bigcup\limits_{\vt\in B_{\varepsilon}(\vec{x}\,)} \varphi_{\vt\,}(\mathcal{N})$ by a given vector $\vec{x}\in\mathbb{R}^d$. Since we can write $\mathbb{R}^d$ as the union $\mathbb{R}^d=\bigcup_{\vec{x}\in\frac{\varepsilon}{2}\mathbb{Z}^d} B_{\varepsilon}(\vec{x}\,)$ of countably many balls $B_{\varepsilon}(\vec{x}\,)$ with centers $\vec{x}\in \frac{\varepsilon}{2}\mathbb{Z}^d$, the union $\mathcal{O}$ of all orbits that hit $\mathcal{N}$, defined by
\begin{equation}\label{E:countableunion}
\mathcal{O}:=\bigcup_{\vt\in\mathbb{R}^d}\varphi_{\vt\,}(\mathcal{N})=\bigcup_{\vec{x}\in\frac{\varepsilon}{2}\mathbb{Z}^d} \bigcup_{\vt\in B_{\varepsilon}(\vec{x}\,)} \varphi_{\vt\,}(\mathcal{N}), 
\end{equation}
is a countable union of measurable subsets of $\Omega$ and therefore measurable. By the local product structure of $\mu$ and its $\mathbb{R}^d$-invariance we have $\mu\circ \phi^{-1}(\mathcal{N}\times B_{\varepsilon}(\vec{x}))=0$ for any $\vec{x}\in\mathbb{R}^d$, so that the set $\mathcal{O}$ in (\ref{E:countableunion}) is seen to have $\mu$-measure zero. Suppose now that $\Lambda\in \mathcal{O}^c$ i.e.\ that $\orb(\Lambda)$ does not hit $\mathcal{N}$. If $K\subset\mathbb{R}^d$ is compact, then the set $h_\Lambda(K)\cap O$ can only have finitely many connected components. By (\ref{E:Sobofinite}) together with the fact that within each connected component, $h_{\Lambda}^\ast f$ is compactly supported in a ball $B$ and we obtain $h_{\Lambda}^\ast f\in W^{2,2}_{loc}(\mathbb{R}^d)$. 
\end{proof}

The cases $k=1,2$ will be of special interest later on, and the following lemma will be used in the proof of Lemma \ref{L:Soboanddomains}. 

\begin{lemma}\label{lem-w22}
 In the case $k=2$, we have
 \[
 \mathcal{W}^{2,2}(\Omega,\mu)=\left\lbrace f\in\mathcal{W}^{1,2}(\Omega,\mu): \Delta f \in L^2(\Omega,\mu)\right\rbrace,
 \]
 where $\Delta$ is the operator defined in~\eqref{E:Laplace}, and 
 \[
 \left\|f\right\|_{\mathcal{W}^{2,2}(\Omega,\mu)}':=\left(\left\|f\right\|_{L^2(\Omega,\mu)}^2+ \left\||\nabla|\right\|_{L^2(\Omega,\mu)}^2+ \left\|\Delta f\right\|_{L^2(\Omega,\mu)}^2\right)^{1/2}
 \]
 defines an equivalent norm on $\mathcal{W}^{2,2}(\Omega,\mu)$. 
 \end{lemma}
 
 \begin{proof}
The mixed partial derivatives of type $\frac{\partial^2 f}{\partial \vec{e}_i\partial\vec{e}_j}$ are square integrable by Lemmas \ref{L:Sobolev_Charact} and \ref{L:W22loc} because multiplication with cut-off functions allows to assume that $f$ is supported within a translated cylinder set \OLe. The functions $h_\Lambda^\ast f$ are then compactly supported in a Euclidean ball, and on this ball we can use the integration by parts for distributional derivatives, which can be proved using approximation by smooth functions, to shift around partial derivative operators. The latter commute by Schwarz' theorem, and Cauchy-Schwarz and elementary Fourier analysis yield the desired integrability.
\end{proof}

\subsection{Approximation in Sobolev spaces}\label{subsec:sob-appr}
In this subsection we provide local approximations of functions in $\mathcal{W}^{k,2}(\Omega,\mu)$ by smooth functions in $C^\infty_{tlc}(\Omega)$. At the same time, we describe the spectrum of the Dirichlet Laplacian on the translated $\varepsilon$-cylinder subsets $\OLe$ of $\Omega$ in terms of the Dirichlet Laplacian on Euclidean balls.
 
Let $\{ b_i^{\mathcal{C}}\}_{i=1}^\infty$ be an orthonormal basis of the separable Hilbert space $L^2(\mathcal{C},\nu_{\mathcal{C}})$ which satisfies the following property: for each $i$ there is $\epsilon_i>0$ such that $b_i^{\mathcal{C}}$ is constant on every clopen subset of $ {\mathcal{C}}$ of diameter at most $\epsilon_i>0$. It is easy to see that such a basis exists and that $\lim_{i\to\infty}\epsilon_i=0$. Moreover, one can show that such a basis naturally defines ${\mathcal{C}}$ and ${\mathcal{O}}$ as projective limits, which allows to use many elements of the analysis presented in~\cite{PARdiamonds,BE04,ST17}. For any function $f\colon \OLe\ \to\mathbb{R}$ we define $\PHii f$ by projecting in the transversal direction to the first $i$ terms of the basis $\{ b_i^{\mathcal{C}}\}_{i=1}^\infty$, i.e.
 \begin{equation}
\PHii f(\overline{\Lambda})=f^{O,i}(\Lambda',\vt\,),
\end{equation}
 with $f^{O,i}$ as in Lemma \ref{L:tlcdense}. If $f\colon\Omega\to\mathbb{R}$ is supported in $\OLe$, the function $\Phi_i f$ is extended by zero outside of $\OLe$. It is clear that \PHii\ projects functions localized on \OLe\ on \emph{tlc} functions. This is because the function $\mathbbm{1}_{\mathcal{C}_l^{(i)}}(\Lambda')$ in \eqref{e-fOi} is locally constant. 
 
The following lemma is evident from our construction and, together with the characterization from Lemma~\ref{L:Sobolev_Charact}, it provides a procedure to approximate functions in the Sobolev space $\mathcal{W}^{k,2}(\Omega,\mu)$ by smooth functions in $C^\infty_{tlc}(\Omega)$. 
 
\begin{lemma}\label{L:local_approx}
If $f\in \mathcal{W}^{k,2}(\Omega,\mu)$ has its support in \OLe, then $\PHii f\in\mathcal{W}^{k,2}(\Omega,\mu)$, 
\begin{equation}\label{e-cont-sob}
\|\PHii f\|_{\mathcal{W}^{k,2}(\Omega,\mu)}\leqslant\| f\|_{\mathcal{W}^{k,2}(\Omega,\mu)}
\end{equation} 
and 
\begin{equation}\label{e-cont-sob-c}
\lim\limits_{i\to\infty}
 \|\PHii f-f\|_{\mathcal{W}^{k,2}(\Omega,\mu)}
=0.
\end{equation} 
The limit in~\eqref{e-cont-sob-c} is monotone non-increasing. 
\end{lemma} 
 
\begin{proof}
Inequality~\eqref{e-cont-sob} follows from the 
proof of Lemma~\ref{L:tlcdense}. The monotonicity and the limit~\eqref{e-cont-sob-c} is also an elementary property of averages, and the zero value of this limit follows, by contradiction, from Lemma~\ref{L:W22loc} and elementary measure theory. The last assertion follows from Lemma~\ref{L:Sobolev_Charact}. 
\end{proof}

In the following corollary we somewhat abuse notion by considering the maps \PHii\ on different functional spaces. However, in each case the natural domain of definition of \PHii\ is evident.
 \begin{corollary} With the notation convention given above, for each $i,k\geqslant1$, the map \PHii\ is: 
 \begin{enumerate}
 \item a contractive projection from $C(\OLe)$ onto a proper subspace of $C_{tlc}(\OLe)$;
 \item a contractive projection from $C^k(\OLe)$ onto a proper subspace of $C_{tlc}^k(\OLe)$;
\item an orthogonal projection from the Hilbert space $\mathcal{W}^{k,2} (\OLe,\mu)$ onto a proper subspace of $\mathcal{W}^{k,2}_{tlc} (\OLe,\mu)$;
\item an orthogonal projection from the Hilbert space $\mathcal{W}^{k,2}_0 (\OLe,\mu)$ onto a proper subspace of $\mathcal{W}^{k,2}_{0,tlc} (\OLe,\mu)$, where $\mathcal{W}^{k,2}_0 (\OLe,\mu)$ is defined as the closure in $\mathcal{W}^{k,2}(\OLe,\mu)$ of the space of smooth functions compactly supported in \OLe. 
 \end{enumerate}
 Moreover, 
 $C^\infty_{tlc}(\OLe)$ is dense in $\mathcal{W}^{k,2}(\OLe,\mu)$; 
 $C^\infty_{0,tlc}(\OLe)$ is dense in $\mathcal{W}^{k,2}_0(\Omega,\mu)$ and 
 $C^\infty_{tlc}(\Omega)$ is dense in $\mathcal{W}^{k,2}(\Omega,\mu)$. 
\end{corollary}

\section{Proof of Theorem \ref{T:main_L2sg}:
 $L^2$-semigroup, Dirichlet form and spectral properties}\label{S:L2}

In this section we work under Assumptions \ref{A:flcapr} and \ref{A:uniqueergodic} and take the unique ergodic measure $\mu$ into account everywhere, even if this is not mentioned explicitly. The results of this section prove Theorem \ref{T:main_L2sg}.

\subsection{Invariance and symmetry}
We begin with the most basic properties of the semigroup $P_t$ that are needed in Theorem~\ref{T:main_L2sg}.

\begin{lemma}\label{L:symmetric}\mbox{ }
\begin{enumerate}
 \item The measure $\mu$ is invariant for $(T_t)_{t>0}$, and we have $\lim_{t\to\infty} T_t(\Lambda,\cdot)=\mu$ in the weak sense for any $\Lambda\in \Omega$. The Feller semigroup $(T_t)_{t>0}$ is $\mu$-symmetric, that is, we have 
\[
\left\langle T_tf, g\right\rangle_{L^2(\Omega,\mu)}=\left\langle f, T_t g\right\rangle_{L^2(\Omega,\mu)},\qquad f,g\in C(\Omega).
\]
It extends uniquely to a conservative Markov semigroup $(P_t)_{t>0}$ of self-adjoint operators on $L^2(\Omega,\mu)$. 
 
\item \label{L:repstillholds} 
If $f\geqslant0$ is a Borel function such that $\int_\Omega |f|^2d\mu<+\infty$, 
 then we have $$P_tf(\Lambda)=\int_{\mathbb{R}^d} p_{\mathbb{R}^d}(t,\vs\,) f(\varphi_{\vs\,}(\Lambda))\,d\vs$$
 for $\mu$-a.e. $\Lambda\in \Omega$ in the sense that the Borel function on the right hand side is in the $L^2(\Omega,\mu)$-class on the left hand side. 
\end{enumerate}
\end{lemma}

\begin{proof}
The invariance of $\mu$ with respect to the action of $\mathbb{R}^d$ and Fubini's theorem imply that $\mu$ is $(T_t)_{t>0}$-invariant. Using the symmetry of $p_{\mathbb{R}^d}(t,\vs\,)$ we can similarly see that $(T_t)_{t>0}$ is $\mu$-symmetric. The existence, uniqueness and conservativity of the extension are clear from the density of $C(\Omega)$ and contractivity. For any $\Lambda\in \Omega$ the family of probability measures $(T_t)_{t>0}$ is tight on the compact space $\Omega$ so that by Prohorov's theorem each sequence of times going to infinity has a subsequence $t_k\uparrow \infty$ for which $\widetilde{\mu}=\lim_{k\to\infty} T_{t_k}(\Lambda,\cdot)$ in the weak sense. However, any probability measure $\widetilde{\mu}$ appearing as such a limit point is also invariant under $\mathbb{R}^d$ action, and so $\widetilde{\mu}=\mu$, since $\mu$ is the unique $\mathbb{R}^d$-invariant probability measure on $\Omega$. 

To show the invariance of $\widetilde{\mu}$ under the translation by $\vt$, one can either apply standard tools from dynamical systems, or in our case use the fact that for any $\vt, \vs \in\mathbb{R}^d$ we have
$
 \frac{p_{\mathbb{R}^d}(t_k,\vs+\vt\,)}{p_{\mathbb{R}^d}(t_k,\vs\,)}
 =
 \exp\Big\{ \frac{-2 \langle \vs,\vt\rangle -|\vt\,|^2}{2t_k}\Big\} \xrightarrow[{k\to\infty}]{}1
$.
By the dominated convergence, for any $f\in C(\Omega)$ we have
\[
\lim_{k\to\infty} \int_{\mathbb{R}^d} p_{\mathbb{R}^d}(t_k,\vs\,)f(\varphi_{\vt+\vs\,}(\Lambda))d\vs=\lim_{k\to\infty} \int_{\mathbb{R}^d} p_{\mathbb{R}^d}(t_k,\vs\,)f(\varphi_{\vs\,}(\Lambda))d\vs.
\]

To prove item (\ref{L:repstillholds}) of this lemma, we can apply the usual nondecreasing approximation of $f$ by the functions $f_n=\min\{n,f\}$ and the monotone convergence theorem. 
\end{proof}

\begin{remark}
The extension of $(T_t)_{t>0}$ to an $L^p$-contractive semigroup was considered in higher generality in \cite[Proposition 7.1]{Candel}, but symmetry in $L^2$ was not discussed. 
\end{remark}
The commutativity relation \eqref{E:KoopmanP} is a direct consequence of the definitions of 
$\Koo_{\vt\,}$ and $P_t$ and \eqref{E:KoopmanT}.

A peculiar feature of the semigroup $(P_t)_{t>0}$ is that, as stated in Theorem \ref{T:main_L2sg}, it does not admit a heat kernel with respect to $\mu$.
\begin{lemma}\label{L:noHK}
The semigroup $(P_t)_{t>0}$ does not admit a heat kernel with respect to the unique invariant measure $\mu$. More precisely, there is no measurable function $p_\mu\colon(0,+\infty)\times \Omega \times \Omega\to\mathbb{R}$ such that 
\[
P_tf(\Lambda_1)=\int_{\Omega} p_\mu(t,\Lambda_1,\Lambda_2)f(\Lambda_2)\mu(d\Lambda_2)
\]
$\mu$-a.e. $\Lambda_1\in\Omega$ for any $f\in b\mathcal{B}(\Omega)$ and $t>0$. 
\end{lemma}

\begin{proof}
Otherwise, since $\mu(\orb(\Lambda_1))=0$, one would obtain that 
\begin{align*}
1=P_t\mathbbm{1}(\Lambda_1)&=\int_\Omega p_\mu(t,\Lambda_1,\Lambda_2)\,\mu(d\Lambda_2)\\
&=\int_{(\orb(\Lambda))^{c}}p_\mu(t,\Lambda_1,\Lambda_2)\,\mu(d\Lambda_2)=\mbbP(X_t^{\Lambda_1}\notin\orb(\Lambda_1)),
\end{align*}
a contradiction.
\end{proof}

  A closely related fact is that, as stated in Theorem \ref{T:main_L2sg}, the semigroup does not improve integrability. This fact can be deduced from the following lemma.
\begin{lemma}\label{L:nothyper}
For every $2<q< +\infty$ there exists a function $f\in L^2(\Omega,\mu)$ such that for all $t>0$ we have $\left\|P_tf\right\|_{L^q(\Omega,\mu)}=+\infty$. 
\end{lemma}

\begin{proof}
Let $\mathcal{O}=\phi^{-1}(\mathcal{C}\times B)$ be an open set as in (\ref{E:transcylsets}) and let $(E_k)_{k=1}^\infty$ be a sequence of pairwise disjoint subsets $E_k$ of $\mathcal{C}$ of positive measure $0<\nu(E_k)<1$  such that $\sum_{k=1}^\infty\nu(E_k)^{s_+}<+\infty$ for some $0<s_+<1$. Since,
because of the unique ergodicity,  
there are subsets of $\mathcal{C}$ of nonzero but arbitrarily small $\nu$-measure, 
such a sequence can be found (otherwise we could find an orbit with positive $\mu$-measure, a contradiction). 
Let $s_-:=\inf\left\lbrace 0<s\leq s_+: \sum_{k=1}^\infty \nu(E_k)^s<+\infty\right\rbrace$, choose a number $s$ such that $s_-<s<1$ and $1+q(s-1)/2<s_-$. Then the function
\[f_0=\sum_{k=1}^\infty \nu(E_k)^{(s-1)/2}\mathbf{1}_{E_k}\]
is in $L^2(\mathcal{C},\nu)$ but $\left\|f_0\right\|_{L^q(\mathcal{C},\nu)}=+\infty$. Now $F_0=\mathbf{1}_{B'}$, where $B'$ is a nonempty open Euclidean ball whose closure is contained in $B$, and consider the function $f(\overline{\Lambda})=f_0(\Lambda')F_0(\vt\,)$, $(\Lambda',\vt\,)=\phi(\overline{\Lambda})$ on $\mathcal{O}$ as in (\ref{e-prod}). Clearly $f\in L^2(\Omega,\mu)$. However, for any $t>0$ we have
\[\left\|P_tf\right\|_{L^q(\Omega,\mu)}^q\geq \left\|f_0\right\|_{L^q(\mathcal{C},\nu)}^q\left(\int_B\left|\int_{\mathbb{R}^d} p_{\mathbb{R}^d}(t,\vs\,)\mathbf{1}_{B'}(\vt+\vs\,)d\vs\,\right|^q d\vt\,\right)=+\infty,\]
note that the second factor on the right hand side is strictly positive.
\end{proof}

\subsection{$L^2$-generator and quadratic forms}
Let $(\mathcal{L},\mathcal{D}(\mathcal{L}))$ be the $L^2(\Omega,\mu)$-generator of $(P_t)_{t>0}$, i.e. the unique non-positive definite self-adjoint operator in $L^2(\Omega,\mu)$ defined by 
\[
\mathcal{D}(\mathcal{L}):=\left\lbrace f\in L^2(\Omega,\mu): \lim_{t\to 0}\frac{1}{t}(P_tf-f) \text{ exists strongly in $L^2(\Omega,\mu)$}\right\rbrace
\]
and 
\[
\mathcal{L}f:=\lim_{t\to 0}\frac{1}{t}(P_tf-f),\quad f\in \mathcal{D}(\mathcal{L}).
\]

The following lemma is implied by (\ref{E:Fellerdomain}), Lemma~\ref{T:TisSg}, and the density of $C_{tlc}^\infty(\Omega)$ in $C(\Omega)$. 
\begin{lemma}\label{L:L2domain} 
We have $C^2(\Omega)\subset \mathcal{D}(\mathcal{L}_{C(\Omega)})\subset \mathcal{D}(\mathcal{L})$, and for $f\in C^2(\Omega)$ the identity
\begin{equation}\label{E:justDelta}
\mathcal{L}f=\frac12\Delta f
\end{equation}
holds in $L^2(\Omega,\mu)$. The space $C_{tlc}^\infty(\Omega)$ is dense in $\mathcal{D}(\mathcal{L})$ and $\mathcal{L}$ is a local operator.
\end{lemma}

\begin{remark}\label{R:productform}
Given an open set $O=\phi^{-1}(\mathcal{C}\times B)$ as in (\ref{E:transcylsets}) we write $(\phi^{-1})^\ast L^2(\mathcal{C},\nu_{\mathcal{C}})\otimes C_c^2(B))$ for the space of finite linear combinations of functions $f$ of product form (\ref{e-prod}) with $f_0\in L^2(\mathcal{C},\nu_{\mathcal{C}})$ and $F_0\in C_c^2(B)$. Extending it by zero, we consider such a function $f$ as a function on all of $\Omega$. For any open set $O=\phi^{-1}(\mathcal{C}\times B)$ as in (\ref{E:transcylsets}) the space $(\phi^{-1})^\ast (L^2(\mathcal{C},\nu_{\mathcal{C}})\otimes C_c^2(B))$ is contained in $\mathcal{D}(\mathcal{L})$, and on this space (\ref{E:justDelta}) holds. This is related to Lemmas~\ref{L:LO} and \ref{L:LO2}.
\end{remark}

%
%


From the general theory of semigroups, see for instance~\cite[Section VIII.6]{RS80}, there is a unique closed quadratic form $(\mathcal{E},\mathcal{D}(\mathcal{E}))$ on $L^2(\Omega,\mu)$ associated with the Markovian semigroup $(P_t)_{t>0}$, which is defined by 
\begin{equation}\label{E:Dformdomain}
\mathcal{D}(\mathcal{E}):=\left\lbrace f\in L^2(\Omega,\mu): \sup_{t>0}\frac{1}{t}\left\langle f-P_tf,f\right\rangle_{L^2(\Omega,\mu)}<+\infty\right\rbrace
\end{equation}
and 
\begin{equation}\label{E:Dform}
\mathcal{E}(f):=\lim_{t\to 0}\frac{1}{t}\left\langle f-P_t f, g\right\rangle_{L^2(\Omega,\mu)},\quad f,g\in\mathcal{D}(\mathcal{E}).
\end{equation}
To the operator $(\mathcal{L},\mathcal{D}(\mathcal{L}))$ the form $(\mathcal{E},\mathcal{D}(\mathcal{E}))$ is uniquely related by the identity
\begin{equation}\label{E:L2gen}
\mathcal{E}(f,g)=-\left\langle \mathcal{L}f,g\right\rangle_{L^2(\Omega,\mu)}
\end{equation}
for all $f\in \mathcal{D}(\mathcal{L})$ and $g\in\mathcal{D}(\mathcal{E})$. Moreover, it is a Dirichlet form, see e.g.~\cite[Theorems 1.3.1 and 1.4.1]{FOT94}. 

The following lemma should be considered together with Lemma  \ref{L:Soboanddomains} below. 
\begin{lemma}\label{L:Dform}
The Dirichlet form $(\mathcal{E},\mathcal{D}(\mathcal{E}))$ is regular and strongly local. The space $C_{tlc}^\infty(\Omega)$ is dense in $\mathcal{D}(\mathcal{E})$. The gradient operator $\nabla$ extends to a closed unbounded operator $\nabla:L^2(\OO,\mu)\to L^2(\OO,\mathbb{R}^d,\mu)$
with dense domain $\mathcal{D}(\mathcal{E})$, and the identity
\begin{equation}\label{e-ee}
\mathcal{E}(f,g)=\frac12\int_{\Omega}\left\langle \nabla f,\nabla g\right\rangle d\mu
\end{equation}
holds for any $f,g\in\mathcal{D}(\mathcal{E})$. Moreover, $(\mathcal{E},\mathcal{D}(\mathcal{E}))$ admits a carr\'e du champ given by $\Gamma(f,g):=\left\langle \nabla f,\nabla g\right\rangle d\mu$.
\end{lemma}
\begin{proof}
The regularity follows from Lemma~\ref{L:symmetric} and the arguments in~\cite[Lemma 2.8]{BBKT10}, which are classical, but not widely available. 
The strong locality follows from Lemma~\ref{L:L2domain} due to the conservativeness of $(P_t)_{t>0}$ and the locality of $\mathcal{L}$. The latter lemma also implies the density of $C_{tlc}^\infty(\Omega)$. For functions from $C_{tlc}^\infty(\Omega)$ identity (\ref{e-ee}) follows from (\ref{E:L2gen}) and Lemma \ref{L:ibP}. 
By density, the gradient operator $\nabla$ and formula (\ref{e-ee}) extend to $\mathcal{D}(\mathcal{E})$ as stated. The last statement is a direct application of~\cite[Definition I.4.1.2 and Theorem I.4.2.1]{BH91}. 
\end{proof}

\begin{remark}
From~\eqref{E:ergodic} it is immediate that for any F\o lner sequence $(A_n)_n$ and any $f,g\in C_{tlc}^1(\Omega)$ we have
\[
\mathcal{E}(f,g)=\lim_{n\to\infty}\frac{1}{2\lambda^d(A_n)}\int_{A_n}\left\langle \nabla_{\mathbb{R}^d} h_{\Lambda}^\ast f(\vt\,), \nabla_{\mathbb{R}^d} h_{\Lambda}^\ast g(\vt\,)\right\rangle \:d\vt
\]
uniformly for every $\Lambda\in \Omega$, where the orbit homeomorphism $ h_{\Lambda}^\ast$ is defined in \eqref{E:algiso}.
\end{remark}

We record another simple property of $(\mathcal{E},\mathcal{D}(\mathcal{E}))$. Given an open subset $G\subset \mathbb{R}^d$, we write
\begin{equation}
\mathcal{E}_{\mathbb{R}^d}^G(f,g)=\frac12\int_G\left\langle \nabla_{\mathbb{R}^d} f, \nabla_{\mathbb{R}^d} g\right\rangle dx
\end{equation}
whenever $f$ and $g$ are functions on $G$ such that the expression makes sense. For functions with the product form (\ref{e-prod}) we can characterize whether they are in $\mathcal{D}(\mathcal{E})$ or not. 

\begin{lemma}\label{L:productf0F0}
Let $O=\phi^{-1}(\mathcal{C}\times B)$ be an open set of type~\eqref{E:transcylsets} and let $f\in b\mathcal{B}(\Omega)$ with $\supp f\subset O$ be a function of the product form~\eqref{e-prod} for some $f_0\in b\mathcal{B}(\mathcal{C})$ and $F_0\in b\mathcal{B}(B)$ compactly supported. Then, $f\in\mathcal{D}(\mathcal{E})$ if and only if $\mathcal{E}^B_{\mbbR^d}(F_0,F_0)<\infty$. In particular, $\mathcal{D}(\mathcal{E})\not\subset C(\Omega)$ for any $d\geqslant 1$. 
\end{lemma}

\begin{proof}
If $F_0\in b\mathcal{B}(B)$ has compact support in $B$ and finite $\mathcal{E}^B_{\mathbb{R}^d}$-energy, Lemma~\ref{L:symmetric}(\ref{L:repstillholds}), formula (\ref{E:Dform}) and the corresponding semigroup approximation for $\mathcal{E}^B_{\mathbb{R}^d}$ yield $f\in\mathcal{D}(\mathcal{E})$.
\end{proof}

\begin{remark} Many typical functional inequalities fail to hold for $(\mathcal{E},\mathcal{D}(\mathcal{E}))$. We follow \cite[Proposition II.1 (and the Remark following it)]{Coulhon96} and say that a Dirichlet form $(\mathcal{Q},\mathcal{D}(\mathcal{Q}))$ on $L^2(\Omega,\mu)$ \emph{satisfies a Nash type inequality} if there exist a continuous function $\theta:(0,+\infty)\to (0,+\infty)$ with $\int_0^{+\infty}\frac{ds}{\theta(s)}<+\infty$ and positive constants $c_1$ and $c_2$ such that 
\[\theta\left(c_1\:\left\|f\right\|_{L^2(\Omega,\mu)}^2\right)\leq c_2\:\mathcal{Q}(f,f)\]
for all $f\in\mathcal{D}(\mathcal{Q})$ with $\left\|f\right\|_{L^1(\Omega,\mu)}=1$. For more classical formulations of this inequality see \cite{CKS87} or \cite{Nash58}. By \cite[Proposition II.1 and its proof]{Coulhon96} it follows that if a Dirichlet form satisfies a Nash type inequality then the associated Markov semigroup of self-adjoint operators is ultracontractive (see also \cite[(2.1) Theorem]{CKS87}). 
Therefore Theorem \ref{T:main_L2sg} (respectively Lemma \ref{L:nothyper} and Remark \ref{R:HKs}) imply the following.

\begin{corollary}\label{C:noNash}
The Dirichlet form $(\mathcal{E},\mathcal{D}(\mathcal{E}))$ does not satisfy a Nash type inequality.
\end{corollary}
\end{remark}

\begin{remark}
Notice also that $(\mathcal{E},\mathcal{D}(\mathcal{E}))$ does not satisfy a local Poincar\'e inequality. More precisely, there are open sets $O\subset \Omega$ of type (\ref{E:transcylsets}) and functions $f\in \mathcal{D}(\mathcal{E})$ such that 
\[
\int_O |\nabla f|^2 d\mu=0 \quad \text{ but }\ \ \int_O|f-f_O|^2d\mu>0.
\]
To see this, let $V$ and $O=\phi^{-1}(\mathcal{C}\times B)$ be open sets of type (\ref{E:transcylsets}) with $\overline{O}\subset V$, and let $\chi\in C_{tlc}^\infty(\Omega)$ be a function with $\supp\chi\subset V$ and $\chi\equiv 1$ on $O$, cf. Corollary \ref{C:bumps}. If $\mathcal{C}=\mathcal{C}_{\Lambda,\varepsilon}$, consider the function $f(\overline{\Lambda}):=\mathbbm{1}_{\mathcal{C}_{\Lambda,\varepsilon'}}(\Lambda')\chi(\vt\,)$ with $0<\varepsilon'<\varepsilon$ such that $\nu_{\mathcal{C}_{\Lambda,\varepsilon'}}<\nu_{\mathcal{C}_{\Lambda,\varepsilon}}$. By Lemma \ref{L:productf0F0}, $f\in\mathcal{D}(\mathcal{E})$, and by locality we have $\int_O |\nabla f|^2 d\mu=0$, but $\left\|f-f_O\right\|_{L^2(O,\mu)}^2>0$.
\end{remark}


\subsection{Caracterization of domains}\label{SS:SoboDom}

We can give the following descriptions of the domains $\mathcal{D}(\mathcal{L})$ and $\mathcal{D}(\mathcal{E})$ in terms of Sobolev spaces. They are to be understood as equalities of vector spaces with equivalent norms.

\begin{lemma}\label{L:Soboanddomains}
We have $\mathcal{D}(\mathcal{L})=\mathcal{W}^{2,2}(\Omega,\mu)$ and $\mathcal{D}(\mathcal{E})=\mathcal{W}^{1,2}(\Omega,\mu)$. 
\end{lemma}

\begin{proof} 
We prove the first statement, and the second follows similarly. For any $f\in C_{tlc}^\infty(\Omega)$ we have 
\[
\left\|f\right\|_{\mathcal{W}^{2,2}(\Omega,\mu)}\leq c(\left\|f\right\|_{L^2(\Omega,\mu)}+\left\|\Delta f\right\|_{L^2(\Omega,\mu)})
\] 
by~\eqref{E:L2gen} and Lemmas \ref{L:L2domain} and \ref{lem-w22}. Therefore, any sequence $(f_n)_n\subset C_{tlc}^\infty(\Omega)$ that is Cauchy in $\mathcal{D}(\mathcal{L})$ 
with respect to the norm 
\begin{equation}
\|f\|_{ \mathcal{D}(\mathcal{L}) }=\left(\left\|f\right\|_{L^2(\Omega,\mu)}^2+\left\| \mathcal{L} f\right\|_{L^2(\Omega,\mu)}^2\right)^{\frac12}
\end{equation}
is also Cauchy in $\mathcal{W}^{2,2}(\Omega,\mu)$. The density of $C_{tlc}^\infty(\Omega)$ in $\mathcal{D}(\mathcal{L})$ implies that $\mathcal{D}(\mathcal{L})\subset \mathcal{W}^{2,2}(\Omega,\mu)$. To prove the equality of these two spaces 
we use the equivalent norm $\left\|\cdot\right\|_{\mathcal{W}^{2,2}(\Omega,\mu)}'$ in $\mathcal{W}^{2,2}(\Omega,\mu)$ from Lemma \ref{lem-w22}, denoting the associated scalar product in $\mathcal{W}^{2,2}(\Omega,\mu)$ by $[\cdot,\cdot]$. We will show that if a function $g\in\mathcal{W}^{2,2}(\Omega,\mu)$ satisfies $[g,f]=0$ for all $f\in\mathcal{D}(\mathcal{L})$ then $g=0$ in $\mathcal{W}^{2,2}(\Omega,\mu)$. This entails $\mathcal{D}(\mathcal{L})= \mathcal{W}^{2,2}(\Omega,\mu)$. Using $C_{tlc}^\infty(\Omega)$ a partition from Lemma~\ref{L:partofunity} we may assume that $g$ is supported in an open set $O=\phi^{-1}(\mathcal{C}\times B)$ of type (\ref{E:transcylsets}). Then, for any $\Lambda'\in\mathcal{C}$, the function $g\circ \phi^{-1}(\Lambda',\cdot)$ has compact support in $B$. It suffices to test $g$ with functions $f\in \phi^{-1}(L^2(\mathcal{C},\nu_{\mathcal{C}})\otimes C_c^2(B))$ of product form (\ref{e-prod}), for which we have $-\left\langle \Delta g, f\right\rangle_{L^2(O,\mu)}=\left\langle \nabla g,\nabla f\right\rangle_{L^2(O,\mu)}=-\left\langle g, \Delta f\right\rangle_{L^2(O,\mu)}$. The latter follows from the fact that
\begin{multline}
-\int_{\mathcal{C}} f_0(\Lambda')\int_B \Delta_{\mathbb{R}^d}g\circ \phi^{-1}(\Lambda',\vt\,)F_0(\vt\,)\:d\vt\nu_{\mathcal{C}}(d\Lambda')\notag\\
=\int_{\mathcal{C}} f_0(\Lambda')\int_B \nabla_{\mathbb{R}^d}g\circ \phi^{-1}(\Lambda',\vt\,) \nabla_{\mathbb{R}^d} F_0(\vt\,)(d\Lambda')\:d\vt \nu_{\mathcal{C}}(d\Lambda'),
\end{multline}
which can be verified by approximating $F_0\in C_c^2(B)$ with $C^\infty_c(B)$-functions. 
This implies that $[f,g]=\left\langle (1-\Delta)g, (1-\Delta)f\right\rangle_{L^2(O,\mu)}$ and therefore
\[
0= \int_{\mathcal{C}}f_0(\Lambda')\int_B (1-\Delta_{\mathbb{R}^d}) g\cphi (\Lambda',\vt\,)(1-\Delta_{\mathbb{R}^d}) F_0(\vt\,)\:d\vt\,\nu_{\mathcal{C}}(d\Lambda').
\]
Varying $f_0$ we can deduce that 
\begin{equation}\label{E:annihilatetensors}
0= \int_B (1-\Delta_{\mathbb{R}^d}) g\cphi (\Lambda',\vt\,)(1-\Delta_{\mathbb{R}^d}) F_0(\vt\,)\:d\vt
\end{equation}
holds a priori for all $\Lambda'\in\mathcal{C}$ outside some $\nu_{\mathcal{C}}$-null set that a priori may depend on $F_0$. Due to the separability of $C_c^2(B)$ we can find a Borel set $\mathcal{N}\subset \mathcal{C}$ of zero $\nu_{\mathcal{C}}$-measure such that (\ref{E:annihilatetensors}) holds for all $\Lambda'\in\mathcal{C}\setminus \mathcal{N}$ and all $F_0\in C_c^2(B)$. For any such $\Lambda'$, let $\chi_{\Lambda'}\in C_c^\infty(B)$ be a cut-off function such that $0\leq \chi_{\Lambda'}\leq 1$ and $\chi_{\Lambda'}\equiv 1$ on a neighbourhood of the support of $(1-\Delta_{\mathbb{R}^d}) g\cphi (\Lambda',\cdot)$. Such a function exists because of the locality of $\Delta_{\mathbb{R}^d}$. 
Defining $F_0:=\chi_{\Lambda'}(1-\Delta_{\mathbb{R}^d}^B)^{-1}\psi$, where $\psi\in C_c(B)$ and $(1-\Delta_{\mathbb{R}^d}^B)^{-1}$ is the $1$-resolvent of the Dirichlet Laplacian $\Delta_{\mathbb{R}^d}^B$ on $B$, we have $F_0\in C_c^2(B)$ by classical elliptic regularity~\cite[Chapter 6]{GT01}, and can conclude 
\[
0= \int_B (1-\Delta_{\mathbb{R}^d}) g\cphi (\Lambda',\vt\,)\psi(\vt\,)\:d\vt.
\]
Since this is true for any $\Lambda'\in\mathcal{C}\setminus \mathcal{N}$ and $\psi\in C_c(B)$, we obtain $(1-\Delta_{\mathbb{R}^d}) g=0$ in $L^2(B)$, which implies $g=0$ in $\mathcal{W}^{2,2}(\Omega,\mu)$.
\end{proof}

The next Corollary follows from Lemma \ref{L:W22loc}.
\begin{corollary} \label{L:W22loc-}\mbox{\ }
\begin{enumerate}
\item If a measurable function $f$ on $\Omega$ is in $\mathcal{D}(\mathcal{E})$ then $h_{\Lambda}^\ast f\in W^{1,2}_{loc}(\mathbb{R}^d)$ for $\mu$-a.e. $\Lambda\in \Omega$.
\item If a measurable function $f$ on $\Omega$ is in $\mathcal{D}(\mathcal{L})$ then $h_{\Lambda}^\ast f\in W^{2,2}_{loc}(\mathbb{R}^d)$ for $\mu$-a.e. $\Lambda\in \Omega$.
\item
If $f\in b\mathcal{B}(\Omega)$ is compactly supported in $O=\phi^{-1}(\mathcal{C}\times\mathcal{B})$ and $f\in\mathcal{D}(\mathcal{L}\cphi )$ then for $\mu$-a.e. $\Lambda\in \Omega$ the restriction of $h_{\Lambda}^\ast f$ to any connected component of $O\cap \orb(\Lambda)$ is a member of $\mathring{W}^{1,2}(B)\cap W^{2,2}(B)$.
\end{enumerate}
\end{corollary}

\subsection{Recurrence and Kusuoka-Hino index} Recall that the semigroup $(P_t)_{t>0}$ is said to be \emph{recurrent} if for any nonnegative $f\in L^1(\Omega,\mu)$ we have $Gf=+\infty$ or $Gf=0$ $\mu$-a.e. where $Gf=\int_0^\infty P_tf dt$. Following \cite[p.48]{FOT94} we say that a Dirichlet form is recurrent if its semigroup is recurrent. An application of \cite[Theorem 1.6.3]{FOT94} (see also \cite[p.45 and Theorem Theorem 2.1.8]{ChFu12}) immediately yields that the Dirichlet form $(\mathcal{E},\mathcal{D}(\mathcal{E}))$ is recurrent.
 
The the concept of \emph{pointwise Kusuoka-Hino index (or briefly, pointwise index)} 
 for strongly local regular Dirichlet forms was studied
in \cite{Hino10}, based on the \emph{martingale dimension} of 
Dirichlet forms introduced by Kusuoka (see \cite{Kusuoka89}). By definition, the pointwise index $p$ 
of $(\mathcal{E},\mathcal{D}(\mathcal{E}))$ is a Borel measurable function $p:\Omega\to \mathbb{N}\cup \left\lbrace +\infty\right\rbrace$ such that for any $N$ and any $f_1,...,f_N\in\mathcal{D}(\mathcal{E})$ we have $\rank \left(\left\langle \nabla f_i,\nabla f_j\right\rangle\right)_{i,j=1}^N \leq p$ $\mu$-a.e. on $\Omega$, and if $\widetilde{p}$ is another such function that satisfies (i) in place of $p$, then $p\leq \widetilde{p}$ $\mu$-a.e. on $\Omega$. See \cite[Definition 2.9]{Hino10} (or \cite[Definition 2.2]{Hino12}). The {martingale dimension of $(\mathcal{E},\mathcal{D}(\mathcal{E}))$} is defined as the essential supremum of the pointwise index. Using the arguments of \cite[Proposition 2.12]{Hino10} and \cite[Example 2.15]{Hino10} together with Lemma \ref{L:directionsforindex} we can see that the pointwise index of $(\mathcal{E},\mathcal{D}(\mathcal{E}))$ equals $d$ $\mu$-a.e. and therefore its martingale dimension equals $d$.

\subsection{Spectrum in the product neighborhood \OLe}
Let $D^\alpha_{\mathbb{R}^d}$ denote the distributional derivative of order $\alpha$ in the Euclidean sense, see e.g.~\cite[Section 5.2.1]{Evans98} 
or~\cite[Section 7.3]{GT01}. For any open set $G\subset \mathbb{R}^d$ and any $k=1,2,\ldots$ the Sobolev space $W^{k,2}(G)$ is defined as 
\[
W^{k,2}(G):=\left\lbrace f\in L^2(G):~D^\alpha_{\mathbb{R}^d} f\in L^2(G)\quad\text{for all }|\alpha|\leq k\right\rbrace,
\]
endowed with the norm 
\[
\left\|f\right\|_{W^{k,2}(G)}:=\left(\sum_{|\alpha|\leq k} \left\|D^\alpha_{\mathbb{R}^d} f\right\|_{L^2(G)}^2\right)^{1/2}.
\] 
Moreover, $W_0^{k,2}(G)$ denotes the closure of $C_c^\infty(G)$ in $W^{k,2}(G)$. Both $W^{k,2}(G)$ and $W_0^{k,2}(G)$ are Hilbert spaces and we refer to~\cite[Section 5]{Evans98} or~\cite[Section 7.5]{GT01} 
for further definitions and basic results about these spaces.

If $G$ is bounded and its boundary $\partial G$ is smooth, the self-adjoint Laplacian $\Delta_{\mathbb{R}^d}^G$ on $G$ with Dirichlet boundary conditions is obtained as the Friedrichs extension of $(\Delta_{\mathbb{R}^d},C_c^\infty(G))$ in $L^2(G)$. This operator has domain 
$W^{1,2}_0(G)\cap W^{2,2}(G)$, see e.g.~\cite[Theorem 8.12]{GT01} or \cite[Chapter VI, Remark 1.7]{EE87} and it is a non-positive definite self-adjoint operator on $L^2(G)$ with pure point spectrum.

In this subsection, we consider an open subset of $\Omega$ of type~\eqref{E:transcylsets}, i.e. $O=\OLe=\phi^{-1}(\mathcal{C}\times B)$, and on the open ball $B\subset \mathbb{R}^d$ consider the Dirichlet Laplacian $\left(\frac12\Delta_{\mathbb{R}^d}^{B}, W_0^{1,2}(B)\cap W^{2,2}(B)\right)$. Let $0<\lambda_1^{B}\leq\lambda_2^{B}\leq\ldots$ denote the eigenvalues of $-\Delta_{\mathbb{R}^d}^{B}$, ordered with multiplicities taken into account and let $\{ b_j^{B}\}_{j=1}^\infty$ be an orthonormal basis in $L^2(B)$ of eigenfunctions of $\Delta_{\mathbb{R}^d}^{B}$, where $b_j^{B}$ is the eigenfunction corresponding to $\lambda_j^{B}$. By classical theory, $b_j^{B}\in C^\infty(B)$.

For each $i,j\geq 1$, we define the function $b_{ij}^{O}\in (\phi^{-1})^\ast (L^2(\mathcal{C},\nu_{\mathcal{C}})\otimes W_0^{1,2}(B)\cap W^{2,2}(B))$ by
\begin{equation}\label{E:local_ONB}
b_{ij}^{O}(\overline{\Lambda}):=b_i^{\mathcal{C}}(\Lambda') b_j^{B}(\vt\,), \quad \overline{\Lambda}=\varphi_{\vt}\,(\Lambda')\in O.
\end{equation}
The collection $\{ b_{ij}^{O}\}_{i,j=1}^\infty$ is an orthonormal basis of $L^2(O,\mu)$ and 
\begin{multline}\label{E:domainviaspectral}
\mathcal{D}(\mathcal{L}^{O}):=\mathcal{W}_0^{1,2}(O,\mu)\cap\mathcal{W}^{2,2}(O,\mu)=
\\ 
\left\lbrace f\in L^2(O,\mu): \sum_{i,j=1}^\infty \left(\lambda_j^{B}\right)^2 |\left\langle f, b_{ij}^O \right\rangle_{L^2(O,\mu)}|^2<+\infty\right\rbrace 
\end{multline}
as well as
\begin{equation}\label{E:defLO}
\mathcal{L}^O f:=\sum_{i,j=1}^\infty\lambda_k^{B} \left\langle f, b_{ij}^O \right\rangle_{L^2(O,\mu)}b_{ij}^O , \quad f\in\mathcal{D}(\mathcal{L}^O ).
\end{equation}
In view of the next lemma, the operator $(\mathcal{L}^O ,\mathcal{D}(\mathcal{L}^O ))$ may be regarded as the natural Dirichlet Laplacian on $O$.

\begin{lemma}\label{L:LO}
The operator $(\mathcal{L}^O ,\mathcal{D}(\mathcal{L}^O))$ is a non-positive definite self-adjoint operator on $L^2(O,\mu)$. It has pure point spectrum and each of its eigenvalues has an infinite dimensional eigenspace. The domain $\mathcal{D}(\mathcal{L}^O)$ contains $(\phi^{-1})^\ast(L^2(\mathcal{C},\nu_{\mathcal{C}})\otimes W_0^{1,2}(B)\cap W^{2,2}(B))$ and for any function $f$ in this space with the product form~\eqref{e-prod} we have 
 \[
 \mathcal{L}^O f(\overline{\Lambda})=f_0(\Lambda')\Delta_{\mathbb{R}^d}^B F_0(\vt\,),\quad \overline{\Lambda}=\varphi_{\vt}\,(\Lambda')\in O.
 \]
 Moreover, for any $f\in (\phi^{-1})^\ast(L^2(\mathcal{C},\nu_{\mathcal{C}})\otimes C_c^2(B))$ we have $\mathcal{L}^O f=\Delta f$.
\end{lemma}

\begin{proof}
The first two statements are clear from the construction of $(\mathcal{L}^O ,\mathcal{D}(\mathcal{L}^O))$. If a function $f\in (\phi^{-1})^\ast(L^2(\mathcal{C},\nu_{\mathcal{C}})\otimes W_0^{1,2}(B)\cap W^{2,2}(B))$ has the product form~\eqref{e-prod}, then
\[
\left\langle f, b_{ij}^O \right\rangle_{L^2(O,\mu)}=\left\langle f_0,b_i^{\mathcal{C}}\right\rangle_{L^2(\mathcal{C},\nu_{\mathcal{C}} )} \left\langle F_0,b_j^B\right\rangle_{L^2(B)}
\]
holds for any eigenfunction $b_{ij}^O$. By virtue of~\eqref{E:domainviaspectral} it follows that $f\in \mathcal{D}(\mathcal{L})$ because $F_0\in W_0^{1,2}(B)\cap W^{2,2}(B)$. The remaining statements follow similarly.
\end{proof}

\begin{remark}
 In a similar manner one can study operators on open neighborhoods of the form $\phi^{-1}(\mathcal{C}\times Q)$ where $\mathcal{C}$ is a Cantor set and $Q$ is a {regular enough} set in $\mathbb{R}^d$. The above arguments also allow to consider differential operators more general than the Laplacian $\Delta$.
\end{remark}

We can also localize the quadratic form $\mathcal{E}$. Given an open subset $O$ of $\Omega$ let $\mathcal{D}(\mathcal{E}^O):=\clos(\mathcal{D}(\mathcal{E})\cap C_c(O))$
with the closure taken in $\mathcal{D}(\mathcal{E})$. Writing 
\[\mathcal{E}^O (f,g):=\mathcal{E}(f,g), \quad f,g\in\mathcal{D}(\mathcal{E}^O),\]
we obtain a strongly local regular Dirichlet form $(\mathcal{E}^O ,\mathcal{D}(\mathcal{E}^O ))$ in $L^2(O,\mu)$.
If $O$ is of type (\ref{E:transcylsets}), this is the Dirichlet form uniquely associated with the operator $\mathcal{L}^O $.

The following lemma is a version of the main result in \cite[Section 5.2.1]{BH91}.

\begin{lemma}\label{L:LO2}
 Assume $O$ is an open set of type (\ref{E:transcylsets}). Then $(\mathcal{E}^O ,\mathcal{D}(\mathcal{E}^O ))$ is the Dirichlet form generated by $(\mathcal{L}^O ,\mathcal{D}(\mathcal{L}^O ))$, that is, 
\[
\mathcal{E}^O (f,g)=-\left\langle\mathcal{L}^O f,g\right\rangle_{L^2(O,\mu)}
\]
for all $f\in\mathcal{D}(\mathcal{L}^O )$ and $g\in\mathcal{D}(\mathcal{E}^O )$. The operator $(\mathcal{L}^O ,\mathcal{D}(\mathcal{L}^O ))$ is the Friedrichs extension of $\Delta$ with domain $(\phi^{-1})^\ast(L^2(\mathcal{C},\nu_{\mathcal{C}})\otimes C_c^2(B))$. In addition, $\int_O f^2d\mu\leq \lambda_1^{-1}\:\mathcal{E}^O(f)$, $f\in \mathcal{D}(\mathcal{E}^O )$.
\end{lemma}

\begin{proof}
 We first claim that the space $(\phi^{-1})^\ast(L^2(\mathcal{C},\nu_{\mathcal{C}})\otimes C_c^2(B))$ is a dense subspace of $\mathcal{D}(\mathcal{E}^O ))$. Approximating with continuous functions compactly supported on $\mathcal{C}$ in the first tensor component, we can see that it is contained in $\mathcal{D}(\mathcal{E}^O )$. Now, let $\mathcal{L}^{\mathcal{E}^O }$ denote the generator of $\mathcal{E}^O $ and suppose that $g\in \mathcal{D}(\mathcal{L}^{\mathcal{E}^O })$ is such that  
\[
\big\langle (1-\mathcal{L}^{\mathcal{E}^O })g,f\big\rangle_{L^2(O,\mu)}=\left\langle g,f\right\rangle_{L^2(O,\mu)}+\mathcal{E}^O (g,f)=0
\]
for all $f$ of product form (\ref{e-prod}) with $f_0\in L^2(\mathcal{C},\nu_{\mathcal{C}})$ and $F_0\in C_c^2(B)$. Then, for each $F_0\in C^2_c(B)$, the finite signed measure $\int_B (1-\mathcal{L}^{\mathcal{E}^O })g(\Lambda',\vt\,)F_0(\vt\,)d\vt\,\nu_{\mathcal{C}}(d\Lambda')$ is the zero measure on $\mathcal{C}$.
Using the separability of $C_c^2(B)$ we can find a $\nu_\mathcal{C}$-null set $\mathcal{N}$ such that for all $\Lambda'\in \mathcal{C}\setminus \mathcal{N}$ the function $(1-\mathcal{L}^{\mathcal{E}^O})g(\Lambda',\cdot,)$ is zero $d\vt\,$-a.e. on $B$. This implies that $g\in \ker(1-\mathcal{L}^{\mathcal{E}^O})$ and therefore $f=0$ in $\mathcal{D}(\mathcal{E}^O )$, proving the claimed density.
 
For $g\in (\phi^{-1})^\ast(L^2(\mathcal{C},\nu_{\mathcal{C}})\otimes C_c^2(B))$ we have $g\in\mathcal{D}(\mathcal{L}^{\mathcal{E}^O })$ and $\mathcal{L}^{\mathcal{E}^O }g=\frac12\Delta g$ by (\ref{E:L2gen}) and Remark \ref{R:productform}. Consequently, both $\mathcal{L}^O $ and $\mathcal{L}^{\mathcal{E}^O }$ are self-adjoint extensions of $\frac12\Delta$ endowed with $(\phi^{-1})^\ast(L^2(\mathcal{C},\nu_{\mathcal{C}})\otimes C_c^2(B))$. Since $\mathcal{L}^{\mathcal{E}^O }
 $ is the smallest extension, we have $\mathcal{D}(\mathcal{L}^{\mathcal{E}^O })\subset \mathcal{D}(\mathcal{L}^O )$. 
 
Finally, we claim that $(\phi^{-1})^\ast(L^2(\mathcal{C},\nu_{\mathcal{C}})\otimes W_0^{1,2}(B)\cap W^{2,2}(B))$ is contained in $\mathcal{D}(\mathcal{L}^{\mathcal{E}^O})$ and for all its elements $f$ we have $\mathcal{L}^{\mathcal{E}^O}f=\mathcal{L}^Of$. Then, $\mathcal{D}(\mathcal{L}^O)\subset \mathcal{D}(\mathcal{L}^{\mathcal{E}^O})$ follows from (\ref{E:defLO}). To see this claim note that if $f$ is as above and of product form (\ref{e-prod}), and $g\in (\phi^{-1})^\ast(L^2(\mathcal{C},\nu_{\mathcal{C}})\otimes C_c^2(B))$, then by Lemma \ref{L:LO} and the Gauss-Green identity for the Dirichlet Laplacian $\Delta_{\mathbb{R}^d}^B$ on $B$ we have $\mathcal{E}^O(f,g)=-\big\langle \mathcal{L}^Of,g\big\rangle_{L^2(O,\mu)}$. By approximation in $\mathcal{D}(\mathcal{E}^O)$, this is true for all $g\in\mathcal{D}(\mathcal{E})$ and the claim is proved.
\end{proof}

\section{Proof of Theorem \ref{T:main_Liouville}: 
 harmonic functions and irreducibility of the Dirichlet form}\label{S:harmonic}
 
In this section we discuss harmonic functions and prove Theorem \ref{T:main_Liouville}. We begin with the classical non-probabilistic approach. After that, we present regularity results for measurable harmonic functions and connections to the probabilistic interpretation. In the end we consider finite energy $L^2(\Omega,\mu)$ harmonic functions and prove the irreducibility of the Dirichlet form. 

\subsection{Proof of Theorem \ref{T:main_Liouville}.}
First, assume that a harmonic function $f$ is bounded from above or below up to a set of $\mu$-measure zero. By the classical Liouville theorem, $f$ is constant on $\mu$-almost every orbit and
therefore $f$ is equal to a constant on a set of full $\mu$-measure because of the unique ergodicity. 
To show this, assume that the union $\mathcal{U}$ of all orbits on which the given harmonic function $f$ is not bounded from above (or below) has zero $\mu$-measure. For any $\Lambda\in \mathcal{U}^c$ the function $h_\Lambda^\ast f$ is harmonic on $\mathbb{R}^d$ and bounded from above (or below), so that by the classical Liouville theorem, $h_\Lambda^\ast f$ is constant on $\mathbb{R}^d$, see for instance \cite[Problem I.2.14]{GT01}. This means $f$ is constant on $\mu$-almost every orbit. By unique ergodicity $f$ must then be equal to a constant on a set of full $\mu$-measure.

 Now 
 let $f$ be harmonic and integrable. Assume, for a moment, that $\mu(\mathcal{U})>0$, 
 which by unique ergodicity implies $\mu(\mathcal{U})=1$.   
   For each $\Lambda\in\mathcal{U}$, $h^\ast_\Lambda f$ must have an unbounded growth at infinity.  Then $h^\ast_\Lambda f$ must have an unbounded growth at infinity in the following sense: 
 \begin{equation}\label{e-7.1}
 \lim_{R\to\infty} 
 \int\limits_{\OO}\left(\frac{1}{\lambda^d(B_R(\vec{\,0}\,))}
 \int\limits_{B_R(\vec{\,0}\,)} \big|f\circ \varphi_\vt\:(\Lambda)\big|d\vt\right)\,d\mu(\Lambda)
 =\infty.
 \end{equation}
 This is because if $u:\mathbb R^d\to\mathbb R$ is a harmonic function, then the following average 
 \begin{equation}\label{e-7.2}
 R\mapsto\frac{1}{\lambda^d(B_R(\vec{\,0}\,))}
 \int_{B_R(\vec{\,0}\,)} \big| u (\vt) \big|d\vt 
 \end{equation} 
 is a nondecreasing function of $R$, see for instance \cite[Section 1.VIII.10]{Doob84} (applied to $-|u|$). This grows to infinity as $R\to\infty$ if $u$ is unbounded. 
However, for any measurable function $u:\OO\to[0,\infty) $ we have 
 \begin{equation}\label{e-7.3}
 \frac{1}{\lambda^d(B_R(\vec{\,0}\,))}\int_{B_R(\vec{\,0}\,)}\left(\int_{\OO}
 v\circ \varphi_\vt\:(\Lambda) d\mu(\Lambda) \right)\, d\vt = 
 \int_{\OO}
 v (\Lambda) d\mu(\Lambda)
 \end{equation} 
 because of 
 the fact that  
 $\int_{\OO}
 v\circ \varphi_\vt\:(\Lambda) d\mu(\Lambda)$ does not depend on $\vt$. Thus \eqref{e-7.1} contradicts 
 the assumption that $f$ is integrable, which completes 
 the proof of Theorem~\ref{T:main_Liouville} in this case.
 
 If $d=1$ and $O$ is an open subset of $\Omega$, then a measurable function $f$ is harmonic in $O$ if and only if it is a linear function of $\vt$ on every orbit in $O$. This follows directly from Definition~\ref{def-harm}. 
   Hence we have  $$
   \lim_{R\to\infty}
   \int\limits_{\OO}\left(
   \frac
   { \lambda^d\Big(B_R(\vec{\,0}\,) \cap \{\vt : \big|f\circ \varphi_\vt\:(\Lambda)\big|>c  \}\Big) }
   {\lambda^d\big(B_R(\vec{\,0}\,)\big)}
    \right)\,d\mu(\Lambda)
   =1.
   $$ for   
   any number $c$. 
  Therefore  by  \eqref{E:ergodic}, adapted for bounded measurable functions
  for $\mu$-a.e. $\Lambda\in\OO$, 
   $$
   \mu(   \{\Lambda'\ : |f(\Lambda')|>c \})=1
   ,$$ which   contradicts the assumption that $f$ is a measurable real valued function and completes the proof of 
   Theorem~\ref{T:main_Liouville}.

\subsection{Mean value properties, regularity results and the probabilistic 
 interpretation. }

Definition \ref{def-harm} is based on the classical definition of harmonicity for functions on open sets in $\mathbb{R}^d$.

\begin{definition}
Let $D\subset \mathbb{R}^d$ be a nonempty open set. A function $f:D\to \mathbb{R}$ is \emph{harmonic in $D$} if $f\in C^2(D)$ and $\Delta_{\mathbb{R}^d}f(\vec{x})=0$ for all $\vec{x}\in D$.
\end{definition}

Recall that classical characterizations for harmonicity in Euclidean domains include 
the Weyl's lemma, and the ball and sphere averaging properties, 
see 
\cite[Sections 1.I.3 and 1.I.3]{Doob84}, 
\cite[Section II.1, Definition 1.1 and Propositions 1.2 and 1.3]{Bass95}.
 
\begin{proposition}\label{prop-harm}
Let $D\subset \mathbb{R}^d$ be a nonempty open set. 

\begin{enumerate}
\item If a function $f:D\to \mathbb{R}$ is harmonic in $D$, then $f\in C^\infty(D)$, it satisfies \eqref{E:Weylslemma}, and it is real analytic.
\item\label{i-shereaveraging}\mbox{[Sphere averaging property]}
A locally bounded Borel measurable function $f:D\to \mathbb{R}$ is harmonic in $D$ if and only if for any open ball $B_r(\vec{x}\,)$ with center $\vec{x}\in D$ and radius $r>0$ such that $\overline{B_r(\vec{x})}\subset D$ we have 
\begin{equation}\label{E:sphereaverage}
f(\vec{x}\,)=\frac{1}{\omega_d\:r^{d-1}}\int_{\partial B_r(\vec{x}\,)}f(\vec{y}\,)\sigma(d\vec{y}\,),
\end{equation}
where $\sigma$ denotes the surface measure on $\partial B_r(\vec{x}\,)$ and $\omega_{d}$ denotes the surface area of the unit sphere in $\mathbb{R}^d$. 

 \item\label{i-ballaveraging}\mbox{[Ball averaging property]}
 A Borel measurable function $f:D\to \mathbb{R}$ is harmonic in $D$ if and only if it is locally integrable on $D$ and for any open ball $B_r(\vec{x}\,)$ with center $\vec{x}\in D$ and radius $r>0$ such that $\overline{B_r(\vec{x})}\subset D$ we have 
 \begin{equation}\label{E:ballaverage}
 f(\vec{x}\,)=\frac{1}{v_d\:r^{d}}\int_{B_r(\vec{x}\,)}f(\vec{y}\,)d\vec{y},
 \end{equation}
 where $v_d$ denotes the volume of the unit ball in $\mathbb{R}^d$.
 
\item\label{i-Weylslemma}\mbox{[Weyl's lemma]} 
If a  Borel measurable locally integrable function $f:D\to \mathbb{R}$ satisfies 
\begin{equation}\label{E:Weylslemma}
\int_{D}f(\vec{y}\,)\Delta \varphi (\vy) d\vec{y}=0 \text{ for all }\varphi\in C^\infty_0(D) 
\end{equation}
then there is a harmonic function $\tilde f:D\to \mathbb{R}$ such that $f(x)=\tilde f(x)$ for $\lambda^d$-a.e.\ $x\in D$. 
In other words, a Borel measurable locally integrable function $f:D\to \mathbb{R}$ has a harmonic $\lambda^d$-version if and only if $\Delta f =0$ in $D$ in the distributional sense. 
\end{enumerate}
\end{proposition}

Note that items (\ref{i-shereaveraging}) and (\ref{i-ballaveraging}) in Proposition~\ref{prop-harm} 
also have $\lambda^d$-a.e.\ versions similar to item (\ref{i-Weylslemma}), but we omit them for the sake of brevity. 

\begin{lemma}\label{lem-ave}
Let $O\subset \Omega$ be open and let $f:\Omega\to\mathbb{R}$ be a measurable function. 

\begin{enumerate}
\item If $f$ is harmonic in $O$ then we have $f\in C^\infty_{\orb}(O)$ and $\Delta f(\Lambda)=0$ for any $\Lambda\in O$. 
\item 
If for each $\Lambda\in O$ the function $h_\Lambda^\ast f$ is locally bounded on an open neighborhood of $\vec{0}$, then $f$ is harmonic in $O$ if and only if for each $\Lambda\in O$ the function $h_\Lambda^\ast f$ satisfies the sphere averaging property (\ref{E:sphereaverage}) in an open neighborhood of $\vec{0}$.

\item The function $f$ is harmonic in $O$ if and only if for each $\Lambda\in O$ the function 
$h_\Lambda^\ast f$ is locally integrable on an open neighborhood of $\vec{0}$ and satisfies the ball averaging property (\ref{E:ballaverage}) there. 

\item  \label{i-Weylslemma-OO}
If a Borel measurable locally integrable function $f:O\to \mathbb{R}$ satisfies 
\begin{equation}\label{E:Weylslemma-OO}
\int_{O}f(\Lambda)\Delta \varphi (\Lambda) d\mu(\Lambda)=0 \text{ for all }\varphi\in C^\infty_0(O) 
\end{equation}
then there is a harmonic function $\tilde f:O\to \mathbb{R}$ such that $f(\Lambda)=\tilde f(\Lambda)$ for $\mu$-a.e.\ $\Lambda\in O$. In other words, a measurable locally integrable function $f:O\to \mathbb{R}$ has a harmonic $\mu$-version if and only if $\Delta f =0$ in $O$ in the distributional sense. 
\end{enumerate}\end{lemma}

\begin{proof}
The first three statements of the lemma are immediate consequences of Definition \ref{def-harm} and Proposition \ref{prop-harm}.

The last assertion follows from  Section~\ref{S:S} and Proposition~\ref{prop-harm}(\ref{i-Weylslemma}) as follows.  
Without loss of generality we can assume that $O=\OLe$.
	 Suppose, for a moment,  that $f$ is  distributionally harmonic in $O$ but does not have a harmonic $\mu$-version. Then there is a set of orbits 
	 of positive $\nu$ measure and a countable set of smooth test functions 
	 compactly supported in $B$ such that on this set 
	 $h_{\Lambda'}^\ast f$ is not \Rd-orthogonal to the \Rd-Laplacian of one of these test functions. Then there is a possibly smaller   
	 set of orbits 
	 of positive $\nu$ measure and $\phi\in C^\infty_0(B)$ such that  
	 $h_{\Lambda'}^\ast f$ is not \Rd-orthogonal to $\Delta_\Rd\phi$. This contradicts the assumption that $f$ is  $\mu$-distributionally harmonic and proves the lemma. 
\end{proof}

Another well known definition of harmonicity for functions on open subsets of $\mathbb{R}^d$ is the probabilistic one,
it goes back to \cite{Doob54} and is based on a probabilistic interpretation of (\ref{E:sphereaverage}) by \cite{Kakutani44, Kakutani45}. We write $\vec{W}^{\vec{x}}$ to denote the Brownian motion on $\mathbb{R}^d$ started at $\vec{x}\in\mathbb{R}$. For any open set $D\subset \mathbb{R}^d$ and any $\vec{x}\in\mathbb{R}^d$ let $\tau_D^{\vec{x}}:=\inf \{ t\geq 0: \vec{W}_t^{\vec{x}} \in \mathbb{R}^d\setminus D \}$ denote the first exit time of $\vec{W}^{\vec{x}}$ from $D$.
\begin{definition}
Let $D\subset \mathbb{R}^d$ be a nonempty open set. A Borel measurable function $f:D\to\mathbb{R}$ is said to be \emph{harmonic in $D$ in the probabilistic sense} if for any bounded open set $D'\subset D$ with $\overline{D'}\subset D$ and any $\vec{x}\in D$ the family $(f(\vec{W}_{t\wedge \tau_{D'}^{\vec{x}}}^{\vec{x}})_{t\geq 0}$ is a $\mathbb{P}$-martingale.
\end{definition}

The following fact was first proved in \cite{Doob54} (see also \cite[Sections 2.IX.6 and 2.IX.8]{Doob84} or \cite[Theorem 13.9]{Dynkin}). For more recent expositions, see \cite[Proposition II.1.3 and Proposition II.1.5 and its proof]{Bass95} or \cite[Definition 9.2.2 and Lemma 9.2.3]{Oks00}.

\begin{proposition}
Let $D\subset \mathbb{R}^d$ be open and let $f:D\to\mathbb{R}$ be Borel measurable and locally bounded. Then $f$ is harmonic in $D$ if and only if $f$ is harmonic in $D$ in the probabilistic sense. 
\end{proposition} 

A probabilistic definition of harmonicity for more general continuous Markov processes has been given in \cite[Section XII.5, in particular 12.18 and Theorem 12.12]{Dynkin}. A probabilistic definition of harmonicity for functions on open subsets $O\subset \Omega$ arises as a special case. For any open set $O\subset \Omega$ and any point $\Lambda\in \Omega$ let $\tau_O^\Lambda:=\inf\left\lbrace t\geq 0: X_t^\Lambda \in \Omega\setminus O\right\rbrace$ be the first exit time of $X^\Lambda$ from $O$. 

\begin{definition}\label{D:prob-harm}
We say that a measurable function $f:\Omega\to\mathbb{R}$ is \emph{harmonic on $O$ in the probabilistic sense} if for any open set $O'\subset O$ with $\overline{O'}\subset O$
and every $\Lambda\in O$ the family $(f(X^\Lambda_{t\wedge \tau_{O'}^\Lambda}))_{t\geq 0}$ is a $\mathbb{P}$-martingale. 
\end{definition}

\begin{proposition}\label{prop-harmonic}
Let $O\subset \Omega$ be open and let $f:\Omega\to\mathbb{R}$ be a measurable function that is locally bounded on $O$. Then $f$ is harmonic in $O$ if and only if it is harmonic in $O$ in the probabilistic sense. 
\end{proposition}\begin{proof}
The diffusion moves in a single orbit, and this orbit is isometric to \Rd, as under the orbit homeomorphism the orbit metric is the Euclidean \Rd.
\end{proof}
 
The same arguments that were used to show the absence of the strong Feller property in Subsection~\ref{subsec-no-strong-Feller} also show the following. 
\begin{corollary}
Let $O\subset \Omega$ be open. There exist bounded and measurable functions $f:\Omega\to\mathbb{R}$ that are harmonic in $O$ but not continuous in $O$, $f\notin C(O)$. 
\end{corollary}

\subsection{Finite energy harmonic functions} We discuss the connection between harmonic functions in the sense of Definition \ref{def-harm} and finite energy $L^2(\Omega,\mu)$ harmonic functions (see \cite{FOT94,ChFu12}). 

Let $O$ be an open subset of $(\Omega,\varrho)$. We say that a measurable function $f:\Omega\to\mathbb{R}$ is \emph{$\mathcal{E}$-harmonic in $O$} if $f\in\mathcal{D}(\mathcal{E})=\mathcal{W}^{1,2}(\Omega,\mu)$
and $\mathcal{E}(f,g)=0$ for all $g\in\mathcal{D}(\mathcal{E})\cap C_c(O)$.

Note that in the context of our paper it is enough to consider $g\in C^1_{tlc}(\Omega)\cap\cap C_c(O)$.

It is known that the notion of harmonicity can be localized, and in our case the following lemma can be easily obtained from a partition of unity as in Lemmas \ref{L:partofunity} and \ref{L:Sobolev_Charact} together with Lemma~\ref{lem-ave}.
In our situation, the $\mathcal{E}$-harmonicity in the above sense is connected to harmonicity on subsets of $\mathbb{R}^d$. Given an open subset $D\subset \mathbb{R}^d$, we say that a measurable function $f:D\to\mathbb{R}$ \emph{is $\mathcal{E}_{\mathbb{R}^d}$-harmonic in $D$} if 
$f\in L^2(O,\mu)$, $\mathcal{E}_{\mathbb{R}^d}^D(f)<+\infty$ and $\mathcal{E}_{\mathbb{R}^d}^D(f,g)=0$ for all $g\in C_0^\infty(D)$.

The second part of the Lemma can be seen using Lemmas  \ref{lem-w22}   and  
\ref{L:productf0F0}. 

\begin{lemma} \label{L:relatetoEuclid}
Let $O$ be an open subset of $\Omega$ and $O_1,...,O_n$ be a finite open cover of $O$ with open subsets $O_i\subset O$. A measurable function $f$ is $\mathcal{E}$-harmonic in $O$ if and only if it is $\mathcal{E}$-harmonic in $O_i$ for all $i$.

Suppose $O=\phi^{-1}(\mathcal{C}\times B)$ is an open subset of $\Omega$ of form (\ref{E:transcylsets}). Then, a measurable function $f\in\mathcal{D}(\mathcal{E})=\mathcal{W}^{1,2}(\Omega,\mu)$ is $\mathcal{E}$-harmonic in $O$ if and only if for $\nu_{\mathcal{C}}$-a.e. $\Lambda'\in\mathcal{C}$, the function $h_{\Lambda'}^\ast f$ is $\mathcal{E}_{\mathbb{R}^d}$-harmonic in $B$. In this case, $f$ can be modified on a set of measure zero so that $h_{\Lambda'}^\ast f\in C^\infty(B)$ for $\nu_{\mathcal{C}}$-a.e. $\Lambda'\in\mathcal{C}$.
\end{lemma}

The notion of $\mathcal{E}$-harmonicity can be characterized by a version of the probabilistic definition of harmonicity, modified for Dirichlet forms (compare to
Definition \ref{D:prob-harm}, Proposition \ref{prop-harmonic}). A bounded measurable function $f$ on $\Omega$ is $\mathcal{E}$-harmonic in an open set $O\subset \Omega$ if and only if $(f(X^\Lambda_{t\wedge \tau_{O'}}))_{t\geq 0}$ is a uniformly integrable $\mathbb{P}$-martingale for $\mathcal{E}$-quasi every $\Lambda\in\Omega$ whenever $O'\subset O$ is a relatively compact open subset of $O$. This has been stated and proved in \cite[Proposition 2.5]{BBKT10} and \cite[Theorem 2.11]{Chen09}).

There is also a Liouville theorem for $\mathcal{E}$-harmonic functions.
 
\begin{corollary}\label{C:Fukushima}
 If a measurable function $f:\Omega\to\mathbb{R}$ is $\mathcal{E}$-harmonic in $\Omega$, then $f$ is constant $\mu$-a.e.
\end{corollary}

\begin{proof} This follows from Theorem~\ref{T:main_Liouville} combined with Lemma  \ref{lem-ave} (4). 
\end{proof}

\subsection{Irreducibility}

From the preceding discussion we can infer the irreducibility of the Dirichlet form $(\mathcal{E},\mathcal{D}(\mathcal{E}))$ stated in Theorem \ref{T:main_L2sg}. Recall that $(P_t)_{t>0}$ denotes the symmetric Markovian semigroup on $L^2(\Omega,\mu)$ associated with $(\mathcal{E},\mathcal{D}(\mathcal{E}))$. We follow the terminology of \cite[Section 1.6]{FOT94} and call a Borel subset $A$ of $(\Omega,\varrho)$ \emph{$(P_t)_{t> 0}$-invariant} if $P_t(\mathbbm{1}_Af)=\mathbbm{1}_AP_t f$ for any $f\in L^2(\Omega,\mu)$ and $t>0$. The Markov semigroup $(P_t)_{t> 0}$ or the Dirichlet form $(\mathcal{E},\mathcal{D}(\mathcal{E}))$ is called \emph{irreducible} if any $(P_t)_{t\geq 0}$-invariant subset $A$ satisfies either $\mu(A)=0$ or $\mu(\Omega\setminus A)=0$. The irreducibility of $(\mathcal{E},\mathcal{D}(\mathcal{E}))$ can be verified as follows.

\begin{proof}
Let $A$ be a $(P_t)_{t>0}$-invariant Borel subset of $\Omega$. If $\mu(A)=0$ there is nothing to prove, so we may assume $\mu(A)>0$. By invariance we have $P_t\mathbbm{1}_A=\mathbbm{1}_A$ $\mu$-a.e. so that $\mathbbm{1}_A\in\mathcal{D}(\mathcal{E})$ by (\ref{E:Dformdomain}) and according to (\ref{E:Dform}), $\mathcal{E}(\mathbbm{1}_A)=0$. Cauchy-Schwarz now implies that $\mathbbm{1}_A$ is $\mathcal{E}$-harmonic in $\Omega$, and since it is bounded, constant $\mu$-a.e. on $\Omega$ by Corollary \ref{C:Fukushima}. Since $\int_\Omega \mathbbm{1}_A\:d\mu=\mu(A)>0$ we must have $\mathbbm{1}_A=1$ $\mu$-a.e. on $\Omega$, which means $\mu(A)=1$.
\end{proof}

\begin{remark}
Alternatively, one can prove the irreducibility of $(\mathcal{E},\mathcal{D}(\mathcal{E}))$ more directly, without using Corollary \ref{C:Fukushima}. This can be done using similar arguments as in Lemma \ref{L:W22loc} and Lemma \ref{L:atmostone} below. If one proceeds this way, then Corollary \ref{C:Fukushima} becomes a simple consequence of  \cite[Theorem 1]{Kajino17}, together with irreducibility and recurrence. Also, a known probabilistic argument could be applied: By \cite[Lemma 6.7.3]{ChFu12} we have $f(X_t^\Lambda)=f(\Lambda)$ for all $t\geq 0$ $\mathbb{P}$-a.s. for $\mathcal{E}$-quasi every $\Lambda\in \Omega$, which implies that for any $c\in\mathbb{R}$ the level set $\left\lbrace \Lambda\in\Omega: f(\Lambda)=c\right\rbrace$ is $(P_t)_{t>0}$-invariant. By irreducibility, $\mu(\left\lbrace \Lambda\in\Omega: f(\Lambda)=c\right\rbrace) = 0 \text{ or } 1$ for any $c\in\mathbb{R}$, which is possible only if $f$ is constant $\mu$-a.e. A more general result is stated in \cite[Proposition 1.1.(i)]{Fukushima16}.
\end{remark}

\section{Proof of Theorem \ref{T:main_Hodge}: 
 Helmholtz-Hodge decomposition for one-dimensional patterns}\label{S:Hodge}

In this section we discuss $L^2$-vector fields associated with the Dirichlet form $(\mathcal{E},\mathcal{D}(\mathcal{E}))$. Motivated by earlier results in \cite{HT15}, Theorem~\ref{T:main_Hodge} provides a Hodge decomposition for the one-dimensional case. 

The proof Theorem~\ref{T:main_Hodge} requires the combination of several concepts. First, recall from the proof of Lemma \ref{L:Dform} that the gradient
\[
\nabla\colon 
W^{1,2}(\Omega,\mu)\subset L^2(\Omega,\mu)\to L^2(\Omega,\mu)
\] 
is a densely defined, closed, unbounded operator with domain $\D(\E)=W^{1,2}(\Omega,\mu)$. The image $\im \nabla$ of $\nabla$ is a closed subspace of $L^2(\Omega,\mu)$, see e.g. \cite[Section 4]{HKT15}. Therefore, we have the orthogonal decomposition 
\eqref{E:Hodgedecomp}
where $(\im \nabla)^\bot$ denotes the orthogonal complement of $\im \nabla$. By ~\cite[Theorem 2.1]{HRT13}, the space $\mcH$ of $L^2$-differential $1$-forms associated with $(\E,\D(\E))$ can be seen as a direct integral over a measurable field of Hilbert spaces $\mcH_\Lambda$, which are isometrically isomorphic to $\mbbR^d$ via 
\begin{equation*}
\iota_\Lambda\colon\mcH_\Lambda\to\mbbR^d,
\end{equation*}
c.f.~\cite[Remark 2.6 (iii)]{HRT13} and \cite[Proposition 4.2]{BK17}. Consequently, there is an isometric isomorphism 
\begin{equation*}
\iota\colon L^2(\Omega, \mathbb{R}^d, \mu)\to\mcH,
\end{equation*}
where $\mcH$ denotes the space of $L^2$-differential $1$-forms associated with $(\E,\D(\E))$. In the one-dimensional case, the Hodge-star operator 
\begin{equation*}
\star_{\omega_0}\colon L^2(\Omega,\mu)\to\mcH,
\end{equation*}
c.f.\ Proposition~\ref{prop:Hi1.01}, where $\omega_0=\iota_\Lambda^{-1}(1)$ for a fixed $\Lambda\in\Omega$, gives an isometric isomorphism, see~\cite[Proposition 4.5]{BK17}. Therefore, defining
\begin{equation*}
\star\colon L^2(\Omega,\mu)\xrightarrow{\iota^{-1}}\mcH\xrightarrow{\star_{\omega_0}^{-1}}L^2(\Omega,\mu)
\end{equation*}
we obtain an isometric isomorphism of $L^2(\Omega,\mu)$ onto itself. In particular, by (\ref{E:Hodgeandaction}) we have $\star \nabla f=f'$ for any $f\in\mathcal{D}(\mathcal{E})$. Finally, the following Lemma~\ref{L:atmostone} yields
\begin{equation*}
(\im\nabla)^\bot\simeq\star(\im\nabla)^\bot=\{\text{constant functions}\}\simeq\mbbR.
\end{equation*}

To prove Theorem \ref{T:main_Hodge} we take a closer look at the structure of $\ker d^\ast$ by following the arguments used to establish Theorem \ref{T:main_Liouville}.

\begin{lemma}\label{L:atmostone}
Assume $d=1$ and $v\in (\im \nabla)^\bot$. Then, there is a Borel set $\mathcal{O}\subset \OO$ of $\mu$-measure zero, such that for any $\Lambda\in \mathcal{O}^c$, the function $\star v$ is $d\vt$-a.e. constant on $\orb(\Lambda)$. As a consequence, $\star v$ is constant $\mu$-a.e.\ on $\OO$.
\end{lemma}


\begin{proof}
We have $\left\langle \star v, f'\right\rangle_{L^2(\Omega,\mu)}=\left\langle \star v, \star \nabla f\right\rangle_{L^2(\Omega,\mu)}=\left\langle v, \nabla f\right\rangle_{L^2(\Omega,\mu)}=0$
for all $f\in\mathcal{D}(\mathcal{E})$. If $O=\phi^{-1}(\mathcal{C}\times (-\varepsilon,\varepsilon))$ is a neighborhood of type (\ref{E:transcylsets}) then for all $f\in (\phi^{-1})^\ast (L^2(\mathcal{C},\nu_\mathcal{C})\otimes C_c^1(-\varepsilon,\varepsilon))$ of product form (\ref{e-prod}) we observe 
\[ \int_{\mathcal{C}}\int_{(-\varepsilon,\varepsilon)} h_{\Lambda'}^\ast(\star v)(\vt\,)f_0(\Lambda')F_0'(\vt\,)\:d\vt d\nu_{\mathcal{C}}(\Lambda')=0\]
by (\ref{e-der2}) and therefore
\[\int_{(-\varepsilon,\varepsilon)} (h_{\Lambda'}^\ast(\star v))'(\vt\,)F_0(\vt\,)\:d\vt=-\int_{(-\varepsilon,\varepsilon)} h_{\Lambda'}^\ast(\star v)(\vt\,)F_0'(\vt\,)\:d\vt=0\]
for $\nu_{\mathcal{C}}$-a.e. $\Lambda'\in\mathcal{C}$. Testing with $F_0$ from a countable dense subspace of $C_c^1(-\varepsilon,\varepsilon)$ we can see that for $\nu_{\mathcal{C}}$-a.e. $\Lambda'\in\mathcal{C}$ the function $h_{\Lambda'}^\ast(\star v)$ must be $d\vt$-a.e. constant on $(-\varepsilon,\varepsilon)$.

Now, let $O_1,...,O_n$ be open sets $\phi(O_i)=\mathcal{C}_i\times (-\varepsilon_i,\varepsilon_i)$ of type (\ref{E:transcylsets}) covering $\OO$. By the preceding there exists a Borel set $\mathcal{N}_i\subset \mathcal{C}_i$ such that for any $\Lambda'\in \mathcal{C}_i\setminus \mathcal{N}_i$ the function $h_{\Lambda'}^\ast(\star v)$ is $d\vt$-a.e. constant on $(-\varepsilon_i,\varepsilon_i)$. Similarly as in the proof of Lemma \ref{L:W22loc}, 
let $\mathcal{O}:=\bigcup_{i=1}^N \bigcup_{\vt\in\mathbb{R}^d}\varphi_{\vt}\,(\mathcal{N}_i)$ denote the union of orbits that hit some $\mathcal{N}_i$. As discussed there, $\mathcal{O}$ is measurable and we have $\mu(\mathcal{O})=0$. If $\Lambda\in\OO$ is such that its $\orb(\Lambda)$ is not contained in $\mathcal{O}$, then $\star v$ is $d\vt$-a.e. constant on each $O_i\cap \orb(\Lambda)$, so that by Lemma
\ref{L:relatetoEuclid}, 
$\star v$ is $d\vt$-a.e. constant on $\orb(\Lambda)$. The second statement is a consequence of the ergodicity of $\mu$.
\end{proof}

\begin{remark} 
For an arbitrary dimension $d$, let $L^2(\OO,\mathbb{R}^d,\mu)$ denote the space of $L^2$-vector fields on $\OO$. For $v\in L^2(\OO,\mathbb{R}^d,\mu)$ we have $v=\sum_{i=1}^d v_i \vec{e}_i$ with suitable $v_i\in L^2(\Omega_ {\Lambda_0},\mu)$, where $\ve_1,...\ve_d$ denote the standard unit vectors in $\mathbb{R}^d$. Given $f\in b\mathcal{B}(\OO)$, we define a vector field $fv\in L^2(\OO,\mathbb{R}^d,\mu)$ by $fv:=\sum_{i=1}^d (fv_i)\vec{e}_i$. Again, the gradient $\nabla$ is considered as a densely defined closed unbounded operator $\nabla:L^2(\OO,\mu)\to L^2(\OO,\mathbb{R}^d,\mu)$ with domain $\mathcal{D}(\mathcal{E})$.

The space $C^\infty_{tlc}(\OO,\mathbb{R}^d)$ of smooth vector fields on $\OO$ is dense in $L^2(\Omega,\mathbb{R}^d,\mu)$. Any element $v$ of $C^\infty_{tlc}(\OO,\mathbb{R}^d)$ can be written in the form $v=\sum_{i=1}^d v_i \vec{e}_i$ with suitable functions $v_i\in C^\infty_{tlc}(\OO)$, and given a function $f\in C^\infty_{tlc}(\OO)$, we have $fv \in C^\infty_{tlc}(\OO,\mathbb{R}^d)$. The gradient operator satisfies the Leibniz rule $\nabla(fg)=f\nabla g+g\nabla f$, $f,g\in C^\infty_{tlc}(\OO)$. By duality, a smooth vector field $v=\sum_{i=1}^d v_i \vec{e}_i$ may be identified with a smooth differential $1$-form $\sum_{i=1}^d v_i dx^i$. Interpreting $h_{\Lambda_0}^\ast$ from~\eqref{E:algiso} as a map $\sum_i f_idx^i\mapsto \sum_i h_{\Lambda_0}^\ast f_i dx^i$, it provides a bijection of $C_{tlc}^\infty(\OO,\mathbb{R}^d)$ onto the space $\Delta_{\Lambda_0}^1$ of smooth $\Lambda_0$-equivariant $1$-forms on $\mathbb{R}^d$, i.e. differential $1$-forms on $\mathbb{R}^d$ with coefficients in $C_{\Lambda_0}^\infty(\mathbb{R}^d)$.
See for instance \cite{Kellendonk:PEC} or \cite[Section 3]{ST:SA} and the references cited there.
\end{remark}

The space $L^2(\OO,\mathbb{R}^d,\mu)$ is isometrically isomorphic to the space $\mathcal{H}$ of $L^2$-differential $1$-forms asociated with $(\mathcal{E},\mathcal{D}(\mathcal{E}))$ as in Definition \ref{D:Habstract}: The space $\mathcal{H}$ can be written as the direct integral over a measurable field of Hilbert spaces $\mathcal{H}_{\Lambda}$, see the Appendix for details. Since the pointwise index of $(\mathcal{E},\mathcal{D}(\mathcal{E}))$ is $d$ $\mu$-a.e. we know that for $\mu$-a.e. $\Lambda\in \Omega$ the fiber $\mathcal{H}_{\Lambda}$ of $\mathcal{H}$ is a vector space of dimension $d$ so that there exists an isometric isomorphism $\iota_\Lambda:\mathcal{H}_{\Lambda}\to \mathbb{R}^d$. This implies the existence of an isometric isomorphism $\iota: \mathcal{H} \to L^2(\OO,\mathbb{R}^d,\mu)$. The gradient operator $\nabla$ is related to the abstract derivation $\partial$ in Definition \ref{D:partial} by $\nabla=\iota\circ \partial$.

\begin{remark}
We conjecture that also for pattern spaces originating from Delone sets in $\mathbb{R}^d$ with arbitrary $d\geq 1$ a Hodge type decomposition can be proved that generalizes (\ref{E:Hodgedecomp}) and a meaning can be given to the summands involved.
\end{remark}

\begin{remark} An alternative proof of Theorem \ref{T:main_Hodge} can be given using the notion of local harmonicity from \cite[Definition 4.1]{HT15}. One can show that \cite[Theorem 4.2]{HT15} applies in the present situation and combine it with Theorem \ref{T:main_Liouville} to verify Theorem~\ref{T:main_Hodge}.
\end{remark}

\appendix

\section{Derivatives, smoothness and partitions of unity} \label{A:S}

We consider canonical differential operators on $\Omega$.

\begin{definition}\label{D:Ck} Let $O\subset \Omega$ be open.
A function $f:O\to\mathbb{R}$ is called \emph{differentiable} in $O$ if for any $\Lambda\in O$ and $\vt\in\mathbb{R}^d$ the limit
\begin{equation}\label{E:diff}
\frac{\partial f}{\partial\vt}(\Lambda)=\lim_{s\to 0}\frac{f\circ \varphi_{s\vt\,}(\Lambda)-f(\Lambda)}{s}
\end{equation}
exists. As usual we say $f$ is \emph{$k$-times differentiable in $O$} if the (k-1)-th derivative of $f$ in the sense of (\ref{E:diff}) exists and is differentiable in $O$. 
\end{definition}

Note that for any function $f$ which is differentiable on all of $\Omega$ and any $\vv\in\mathbb{R}^d$ we have 
\[
\frac{\partial}{\partial\vt}\:(f\circ \varphi_{\vec{v}})=\frac{\partial f}{\partial\vt}\circ \varphi_{\vec{v}}.
\] 

Let $O\subset \Omega$ be open. For a function $f$ differentiable in $O$ we define the gradient $\nabla f$ of $f$ as the vector field
\begin{equation}\label{E:gradient}
\nabla g=\sum_{i=1}^d \frac{\partial g}{\partial\ve_i}\:\vec{e}_i,
\end{equation} 
where $\ve_1,...,\ve_d$ are the standard unit base vectors in $\mathbb{R}^d$. For a function $f$ twice differentiable in $O$ we define the Laplacian $\Delta f$ of $f$ by 
\begin{equation}\label{E:Laplace}
\Delta f:=\sum_{i=1}^d \frac{\partial^2 f}{\partial \ve_i\:^2},
\end{equation}
where $\frac{\partial^2 f}{\partial \ve_i\:^2}:=\frac{\partial}{\partial \ve_i}\left(\frac{\partial f}{\partial \ve_i}\right)$. If $f$ is $k$-times differentiable in $O$ and $\alpha=(\alpha_1,...,\alpha_d)$ is a multiindex with $|\alpha|\leq k$ then we can similarly introduce general mixed derivatives 
\begin{equation}\label{E:Dalpha}
D^\alpha f=\frac{\partial^{|\alpha|}f}{\partial \ve_1\:^{\alpha_1}\cdots \partial \ve_d\:^{\alpha_d}}.
\end{equation}

If $O=\phi^{-1}(\mathcal{C}\times B)$ is an open set of type (\ref{E:transcylsets}) and $f$ is $k$-times differentiable in $O$, then, using the notation (\ref{E:fO}), we have
\begin{equation}\label{E:DalphafO}
D^\alpha f(\overline{\Lambda}) = D^\alpha_{\mathbb{R}^d} f\cphi (\Lambda',\vt\,), \quad \overline{\Lambda}=\varphi_{\vt\,}(\Lambda')\in O,
\end{equation}
provided $|\alpha|\leq k$. The differential operators $D_{\mathbb{R}^d}^\alpha$ on the right hand side are considered in the usual Euclidean sense with respect to $\vt$. Using (\ref{E:DalphafO}) we can also verify the locality of the operators $D^\alpha$. 
\begin{corollary}\label{C:local}
The operators $D^\alpha$ are local, i.e. if $f,g$ are $k$-times differentiable and satisfy $f=g$ in an open set $O$ then for any 
$\alpha$ with $|\alpha|\leq k$ and any $\Lambda\in O$ we have $D^\alpha f(\Lambda)=D^\alpha g(\Lambda)$. 
\end{corollary}

If $O=\phi^{-1}(\mathcal{C}\times B)$ is an open set of type (\ref{E:transcylsets}) and, similarly as in (\ref{e-prod}), $f\in b\mathcal{B}(O)$ is a function defined as the product $f(\bar\Lambda):=f_0(\Lambda')F_0(\vt\,)$ of some $f_0\in b\mathcal{B}(O)$ and $F_0\in C^k(B)$, whose supports are contained in $\mathcal{C}$ and $B$, respectively, then
\begin{equation}\label{e-der2}
D^\alpha f(\bar\Lambda)=f_0(\Lambda')D_{\mathbb{R}^d}^\alpha F_0(\vt\,),\quad \bar\Lambda=\varphi_{\vt\,}(\Lambda')\in O, 
\end{equation}
$|\alpha|\leq k$, reducing derivatives on the fractal-like space $\Omega$ to ordinary calculus derivatives.

\begin{notation}
We write $C^k(O)$ ($k=1,2,...$ or $k=\infty$) for the space of $k$-times continuously differentiable functions, that is, the space of continuous functions $f$ on $O$ whose derivatives up to the order $k$ in the sense of (\ref{E:diff}) are all continuous functions on $O$.
\end{notation}

We endow the space $C^k(O)$ with the norm $$\left\|f\right\|_{C^k}:=\sup_{\Lambda\in O}|f(\Lambda)| +\sum_{0<|\alpha|\leq k}\sup_{\Lambda\in O} |D^\alpha f(\Lambda)|.$$

We recall the definition of transversally locally constant functions and collect some known results for later use.
 
\begin{definition}\label{D:tlc}
A function $f\in C(\Omega)$ is called \emph{transversally locally constant} if for any $\Lambda\in\Omega$ there exists $\varepsilon>0$ such that
\[
B_{\frac{1}{\varepsilon}}(\vec{0})\cap\Lambda' = B_{\frac{1}{\varepsilon}}(\vec{0})\cap\Lambda\quad\text{ implies }\quad f(\Lambda') = f(\Lambda),
\]
i.e. if for any $\Lambda\in\Omega$ there exists $\varepsilon>0$ such that $f$ is constant on $\CLe$.
We write $C_{tlc}(\Omega)$ to denote the space of transversally locally constant functions and $C^k_{tlc}(\Omega)$ to denote the space of functions $f\in C^k(\Omega)$, $k=1,2,...$ or $k=\infty$, that are locally transversally constant.
\end{definition} 

Note that if $f\in C^k_{tlc}(\Omega)$ then also $f\circ \varphi_{\vt\:}\in C^k_{tlc}(\Omega)$ for any fixed $\vt\in\mathbb{R}^d$. Transversally locally constant functions on $\Omega$ are linked to the following specific type of functions on $\mathbb{R}^d$.

\begin{definition}\label{D:eqvar}
A function $F\in C(\mathbb{R}^d)$ is \emph{(strongly) $\Lambda$-equivariant} if there exists $R>0$ such that if for two $\vx,\vy\in\mathbb{R}^d$ we have that 
\[
B_R(\vx)\cap\Lambda = B_R(\vy)\cap\Lambda\quad\text{ implies }\quad f(\vx) = f(\vy).
\] 
We write $C_\Lambda(\mathbb{R}^d)$ for the space of all $\Lambda$-equivariant functions on $\mathbb{R}^d$, and by $C_\Lambda^k(\mathbb{R}^d)$ (for $k=1,2,...$ or $k=\infty$) we denote the spaces of functions $f\in C^k(\mathbb{R}^d)$ that are $\Lambda$-equivariant. 
\end{definition}
See \cite[Definition 2.1]{Kellendonk:PEC}. The space $C_\Lambda^\infty(\mathbb{R}^d)$ is uniformly dense in $C_\Lambda(\mathbb{R}^d)$. Let $\overline{C}_\Lambda(\mathbb{R}^d)$ denote the uniform closure of $C_\Lambda(\mathbb{R}^d)$.

Recall the notation (\ref{E:algiso}). We summarize some well known statements from \cite[Lemma 4.2]{Kellendonk:PEC} and \cite[Proposition 22]{KP:RS}, see also \cite[Theorem 6]{ST:SA}. Because we assume repetitivity and therefore have (\ref{E:minimal}) they read as follows. 

\begin{lemma}\label{L:basicprops}
Let $\Lambda\in\Omega$. The restriction of $h_{\Lambda}^\ast$ to continuous functions defines an isomorphism
of Banach algebras $h_{\Lambda}^\ast:C(\Omega)\to \overline{C}_{\Lambda}(\mathbb{R}^d)$. The algebras $C_{tlc}(\Omega)$ and $C_{tlc}^k(\Omega)$ are uniformly dense in $C(\Omega)$ and their images under $h_{\Lambda}^\ast$ are $C_{\Lambda}(\mathbb{R}^d)$ and $C_{\Lambda}^k(\mathbb{R}^d)$, respectively.
\end{lemma}

We recall a useful localization argument from \cite[Lemma 4.2 and its proof]{Kellendonk:PEC}.

\begin{proposition}\label{P:comb} Let $\Lambda\in \Omega$ and let $\eta\in C(\mathbb{R}^d)$ be a function supported in a ball $B_{\varepsilon}(\vec{0}\,)$. Then there is a function $f\in C_{tlc}(\Omega)$ that is supported in $O_{\Lambda,\varepsilon}$, and satisfies $h_{\Lambda'}^\ast f(\vt\,)=\eta(\vt\,)$ for all $\Lambda'\in\mathcal{C}_{\Lambda,\varepsilon}$ and $\vt\in B_{\varepsilon}(\vec{0})$. Moreover, if $\eta \in C^k(\mathbb{R}^d)$ then $f\in C_{tlc}^k(\Omega)$.
\end{proposition}

\begin{proof} In our notation this proposition 
 immediately follows from \eqref{e-prod}, and in some sense is equivalent to the local product structure of pattern spaces in subsection \ref{subsec:LocalStru}.
 For the convenience of the reader we provide the part of the proof essentially taken from \cite[Lemma 4.2]{Kellendonk:PEC}.
 Let $\varepsilon<\frac{1}{\varepsilon_0}\wedge \frac{\varepsilon_0}{2}$ , where $\varepsilon_0$ is the minimal Euclidean distance between points in $\Lambda$. Then the distance of two different points in $$S_{\Lambda,\varepsilon}:=\left\lbrace \vs\in\mathbb{R}^d: B_{\frac{1}{\varepsilon}}(\vec{0})\cap \Lambda\subset \varphi_{\vec{s}}\,(\Lambda)\right\rbrace$$ is at least $\varepsilon_0$. Let $\delta_{\Lambda,\varepsilon}$ be a Dirac comb on $S_{\Lambda,\varepsilon}$. Then for any $\vs\in S_{\Lambda,\varepsilon}$ the function $\delta_{\Lambda,\varepsilon}\ast \eta(\vs+\cdot)$ is continuous on $\mathbb{R}^d$ and its restriction to $B_{\varepsilon_0/2}(\vec{0}\,)$ is supported in $B_{\varepsilon}(\vec{0}\,)$ and equals $\eta$ on $ B_{\varepsilon}(\vec{0}\,)$. The function $\delta_{\Lambda,\varepsilon}\ast \eta$ is $\Lambda$-equivariant: Suppose $\vx_1, \vx_2\in\mathbb{R}^d$. Then we can can write $\vx_i=\vs_i+\vt_i$, where $\vs_i$ is a point in $S_{\Lambda,\varepsilon}$ minimizing the distance between $\vx_i$ and $S_{\Lambda,\varepsilon}$. Suppose that the sets $B_{\frac{8}{\varepsilon}}(\vx_i)\cap \Lambda$ agree. Then, since $|\vt\,|<\varepsilon_0<\frac{1}{\varepsilon}$, both sets $B_{\frac{4}{\varepsilon}}(\vs_i)\cap \Lambda$ agree with $B_{\frac{4}{\varepsilon}}(\vec{0}\,)\cap \Lambda$. Consequently we have $B_{\frac{2}{\varepsilon}}(\vec{0}\,)\cap \varphi_{\vt_1}(\Lambda)=B_{\frac{2}{\varepsilon}}(\vec{0}\,)\cap \varphi_{\vt_2}(\Lambda)$, which implies that 
 $\vt_1-\vt_2\in S_{\Lambda,\varepsilon}$ and therefore $\delta_{\Lambda,\varepsilon}\ast \eta(\vx_1)=\delta_{\Lambda,\varepsilon}\ast \eta(\vx_2)$. This shows that $\delta_{\Lambda,\varepsilon}\ast \eta$ is $\Lambda$-equivariant. By Lemma \ref{L:basicprops} the function $f(\overline{\Lambda}):=(\delta_{\Lambda,\varepsilon}\ast \eta)(h_\Lambda^{-1})(\overline{\Lambda}))$, $\overline{\Lambda}=h_\Lambda(\vt\,)$, $\vt\in\mathbb{R}^d$, is in $C_{tlc}(\mathbb{R}^d)$. The stated properties are immediate. 
\end{proof}

Formula \eqref{e-prod} allows to construct smooth bump functions.

\begin{corollary}\label{C:bumps}
Let $O=O_{\Lambda,\varepsilon}$ be an open set of type (\ref{E:transcylsets}) with $\phi_{\Lambda,\varepsilon}(O)=\mathcal{C}_{\Lambda, \varepsilon}\times B_{\varepsilon}(\vec{0}\,)$. Let $K\subset \mathbb{R}^d$ be compact and such that $K\subset B_{\varepsilon}(\vec{0}\,)$. Then, writing $V:=\phi_{\Lambda,\varepsilon}^{-1}(\mathcal{C}_{\Lambda,\varepsilon}\times K)$, we can find a function $\chi\in C^\infty_{tlc}(\Omega)$ such that $0\leq \chi\leq 1$, $\supp\:\chi\subset O$, $\chi\equiv 1$ on $V$, and $\chi\cphi $ depends only on the second argument.
\end{corollary}

We can obtain $C^\infty_{tlc}$-partitions of unity by a variant of the standard construction, this yields Lemma \ref{L:partofunity}.

For transversally locally constant functions we can also use (\ref{E:algiso}) connect the operators $D^\alpha$ to differential operators $D_{\mathbb{R}^d}^\alpha$ on $\mathbb{R}^d$.
\begin{corollary}\label{C:connecttoRd}
Let $\Lambda\in \Omega$. For any multiindex $\alpha$ with $|\alpha|\leq k$ and any $f\in C_{tlc}^k(\Omega)$ we have 
\[h_{\Lambda}^\ast D^\alpha f=D_{\mathbb{R}^d}^\alpha h_{\Lambda}^\ast f.\]
\end{corollary}

\begin{proof}
If $f\in C_{tlc}^k(\OO)$ has compact support contained in an open set $O$ of type (\ref{E:transcylsets}), $h_\Lambda^\ast f$ is in $C^k_{\Lambda}(\mathbb{R}^d)$ and of form $\delta_{\Lambda,\varepsilon}\ast \eta$. Also $D^\alpha f$ is supported in $O$, we have $h_\Lambda^\ast D^\alpha f\in C_\Lambda(\mathbb{R}^d)$ and 
\[h_{\Lambda}^\ast(D^\alpha f)(\vt\,)= \delta_{\Lambda,\varepsilon}\ast D^\alpha_{\mathbb{R}^d} \eta (\vt\,)= D^\alpha_{\mathbb{R}^d}(\delta_{\Lambda,\varepsilon}\ast \eta)(\vt\,)=D^\alpha_{\mathbb{R}^d} h_{\Lambda}^\ast(f)(\vt).\] 
The general case follows using Lemma \ref{L:partofunity}.
\end{proof}

Another useful consequence of Proposition \ref{P:comb} is the following, which can be seen by localizing to functions supported in sets of type (\ref{E:transcylsets}) and using standard mollification in $\mathbb{R}^d$.
\begin{corollary}\label{C:Cldense} 
 For any $l\geq k$ the space $C_{tlc}^l(\Omega)$ is a dense subspace of $C^k_{tlc}(\Omega)$.
\end{corollary}

 The following lemma is used to approximate smooth functions by the \emph{tlc} functions, and is also used to define approximations in Sobolev spaces in Subsection~\ref{subsec:sob-appr}. 
\begin{lemma}\label{L:tlcdense}
The space $C^\infty_{tlc}(\Omega)$ is dense in $C^k(\Omega)$ for any $k\geq 1$.
\end{lemma}
\begin{proof}
By Lemma \ref{L:partofunity} and Corollary \ref{C:Cldense} it suffices to show that any $f\in C^k(\Omega)$ supported in an open set $O=\phi^{-1}(\mathcal{C}\times B)$ can be approximated in the $C^k$-norm by functions from $C^k_{tlc}(\Omega)$. 
We consider the function $f\cphi (\Lambda',\vt\,)=f(\overline{\Lambda})$ on $\mathcal{C}\times B$. 
Since $\mathcal{C}$ is a Cantor set, we can find a finite Borel measure $\nu$ on $\mathcal{C}$ and for every $i\geq 1$ a 
partition $\mathcal{C}_1^{(i)},...,\mathcal{C}_{N_i}^{(i)}$ of $\mathcal{C}$ into sets $\mathcal{C}_l^{(i)}$ of positive $\nu$-measure and such that $\lim_i \max_{l=1,...,N_i} \diam_{\varrho}(\mathcal{C}_l^{(i)})=0$. 
 Note that in most cases we consider 
$\nu=\nu_{\mathcal{C}}$ from~\eqref{E:localprodmeas}, but in this particular lemma $\nu$ 
can be any finite measure with full support.

We define the functions 
\begin{equation}\label{e-fOi}
f^{O,i}(\Lambda',\vt\,):=\sum_{l=1}^{N_i} c_l^{(i)}(\vt\,)\mathbbm{1}_{\mathcal{C}_l^{(i)}}(\Lambda'), \quad (\Lambda',\vt\,)\in \mathcal{C}\times B,
\end{equation} 
where
\[
c_l^{(i)}(\vt\,):=\frac{1}{\nu(\mathcal{C}_l^{(i)})}\int_{\mathcal{C}_l^{(i)}} f\cphi (\Lambda',\vt\,)\:\nu(d\Lambda').
\]
Clearly, $\lim_i f^{O,i}=f\cphi $ uniformly on $\mathcal{C}\times B$. Moreover, for any $|\alpha|\leq k$ we have 
\[
D^\alpha_{\mathbb{R}^d} f^{O,i}(\Lambda',\vt\,)= \sum_{l=1}^{N_i} D^\alpha_{\mathbb{R}^d}c_l^{(i)}(\vt\,)\mathbbm{1}_{\mathcal{C}_l^{(i)}}(\Lambda')
\] 
and by dominated convergence
\[
D^\alpha_{\mathbb{R}^d} c_l^{(i)}(\vt\,) =\frac{1}{\nu(\mathcal{C}_l^{(i)})}\int_{\mathcal{C}_l^{(i)}} D^\alpha_{\mathbb{R}^d}f\cphi (\Lambda',\vt\,)\:\nu(d\Lambda').
\]
Since $D^\alpha_{\mathbb{R}^d}f\cphi $ is uniformly continuous, we also have $\lim_i D^\alpha_{\mathbb{R}^d}f^{O,i}=D^\alpha_{\mathbb{R}^d}f\cphi $ uniformly on $\mathcal{C}\times B$.
\end{proof}
\begin{remark}
In the uniquely ergodic case we can use $\nu_{\mathcal{C}}$ in place of $\nu$.
\end{remark}

The following lemma follows easily from Corollary \ref{C:bumps}. We use the notation $\vt=(t_1,...,t_d)$ for a vector $\vt\in\mathbb{R}^d$.
\begin{lemma}\label{L:directionsforindex}
Assume $O$ and $V$ are sets of type (\ref{E:transcylsets}) such that $O=\phi^{-1}(\mathcal{C}\times B)$ and
$O\subset \subset V$. Then for any $i=1,...,d$ there exists a function $f_i\in C_{tlc}^\infty(\Omega)$ supported in $V$ and such that on $O$ we have $h_{\Lambda'}^\ast f_i=t_i\ve_i$ for all $\Lambda'\in \mathcal{C}$ and $\vt\in B$. Moreover, for all $\overline{\Lambda}\in O$ we have 
\[\frac{\partial f_i}{\partial \vec{e_j}}\,(\overline{\Lambda})=\begin{cases} 1 & \text{if $j=i$}\\ 0 & \text{if $j\neq i$}.\end{cases}\]
\end{lemma}

\section{Hodge star operators for Dirichlet forms}\label{A:H}

We discuss the definition of Hodge star operators in a general setup. Let $X$ be a locally compact separable metric space, $\mu$ a Radon measure on $X$ with full support and $(\mathcal{E},\mathcal{D}(\mathcal{E}))$ a regular Dirichlet form on $L^2(X,\mu)$. There are several articles concerned with $L^2$-differential $1$-forms associated with Dirichlet forms, see for instance \cite{CS03, IRT12, HRT13, HT15, HT15b}. A brief description is as follows.

The space $\mathcal{D}(\mathcal{E})\cap C_c(X)$ is an algebra, and on the space $\mathcal{D}(\mathcal{E})\cap C_c(X)\otimes \mathcal{D}(\mathcal{E})\cap C_c(X)$ we can introduce a non-negative definite symmetric bilinear form by extending
\[\left\langle a\otimes b, c\otimes d\right\rangle_{\mathcal{H}}:=\mathcal{E}(abd,c)+\mathcal{E}(a,bcd)-\mathcal{E}(ac,bd).\]
Let $\left\|\cdot\right\|_{\mathcal{H}}=\sqrt{\left\langle\cdot,\cdot\right\rangle_{\mathcal{H}}}$ be the associated Hilbert seminorm. 
\begin{definition}\label{D:Habstract}
The Hilbert space $\mathcal{H}$ of $L^2$-differential $1$-forms associated with $(\mathcal{E},\mathcal{D}(\mathcal{E}))$ is defined as the completion of the quotient space $(\mathcal{D}(\mathcal{E})\cap C_c(X)\otimes \mathcal{D}(\mathcal{E})\cap C_c(X))/\ker \left\|\cdot\right\|_{\mathcal{H}}$ obtained by factoring out zero seminorm elements.
\end{definition} 

One can then introduce an abstract derivation operator associated with $(\mathcal{E},\mathcal{D}(\mathcal{E}))$.
\begin{definition}\label{D:partial}
We define the abstract derivation $\partial:\mathcal{D}(\mathcal{E})\cap C_c(X)\to\mathcal{H}$ 
associated with $(\mathcal{E},\mathcal{D}(\mathcal{E}))$ by\[\partial f:=f\otimes \mathbbm{1},\quad f\in\mathcal{D}(\mathcal{E})\cap C_c(X).\]
\end{definition}

One can define uniformly bounded actions of the algebra $\mathcal{D}(\mathcal{E})\cap C_c(X)$ on $\mathcal{H}$ and then see that $\partial$ satisfies a Leibniz rule. The operator $\partial$ extends to a closed operator $\partial: L^2(X,\mu)\to\mathcal{H}$ with domain $\mathcal{D}(\mathcal{E})$. For these facts and further details see for instance \cite[Section 2]{HRT13}. 

\begin{lemma}\label{L:omegaex}
Let $(\mathcal{E},\mathcal{D}(\mathcal{E}))$ be a regular Dirichlet form on $L^2(X,\mu)$ that admits a carr\'e du champ. Then, there exists some $\omega\in\mathcal{H}$ such that
\begin{equation}\label{eq:HSO01}
\int_X f\,d\mu=\langle f\,\omega,\omega\rangle_{\mathcal{H}}
\end{equation}
for any $f\in b\mathcal{B}(X)$.
\end{lemma}

Following \cite[Chapter 3]{Eb99} it was shown in \cite[Section 2]{HRT13} that since $(\mathcal{E},\mathcal{D}(\mathcal{E}))$ admits a carr\'e du champ, the space $\mathcal{H}$ is isometrically isomorphic to the direct integral with respect to $\mu$ of a certain measurable field of Hilbert spaces $(\mathcal{H}_{x})_{x\in X}$, see~\cite[Theorem 2.1]{HRT13}. These `fibers' $\mathcal{H}_x$ play the role of (co-)tangent spaces. For $f,g\in\mathcal{D}(\mathcal{E})$ it holds that $\left\langle \partial f,\partial g\right\rangle_{\mathcal{H}_x}=
\Gamma(f,g)(x)$ for $\mu$-a.e.\ $x\in X$, where $\Gamma$ denotes the carr\'e du champ operator. If $(\mathcal{E},\mathcal{D}(\mathcal{E}))$ is strongly local, then for $\mu$-a.e.\ $x\in X$ the dimension of $\mathcal{H}_x$ equals the pointwise Kusuoka-Hino index $p(x)$ of $(\mathcal{E},\mathcal{D}(\mathcal{E}))$ at $x$. See \cite[Definition 2.9]{Hino10} and \cite[Proposition 4.2]{BK17}.

\begin{proof}
There is a $1$-form $\widetilde{\omega}\in\mathcal{H}$ such that $0<\left\|\widetilde{\omega}\right\|_{\mathcal{H}_{x}}<+\infty$ for all $x\in M$, where $M\subset X$ is a Borel set with $\mu(M^c)=0$.
For instance, we can use $\widetilde{\omega}=\sum_{i=1}^\infty 2^{-i}\eta_i$, where $(\eta_i)_i$ is a measurable field of orthonormal bases, see \cite[Proposition II.4.1]{Dix81}, \cite[Lemma 8.12]{Tak02} or \cite[Remark 2.4]{HRT13}. Now consider $\omega:=\mathbbm{1}_M(\left\|\widetilde{\omega}\right\|_{\mathcal{H}_{\cdot}})^{-1}\widetilde{\omega}$.
\end{proof}

Also the following definition makes sense in this general setup.

\begin{definition} Let $(\mathcal{E},\mathcal{D}(\mathcal{E}))$ be a regular Dirichlet form on $L^2(X,\mu)$ admitting a carr\'e du champ.
For each $\omega\in\mathcal{H}$ such that $\|\omega\|_{\mathcal{H}_x}=1$ $\mu$-a.e. $\in X$ define
\begin{align}\label{eq:HSO03}
\star_\omega\colon&L^2(X,\mu)\times\mathcal{H}\quad\!\longrightarrow\quad L^2(X,\mu)\times\mathcal{H}\nonumber\\
&\qquad\quad\; (f,\eta)\,\quad\longmapsto\quad (\langle \eta,\omega\rangle_{\mathcal{H}_\cdot},f\!\cdot\omega).
\end{align}
\end{definition}

If the pointwise Kusuoka-Hino index of $(\mathcal{E},\mathcal{D}(\mathcal{E}))$ is one $\mu$-a.e., then we have $\dim \mathcal{H}_{x}=1$ $\mu$-a.e. If $\omega\in\mathcal{H}$ is such that $\left\|\omega\right\|_{\mathcal{H}_x}=1$ for $\mu$-a.e.\ $x\in X$, then each $\eta\in\mathcal{H}$ has the form $g\,\omega$ with some uniquely defined $g\in L^2(X,\mu)$. In this case, (\ref{eq:HSO03}) rewrites $(f,g\,\omega)\mapsto (g, f\,\omega)$. This motivates the following observation made in \cite[Proposition 4.5]{BK17}.

\begin{proposition}\label{prop:Hi1.01}
Let $(\mathcal{E},\mathcal{D}(\mathcal{E}))$ be a regular Dirichlet form on $L^2(X,\mu)$ admitting a carr\'e du champ and having pointwise index one $\mu$-a.e. Then, for each $\omega\in\mathcal{H}$ such that $\|\omega\|_{\mathcal{H}_x}=1$ $\mu$-a.e.\ $x\in X$, the linear operator defined by 
\begin{align}\label{eq:Hi1.01}
\star_{\omega}\colon&L^2(X,\mu)\quad\!\longrightarrow\quad\; \mathcal{H}\nonumber\\
&\qquad f\qquad\longmapsto\quad f\!\cdot\omega
\end{align}
is an isometry, both fiberwise and globally, i.e. 
$|f(x)|=\|{\star}_{\omega}f\|_{\mathcal{H}_x}$ $\mu$-a.e. $x\in X$ and $\|f\|_{L^2(X,\mu)}=\|{\star}_{\omega}f\|_{\mathcal{H}}$.
\end{proposition}

If $\omega\in \mathcal{H}$ as in Proposition \ref{prop:Hi1.01} is fixed, and $\eta\in\mathcal{H}$, according to the above we have $\eta=g\,\omega$ with a uniquely defined function $g\in L^2(X,\mu)$. This shows that $\star_\omega$ is a bijection. As in the classical theory we denote its inverse $(\star_\omega)^{-1}$ again by the symbol $\star_{\omega}$, so that 
 \begin{align}\label{eq:Hi1.01inv}
\star_{\omega}\colon&\qquad\mathcal{H}\quad\!\longrightarrow\quad\; L^2(X,\mu)\nonumber\\
& \eta=g\,\omega\ \longmapsto\quad g.
\end{align}

\begin{definition}\label{D:Hodgestar}
Assume that the pointwise Kusuoka-Hino index of $(\mathcal{E},\mathcal{D}(\mathcal{E}))$ is one $\mu$-a.e. Let $\omega\in\mathcal{H}$ be such that $\left\|\omega\right\|_{\mathcal{H}_x}=1$ for $\mu$-a.e.\ $x\in X$. To the operator $\star_\omega$ as in (\ref{eq:Hi1.01}) and to its inverse in (\ref{eq:Hi1.01inv}) we refer as the \emph{Hodge star operator with respect to $\omega$}.
\end{definition}

It follows immediately from the definition that for $\eta\in\mathcal{H}$ and $f\in b\mathcal{B}(\Omega)$ it holds that
\begin{equation}\label{E:Hodgeandaction}
\star_\omega (f\eta)=f\star_\omega \eta.
\end{equation}

{ 
 \section{Dynamical spectrum and invariant sets}\label{SS:DsIs}
 This section is presented here in connection to Corollary~\ref{c:S}. It opens the possibility to study spectral analysis of the heat semigroup and the Laplacian, which will be the subject of future work. 
 
 
 
 
 Recall that 
 in the present setting the \textit{Koopman operators} $\{\Koo_{\vec{t}}\}_{{\vec{t}}\in\mathbb{R}^d}$ provide a family of unitary operators on $L^2(\Omega,\mu)$ defined by
 \eqref{E:KoopmanP}
 for ${\vec{t}}\in\mathbb{R}^d$. An eigenfunction of $\Koo_{\vec{t}\,}$ 
 with eigenvalue ${{{\vec{\alpha}}}}\in\mathbb{R}^d$ 
 is a nonzero function $f\in L^2$ satisfying
 \begin{equation}
 \label{eqn:evalue}
 \Koo_{\vec{t}} f = e^{2\pi i \langle {\vec{t}},{\vec{\alpha}} \rangle} f
 \end{equation}
 for all ${\vec{t}}\in\mathbb{R}^d$. The collection of all eigenvalues is called the \textit{dynamical spectrum}. This spectral approach allows us to characterize ergodicity by the action of the Koopman operator on $L^2(\Omega,\mu)$.
 
 In the particular case of $\Omega$ being the pattern space of an aperiodic Delone set of finite local complexity, the dynamical spectrum carries a great amount of information about the structure of the canonical transversal. In addition, the dynamical spectrum is related to the so-called diffraction spectrum of the quasicrystal modeled by the Delone set. Thus, the spectrum is very important from a dynamics point of view as well as from a mathematical physics point of view. We refer the reader to \cite{BaakeLenzvanEnter15, BaakeLenz17} for up-to-date perspectives on the dynamical spectrum and its role in aperiodic order.
 \begin{definition}
  The Koopman family of operators has \emph{pure point spectrum} if $L^2(\Omega,\mu)$ admits a basis of eigenfunctions.
\end{definition}
 \begin{proposition}
 Let $\varphi_{\vec{t}}:\Omega\rightarrow \Omega$ be a minimal and uniquely ergodic action of $\mathbb{R}^d$ on a compact metric space $\Omega$ with pure discrete spectrum and denote by $\mu$ the unique $\mathbb{R}^d$-invariant probability measure for this action. If, for ${\vec{\tau}}\in\mathbb{R}^d$, we have that $\langle{\vec{\tau}},{\vec{\alpha}}\rangle\not\in\mathbb{Z}$ for all eigenvalues ${\vec{\alpha}}$ corresponding to $f\in L^2(\Omega,\mu)$ which are non-constant $\mu$-almost everywhere, then for any $\mu$-measurable set $A\subset \Omega$ with the property that $\varphi_{\vec{\tau}}(A) = A$ it is true that $\mu(A)\in \{0,1\}$.
 \end{proposition}
 \begin{proof}
 Recall that an action on a compact metric space $\Omega$ is ergodic with respect to an invariant probability measure $\mu$ if there are no non-constant invariant functions in $L^2(\Omega,\mu)$ \cite[\S 2]{EW:book}. As such, for a fixed ${\vec{\tau}}\in\mathbb{R}^d$, the map $\varphi_{\vec{\tau}}:\Omega\rightarrow\Omega$ is ergodic if and only if $\langle {\vec{\tau}},{\vec{\alpha}}\rangle\not\in\mathbb{Z}$. Indeed, if $\langle {\vec{\tau}},{\vec{\alpha}}\rangle\in\mathbb{Z}$ for an eigenvalue ${\vec{\alpha}}$ corresponding to a non-constant function, then by (\ref{eqn:evalue}) and the fact that $L^2(\Omega,\mu)$ is generated by eigenfunctions, we have an invariant function in $L^2(\Omega,\mu)$ which is not constant $\mu$-almost everywhere, which is equivalent to the map $\varphi_{\vec{\tau}}$ not being ergodic with respect to $\mu$.
 
 Thus, if $\langle {\vec{\tau}},{\vec{\alpha}}\rangle\not\in\mathbb{Z}$ for all eigenvalues ${\vec{\alpha}}$ associated with non-constant functions in $L^2(\Omega,\mu)$, then the map $\varphi_{\vec{\tau}}$ is ergodic with respect to $\mu$. As such, every $\mu$-measurable set $A$ which is invariant under $\varphi_{\vec{\tau}}$ either has full or null measure.
 \end{proof}
 
}


\begin{thebibliography}{100}
	
	\bibitem{A1}
	Eric Akkermans.
	\newblock Statistical mechanics and quantum fields on fractals.
	\newblock In {\em Fractal Geometry and Dynamical Systems in Pure and Applied
		Mathematics II: Fractals in Applied Mathematics}, volume 601, pages 1--22.
	AMS, Providence, USA, 2013.
	
	\bibitem{A2}
	Eric Akkermans, Gerald Dunne, and Eli Levy.
	\newblock Wave propagation in one-dimension: methods and applications to
	complex and fractal structures.
	\newblock In {\em Optics of Aperiodic Structures}, volume Dal Negro, L. (Ed.).,
	pages 407--450. New York: Pan Stanford., 2013.
	
	\bibitem{A3}
	Eric Akkermans, Gerald Dunne, and Alexander Teplyaev.
	\newblock Thermodynamics of photons on fractals.
	\newblock {\em Physical review letters}, 105(23):230407, 2010.
	
	\bibitem{ACCDP:survey}
	Jos\'e Aliste-Prieto, Daniel Coronel, Mar\'\i a~Isabel Cortez, Fabien Durand,
	and Samuel Petite.
	\newblock Linearly repetitive {D}elone sets.
	\newblock In {\em Mathematics of aperiodic order}, volume 309 of {\em Progr.
		Math.}, pages 195--222. Birkh\"auser/Springer, Basel, 2015.
	
	\bibitem{PARdiamonds}
	Patricia Alonso-Ruiz.
	\newblock Explicit formulas for heat kernels on diamond fractals.
	\newblock {\em arXiv:1712.00385 To appear in: Comm. in Mathematical Physics},
	2018.
	
	\bibitem{AvilaDamanikZhang}
	Artur Avila, David Damanik, and Zhenghe Zhang.
	\newblock Singular density of states measure for subshift and quasi-periodic
	{S}chr\"odinger operators.
	\newblock {\em Comm. Math. Phys.}, 330(2):469--498, 2014.
	
	\bibitem{BG:book1}
	Michael Baake and Uwe Grimm.
	\newblock {\em Aperiodic order. {V}ol. 1}, volume 149 of {\em Encyclopedia of
		Mathematics and its Applications}.
	\newblock Cambridge University Press, Cambridge, 2013.
	\newblock A mathematical invitation, With a foreword by Roger Penrose.
	
	\bibitem{BaakeLenz17}
	Michael Baake and Daniel Lenz.
	\newblock Spectral notions of aperiodic order.
	\newblock {\em Discrete Contin. Dyn. Syst. Ser. S}, 10:161--190, 2017.
	
	\bibitem{BaakeLenzvanEnter15}
	Michael Baake, Daniel Lenz, and Aernout van Enter.
	\newblock Dynamical versus diffraction spectrum for structures with finite
	local complexity.
	\newblock {\em Ergodic Theory Dynam. Systems}, 35:2017--2043, 2015.
	
	\bibitem{A5}
	Florent Baboux, Eli Levy, Aristide Lema{\^\i}tre, Carmen G{\'o}mez, Elisabeth
	Galopin, Luc Le~Gratiet, Isabelle Sagnes, Alberto Amo, Jacqueline Bloch, and
	Eric Akkermans.
	\newblock Measuring topological invariants from generalized edge states in
	polaritonic quasicrystals.
	\newblock {\em Physical Review B}, 95(16):161114, 2017.
	
	\bibitem{BGL14}
	Dominique Bakry, Ivan Gentil, and Michel Ledoux.
	\newblock {\em Analysis and {G}eometry of {M}arkov {D}iffusion {O}perators},
	volume 348 of {\em Grundlehren der math. Wiss.}
	\newblock Springer, New York, 2014.
	
	\bibitem{Bandt}
	C.~Bandt and P.~Gummelt.
	\newblock Fractal {P}enrose tilings. {I}. {C}onstruction and matching rules.
	\newblock {\em Aequationes Math.}, 53(3):295--307, 1997.
	
	\bibitem{BBKT10}
	Martin~T. Barlow, Richard~F. Bass, Takashi Kumagai, and Alexander Teplyaev.
	\newblock Uniqueness of {B}rownian motion on {S}ierpinski carpet.
	\newblock {\em J. Eur. Math. Soc.}, 12:655--701, 2010.
	
	\bibitem{BE04}
	Martin~T. Barlow and Steven~N. Evans.
	\newblock Markov processes on vermiculated spaces.
	\newblock In {\em Random walks and geometry}, pages 337--348. Walter de
	Gruyter, Berlin, 2004.
	
	\bibitem{Bass95}
	Richard~F. Bass.
	\newblock {\em Probabilistic techniques in analysis}.
	\newblock Probability and its Applications (New York). Springer-Verlag, New
	York, 1995.
	
	\bibitem{BG12}
	Richard~F. Bass and Maria Gordina.
	\newblock Harnack inequalities in infinite dimensions.
	\newblock {\em J. Funct. Anal.}, 263(11):3707--3740, 2012.
	
	\bibitem{BK17}
	Fabrice Baudoin and Daniel~J Kelleher.
	\newblock Differential one-forms on {D}irichlet spaces and {B}akry-{{E}}mery
	estimates on metric graphs.
	\newblock {\em arXiv:1604.02520, to apear in Trans. Amer. Math. Soc.}, pages
	1--42, 2017.
	
	\bibitem{Bellissard00}
	J.~Bellissard, D.~J.~L. Herrmann, and M.~Zarrouati.
	\newblock Hulls of aperiodic solids and gap labeling theorems.
	\newblock In {\em Directions in mathematical quasicrystals}, volume~13 of {\em
		CRM Monogr. Ser.}, pages 207--258. Amer. Math. Soc., Providence, RI, 2000.
	
	\bibitem{BBG06}
	Jean Bellissard, Riccardo Benedetti, and Jean-Marc Gambaudo.
	\newblock Spaces of tilings, finite telescopic approximations and gap-labeling.
	\newblock {\em Comm. Math. Phys.}, 261:1--41, 2006.
	
	\bibitem{Bellissard15}
	Jean~V. Bellissard.
	\newblock Delone sets and material science: a program.
	\newblock In {\em Mathematics of aperiodic order}, volume 309 of {\em Progr.
		Math.}, pages 405--428. Birkh\"auser/Springer, Basel, 2015.
	
	\bibitem{BCR11}
	Lucian Beznea, Aurel Cornea, and Michael R\"ockner.
	\newblock Potential theory of infinite dimensional {L}\'evy processes.
	\newblock {\em J. Funct. Anal.}, 261(10):2845--2876, 2011.
	
	\bibitem{BH91}
	Nicholas Bouleau and Francis Hirsch.
	\newblock {\em Dirichlet {F}orms and {A}nalysis on {W}iener {S}pace}.
	\newblock deGruyter Studies in Mathematics. Walter deGruyter, Berlin, 1991.
	
	\bibitem{Breuillard}
	E.~Breuillard.
	\newblock Distributions diophantiennes et th\'eor\`eme limite local sur {$\Bbb
		R^d$}.
	\newblock {\em Probab. Theory Related Fields}, 132(1):39--73, 2005.
	
	\bibitem{Candel}
	Alberto Candel.
	\newblock The harmonic measures of {L}ucy {G}arnett.
	\newblock {\em Adv. Math.}, 176(2):187--247, 2003.
	
	\bibitem{CanCon00}
	Alberto Candel and Lawrence Conlon.
	\newblock {\em Foliations I}.
	\newblock Amer. Math. Soc., Providence, 2000.
	
	\bibitem{CanCon03}
	Alberto Candel and Lawrence Conlon.
	\newblock {\em Foliations II}.
	\newblock Amer. Math. Soc., Providence, 2003.
	
	\bibitem{CKS87}
	E.A. Carlen, S.~Kusuoka, and D.W. Stroock.
	\newblock Upper bounds for symmetric {M}arkov transition functions.
	\newblock {\em Ann. Inst. Henri Poincar\'e}, 23(2):245--287, 1987.
	
	\bibitem{Carron02}
	Gilles Carron.
	\newblock ${L}^2$-harmonic forms on non-compact {R}iemannian manifolds.
	\newblock {\em Proc. Centre Math. Appl. Austral. Nat. Univ.}, 40:49--59, 2002.
	
	\bibitem{Cheeger}
	J.~Cheeger.
	\newblock Differentiability of {L}ipschitz functions on metric measure spaces.
	\newblock {\em Geom. Funct. Anal.}, 9(3):428--517, 1999.
	
	\bibitem{Chen09}
	Zhen-Qing Chen.
	\newblock On notions of harmonicity.
	\newblock {\em Proc. Amer. Math.Soc.}, 137:3497--3510, 2009.
	
	\bibitem{ChFu12}
	Zhen-Qing Chen and Masatoshi Fukushima.
	\newblock {\em {S}ymmetric {M}arkov {P}rocesses, {T}ime {C}hange, and
		{B}oundary {T}heory}, volume~35 of {\em {L}ondon {M}athematical {S}ociety
		{M}onographs}.
	\newblock {P}rinceton {U}niversity {P}ress, {P}rinceton and {O}xford, 2012.
	
	\bibitem{CS03}
	F.~Cipriani and J.-L. Sauvageot.
	\newblock Derivations as square roots of {D}irichlet forms.
	\newblock {\em J. Funct. Anal.}, 201:78--120, 2003.
	
	\bibitem{CortezNavas}
	Mar\'\i a~Isabel Cortez and Andr\'es Navas.
	\newblock Some examples of repetitive, nonrectifiable {D}elone sets.
	\newblock {\em Geom. Topol.}, 20(4):1909--1939, 2016.
	
	\bibitem{Coulhon96}
	Thierry Coulhon.
	\newblock Ultracontractivity and nash type inequalities.
	\newblock {\em J. Funct. Anal.}, 141:510--539, 1996.
	
	\bibitem{DamanikGorodetskiYessen}
	David Damanik, Anton Gorodetski, and William Yessen.
	\newblock The {F}ibonacci {H}amiltonian.
	\newblock {\em Invent. Math.}, 206(3):629--692, 2016.
	
	\bibitem{A8}
	A.~Dareau, E.~Levy, M.~Bosch Aguilera, R.~Bouganne, E.~Akkermans, F.~Gerbier,
	and J.~Beugnon.
	\newblock Revealing the topology of quasicrystals with a diffraction
	experiment.
	\newblock {\em Phys. Rev. Lett.}, 119:215304, Nov 2017.
	
	\bibitem{Davies89}
	Edward~Brian Davies.
	\newblock {\em Heat kernels and spectral theory}.
	\newblock Cambride University Press, Cambridge, 1989.
	
	\bibitem{DaviesSimon84}
	Edward~Brian Davies and Barry Simon.
	\newblock Ultracontractivity and the heat kernel for schr\"odinger operators
	and dirichlet laplacians.
	\newblock {\em J. Funct. Anal.}, 59:335--395, 1984.
	
	\bibitem{DG1}
	Manfred Denker and Mikhail Gordin.
	\newblock Gibbs measures for fibred systems.
	\newblock {\em Adv. Math.}, 148(2):161--192, 1999.
	
	\bibitem{divincenzo1999quasicrystals}
	David~P DiVincenzo and Paul~J Steinhardt.
	\newblock {\em Quasicrystals: the State of the Art}, volume~16.
	\newblock World scientific, 1999.
	
	\bibitem{Dix81}
	J~Dixmier.
	\newblock {\em Von {N}eumann {A}lgebras}.
	\newblock North-Holland Math. Lib. 27. North-Holland, Amsterdam, 1981.
	
	\bibitem{Dodziuk77}
	Jozef Dodziuk.
	\newblock De {R}ham-{H}odge theory for {$L^2$}-cohomology of infinite
	coverings.
	\newblock {\em Topology}, 16:157–165, 1977.
	
	\bibitem{Dodziuk82}
	Jozef Dodziuk.
	\newblock ${L}^2$-harmonic forms on complete manifolds.
	\newblock {\em Semin. differential geometry, Ann. Math. Stud.}, 102:291--302,
	1982.
	
	\bibitem{DolgopyatFernando}
	Dmitry Dolgopyat and Kasun Fernando.
	\newblock An error term in the central limit theorem for sums of discrete
	random variables.
	\newblock {\em preprint}, 2018.
	
	\bibitem{Doob54}
	Joseph~L. Doob.
	\newblock Semimartingales and subharmonic functions.
	\newblock {\em Trans. Amer. Math. Soc.}, 77(11):86--121, 1954.
	
	\bibitem{Doob84}
	Joseph~L. Doob.
	\newblock {\em Classical {P}otential {T}heory and {I}ts {P}robabilistic
		{C}ounterpart}, volume 262 of {\em Grundlehren der mathematischen
		Wissenschaften, Band 262}.
	\newblock Springer-Verlag, New York, 1984.
	
	\bibitem{duneau1985quasiperiodic}
	Michel Duneau and Andr{\'e} Katz.
	\newblock Quasiperiodic patterns.
	\newblock {\em Physical review letters}, 54(25):2688, 1985.
	
	\bibitem{Dynkin}
	E.~B. Dynkin.
	\newblock {\em Markov processes. {V}ols. {I}, {II}}.
	\newblock Translated with the authorization and assistance of the author by J.
	Fabius, V. Greenberg, A. Maitra, G. Majone. Die Grundlehren der
	Mathematischen Wissenschaften, B\"ande 121 und 122. Academic Press Inc.,
	Publishers, New York; Springer-Verlag, Berlin-G\"ottingen-Heidelberg, 1965.
	
	\bibitem{Eb99}
	Andreas Eberle.
	\newblock {\em Uniqueness and non-uniqueness of semigroups generated by
		singular diffusion operators}, volume 1718 of {\em Lect. Notes Math.}
	\newblock Springer, New York, 1999.
	
	\bibitem{EE87}
	D.~E. Edmunds and W.~D. Evans.
	\newblock {\em Spectral theory and differential operators}.
	\newblock Oxford Mathematical Monographs. The Clarendon Press, Oxford
	University Press, New York, 1987.
	\newblock Oxford Science Publications.
	
	\bibitem{EW:book}
	Manfred Einsiedler and Thomas Ward.
	\newblock {\em Ergodic theory with a view towards number theory}, volume 259 of
	{\em Graduate Texts in Mathematics}.
	\newblock Springer-Verlag London, Ltd., London, 2011.
	
	\bibitem{elser1985crystal}
	Veit Elser and Christopher~L Henley.
	\newblock Crystal and quasicrystal structures in al-mn-si alloys.
	\newblock {\em Physical Review Letters}, 55(26):2883, 1985.
	
	\bibitem{EthierKurtz86}
	Stuart~N. Ethier and Thomas~G. Kurtz.
	\newblock {\em Markov {P}rocesses, {C}haracterization and {C}onvergence}.
	\newblock Wiley Series in Probability and Mathematical Statistics. Wiley, New
	York, 1986.
	
	\bibitem{Evans98}
	Lawrence~C. Evans.
	\newblock {\em Partial {D}ifferential {E}quations}, volume~19 of {\em Graduate
		Studies in Mathematics}.
	\newblock American Mathematical Society, Providence, Rhode Island, 1998.
	
	\bibitem{FHK:CCPT}
	A.H. Forrest, J.R. Hunton, and J.~Kellendonk.
	\newblock Cohomology of canonical projection tilings.
	\newblock {\em Comm. Math. Phys.}, 226:289--322, 2002.
	
	\bibitem{Fukushima16}
	Masatoshi Fukushima.
	\newblock Liouville property of harmonic functions of finite energy for
	{D}irichlet forms.
	\newblock In {\em Stochastic Partial Differential Equations and Related Fields
		- In Honor of Michael R\"ockner, SPDERF, Bielefeld, Germany, October 10-14,
		2016}, Springer Proceedings in Mathematics and Statistics. Springer, 2016.
	
	\bibitem{FOT94}
	Masatoshi Fukushima, Yoichi Oshima, and Masayoshi Takeda.
	\newblock {\em Dirichlet {F}orms and {S}ymmetric {M}arkov {P}rocesses},
	volume~19 of {\em de{G}ruyter {S}tudies in {M}athematics}.
	\newblock Walter de{G}ruyter, Berlin, 1994.
	
	\bibitem{Gardner77}
	M.~Gardner.
	\newblock Extraordinary nonperiodic tiling that enriches the theory of tiles.
	\newblock {\em Scientific American}, pages 110--119, December 1977.
	
	\bibitem{GarnettL}
	Lucy Garnett.
	\newblock Foliations, the ergodic theorem and {B}rownian motion.
	\newblock {\em J. Funct. Anal.}, 51(3):285--311, 1983.
	
	\bibitem{GT01}
	David Gilbarg and Neil~S. Trudinger.
	\newblock {\em Elliptic partial differential equations of second order}.
	\newblock Classics in Mathematics. Springer-Verlag, Berlin, 2001.
	\newblock Reprint of the 1998 edition.
	
	\bibitem{G1}
	M.~I. Gordin.
	\newblock The central limit theorem for stationary processes.
	\newblock {\em Dokl. Akad. Nauk SSSR}, 188:739--741, 1969.
	
	\bibitem{G2}
	M.~I. Gordin and B.~A. Lifsic.
	\newblock Central limit theorem for stationary {M}arkov processes.
	\newblock {\em Dokl. Akad. Nauk SSSR}, 239(4):766--767, 1978.
	
	\bibitem{G3}
	Mikhail Gordin.
	\newblock Homoclinic approach to the central limit theorem for dynamical
	systems.
	\newblock In {\em Doeblin and modern probability ({B}laubeuren, 1991)}, volume
	149 of {\em Contemp. Math.}, pages 149--162. Amer. Math. Soc., Providence,
	RI, 1993.
	
	\bibitem{GP1}
	Mikhail Gordin and Magda Peligrad.
	\newblock On the functional central limit theorem via martingale approximation.
	\newblock {\em Bernoulli}, 17(1):424--440, 2011.
	
	\bibitem{GW1}
	Mikhail Gordin and Michel Weber.
	\newblock On the almost sure central limit theorem for a class of {$\Bbb
		Z^d$}-actions.
	\newblock {\em J. Theoret. Probab.}, 15(2):477--501, 2002.
	
	\bibitem{Gross67}
	Leonard Gross.
	\newblock Abstract {W}iener spaces.
	\newblock In {\em Proc. {F}ifth {B}erkeley {S}ympos. {M}ath. {S}tatist. and
		{P}robability ({B}erkeley, {C}alif., 1965/66), {V}ol. {II}: {C}ontributions
		to {P}robability {T}heory, {P}art 1}, pages 31--42. Univ. California Press,
	Berkeley, Calif., 1967.
	
	\bibitem{Gross67jfa}
	Leonard Gross.
	\newblock Potential theory on {H}ilbert space.
	\newblock {\em J. Functional Analysis}, 1:123--181, 1967.
	
	\bibitem{Gross75}
	Leonard Gross.
	\newblock Logarithmic sobolev inequalities.
	\newblock {\em Amer. J. of Math.}, 97:1061--1083, 1975.
	
	\bibitem{Hansen}
	W.~Hansen.
	\newblock Semipolar sets and quasibalayage.
	\newblock {\em Math. Ann.}, 257(4):495--517, 1981.
	
	\bibitem{HKST}
	Juha Heinonen, Pekka Koskela, Nageswari Shanmugalingam, and Jeremy~T. Tyson.
	\newblock {\em Sobolev spaces on metric measure spaces}, volume~27 of {\em New
		Mathematical Monographs}.
	\newblock Cambridge University Press, Cambridge, 2015.
	\newblock An approach based on upper gradients.
	
	\bibitem{Hino10}
	Masanori Hino.
	\newblock Energy measures and indices of {D}irichlet forms, with applications
	to derivatives on soms fractals.
	\newblock {\em Proc. London Math. Soc.}, 100(3):269--302, 2010.
	
	\bibitem{Hino12}
	Masanori Hino.
	\newblock Measurable {R}iemannian structures associated with strong local
	{D}irichlet forms.
	\newblock {\em Math. Nachr.}, 286(14-15):1466--1478, 2013.
	
	\bibitem{HKT15}
	M.~Hinz, D.~Kelleher, and A.~Teplyaev.
	\newblock Metrics and spectral triples for {D}irichlet and resistance forms.
	\newblock {\em J. Noncomm. Geom.}, 9(2):359--390, 2015.
	
	\bibitem{HRT13}
	M.~Hinz, M.~R\"ockner, and A.~Teplyaev.
	\newblock Vector analysis for {D}irichlet forms and quasilinear {PDE} and
	{SPDE} on fractals.
	\newblock {\em Stoch. Proc. Appl.}, 123:4373--4406, 2013.
	
	\bibitem{H15m}
	Michael Hinz.
	\newblock Magnetic energies and {F}eynman-{K}ac-{I}t\^o formulas for symmetric
	{M}arkov processes.
	\newblock {\em Stoch. Anal. Appl.}, 33(6):1020--1049, 2015.
	
	\bibitem{HT13m}
	Michael Hinz and Alexander Teplyaev.
	\newblock Dirac and magnetic {S}chr\"odinger operators on fractals.
	\newblock {\em J. Funct. Anal.}, 265(11):2830--2854, 2013.
	
	\bibitem{HT15b}
	Michael Hinz and Alexander Teplyaev.
	\newblock Finite energy coordinates and vector analysis on fractals.
	\newblock In {\em Fractal {G}eometry and {S}tochastics {V}}, volume~70 of {\em
		Progress in Probab.}, pages 209--227. Birkh\"auser/Springer, Basel, 2015.
	
	\bibitem{HT15}
	Michael Hinz and Alexander Teplyaev.
	\newblock Local {D}irichlet forms, {H}odge theory, and the {N}avier-{S}tokes
	equation on topologically one-dimensional fractals.
	\newblock {\em Trans. Amer. Math. Soc.}, 367:1347--1380, 2015.
	\newblock Corrected version: arXiv:1206.6644, Corrigendum to appear in Trans.
	Amer. Math. Soc. (2017+).
	
	\bibitem{IRT12}
	M.~Ionescu, L.~Rogers, and A.~Teplyaev.
	\newblock Derivations and {D}irichlet forms on fractals.
	\newblock {\em J. Funct. Anal.}, 263:2141--2169, 2012.
	
	\bibitem{JitomirskayaLiu}
	Svetlana Jitomirskaya and Wencai Liu.
	\newblock Universal hierarchical structure of quasiperiodic eigenfunctions.
	\newblock {\em Ann. of Math. (2)}, 187(3):721--776, 2018.
	
	\bibitem{Jost02}
	J\"urgen Jost.
	\newblock {\em Riemannian {G}eometry and {G}eometric {A}nalysis}.
	\newblock Springer, Heidelberg, 2002.
	
	\bibitem{Kaimanovich89}
	Vadym Kaimanovich.
	\newblock Brownian motion on foliations: entropy, invariant measures, mixing.
	\newblock {\em Funct. Anal. Appl.}, 22:326--328, 1989.
	
	\bibitem{Kajino17}
	Naotaka Kajino.
	\newblock Equivalence of recurrence and {L}iouville property for symmetric
	{D}irichlet forms.
	\newblock {\em Sci. Journal of Volgograd State Univ. Math. Physics and Comp.
		Sim.}, 20(3):89--97, 2017.
	
	\bibitem{Kakutani44}
	Shizuo Kakutani.
	\newblock Two-dimensional brownian motion and harmonic functions.
	\newblock {\em Proc. Imp. Acad. Tokyo}, 20:706--714, 1944.
	
	\bibitem{Kakutani45}
	Shizuo Kakutani.
	\newblock Markoff process and the dirichlet problem.
	\newblock {\em Proc. Japan Acad.}, 21:227--233, 1945.
	
	\bibitem{Kellendonk:PEC}
	Johannes Kellendonk.
	\newblock Pattern-equivariant functions and cohomology.
	\newblock {\em J. Phys. A}, 36(21):5765--5772, 2003.
	
	\bibitem{kellendonk2015mathematics}
	Johannes Kellendonk, Daniel Lenz, and Jean Savinien.
	\newblock {\em Mathematics of aperiodic order}, volume 309.
	\newblock Springer, 2015.
	
	\bibitem{KP00}
	Johannes Kellendonk and Ian Putnam.
	\newblock Tilings, $c^\ast$-algebras, and $k$-theory.
	\newblock In {\em Directions in mathematical quasicrystals}, volume~13 of {\em
		CRM Monograph Series}, pages 177--206. Amer. Math. Soc., Providence, RI,
	2000.
	
	\bibitem{KP:RS}
	Johannes Kellendonk and Ian~F. Putnam.
	\newblock The {R}uelle-{S}ullivan map for actions of {$\Bbb R^n$}.
	\newblock {\em Math. Ann.}, 334(3):693--711, 2006.
	
	\bibitem{KV1}
	C.~Kipnis and S.~R.~S. Varadhan.
	\newblock Central limit theorem for additive functionals of reversible {M}arkov
	processes and applications to simple exclusions.
	\newblock {\em Comm. Math. Phys.}, 104(1):1--19, 1986.
	
	\bibitem{KLObook}
	Tomasz Komorowski, Claudio Landim, and Stefano Olla.
	\newblock {\em Fluctuations in {M}arkov processes}, volume 345 of {\em
		Grundlehren der Mathematischen Wissenschaften [Fundamental Principles of
		Mathematical Sciences]}.
	\newblock Springer, Heidelberg, 2012.
	\newblock Time symmetry and martingale approximation.
	
	\bibitem{KSZ14}
	Pekka Koskela, Nageswari Shanmugalingam, and Yuan Zhou.
	\newblock Geometry and analysis of {D}irichlet forms ({II}).
	\newblock {\em J. Funct. Anal.}, 267(7):2437--2477, 2014.
	
	\bibitem{KZ12}
	Pekka Koskela and Yuan Zhou.
	\newblock Geometry and analysis of {D}irichlet forms.
	\newblock {\em Adv. Math.}, 231(5):2755--2801, 2012.
	
	\bibitem{Kusuoka89}
	Shigeo Kusuoka.
	\newblock Dirichlet forms on fractals and products of random matrices.
	\newblock {\em Publ. Res. Inst. Math. Sci.}, 25:659--680, 1989.
	
	\bibitem{LP03}
	Jeffrey~C. Lagarias and Peter A.~B. Pleasants.
	\newblock Repetitive {D}elone sets and quasicrystals.
	\newblock {\em Ergodic Theory Dynamical Systems}, 23:831--867, 2003.
	
	\bibitem{LMS}
	J.-Y. Lee, R.~V. Moody, and B.~Solomyak.
	\newblock Pure point dynamical and diffraction spectra.
	\newblock {\em Ann. Henri Poincar\'e}, 3(5):1003--1018, 2002.
	
	\bibitem{LT16}
	Daniel Lenz and Alexander Teplyaev.
	\newblock Expansion in generalized eigenfunctions for {L}aplacians on graphs
	and metric measure spaces.
	\newblock {\em Trans. Amer. Math. Soc.}, 368(7):4933--4956, 2016.
	
	\bibitem{levine1984quasicrystals}
	Dov Levine and Paul~Joseph Steinhardt.
	\newblock Quasicrystals: a new class of ordered structures.
	\newblock {\em Physical review letters}, 53(26):2477, 1984.
	
	\bibitem{MSch06}
	Calvin~C. Moore and Claude~L. Schochet.
	\newblock {\em Global analysis on foliated spaces}, volume~9 of {\em
		Mathematical Sciences Research Institute Publications}.
	\newblock Cambridge University Press, New York, second edition, 2006.
	
	\bibitem{Nash58}
	John Nash.
	\newblock Continuity of solutions of parabolic and elliptic equations.
	\newblock {\em Amer. J. Math.}, 80:931--954, 1958.
	
	\bibitem{Nelson66}
	Edward Nelson.
	\newblock A quartic interaction in two dimensions.
	\newblock In {\em Mathematical Theory of Elementary Particles}, pages 69--73.
	MIT, 1966.
	
	\bibitem{Oks00}
	Bernt {\O}ksendal.
	\newblock {\em Stochastic {D}ifferential {E}quations}.
	\newblock Springer-Verlag, Heidelberg, 5 edition, 2000.
	
	\bibitem{PB09}
	John Pearson and Jean Bellissard.
	\newblock Noncommutative {R}iemannian geometry and diffusion on ultrametric
	{C}antor sets.
	\newblock {\em J. Noncomm. Geo.}, 3:447--480, 2009.
	
	\bibitem{Penrose}
	R.~Penrose.
	\newblock Pentaplexity: a class of nonperiodic tilings of the plane.
	\newblock {\em Math. Intelligencer}, 2(1):32--37, 1979/80.
	
	\bibitem{RS80}
	M.~Reed and B.~Simon.
	\newblock {\em Methods of {M}odern {M}athematical {P}hysics, vol. 1}.
	\newblock Acad. Press, San Diego, 1980.
	
	\bibitem{RudinRCA}
	Walter Rudin.
	\newblock {\em Real and complex analysis}.
	\newblock McGraw-Hill Book Co., New York, third edition, 1987.
	
	\bibitem{RudinFA}
	Walter Rudin.
	\newblock {\em Functional analysis}.
	\newblock International Series in Pure and Applied Mathematics. McGraw-Hill,
	Inc., New York, second edition, 1991.
	
	\bibitem{sadun:book}
	Lorenzo Sadun.
	\newblock {\em Topology of tiling spaces}, volume~46 of {\em University Lecture
		Series}.
	\newblock American Mathematical Society, Providence, RI, 2008.
	
	\bibitem{Sadun2015}
	Lorenzo Sadun.
	\newblock Cohomology of hierarchical tilings.
	\newblock In {\em Mathematics of aperiodic order}, volume 309 of {\em Progr.
		Math.}, pages 73--104. Birkh\"auser/Springer, Basel, 2015.
	
	\bibitem{ST:SA}
	Scott Schmieding and Rodrigo Trevi{\~n}o.
	\newblock {Self affine {D}elone sets and deviation phenomena}.
	\newblock {\em preprint arXiv:1511.07557}, pages 1--42, 2015.
	
	\bibitem{SchmiedingTrevino2018}
	Scott Schmieding and Rodrigo Trevi{\~{n}}o.
	\newblock Self affine delone sets and deviation phenomena.
	\newblock {\em Communications in Mathematical Physics}, 357(3):1071--1112, Feb
	2018.
	
	\bibitem{SBGC82}
	D.~Shechtman, I.~Blech, D.~Gratias, and J.W. Cahn.
	\newblock Metallic {P}hase with {L}ong-{R}ange {O}rientational {O}rder and {N}o
	{T}ranslational {S}ymmetry.
	\newblock {\em Phys. Rev. Letters}, 53:1951--1953, 1984.
	
	\bibitem{solomyak:SS}
	Boris Solomyak.
	\newblock Dynamics of self-similar tilings.
	\newblock {\em Ergodic Theory Dynam. Systems}, 17(3):695--738, 1997.
	
	\bibitem{ST17}
	Benjamin Steinhurst and Alexander Teplyaev.
	\newblock Spectral analysis and dirichlet forms on barlow-evans fractals.
	\newblock {\em arXiv preprint arXiv:1204.5207}, 2017.
	
	\bibitem{Sturm1}
	Karl-Theodor Sturm.
	\newblock On the geometry of metric measure spaces. {I}.
	\newblock {\em Acta Math.}, 196(1):65--131, 2006.
	
	\bibitem{Sturm2}
	Karl-Theodor Sturm.
	\newblock On the geometry of metric measure spaces. {II}.
	\newblock {\em Acta Math.}, 196(1):133--177, 2006.
	
	\bibitem{Su15}
	Kiyotaka Suzaki.
	\newblock {An SDE approach to leafwise diffusions on foliated spaces and its
		applications}.
	\newblock {\em Tohoku {M}athematical {J}ournal}, 67(2):247--272, 2015.
	
	\bibitem{Tak02}
	M~Takesaki.
	\newblock {\em Theory of Operator Algebras I}.
	\newblock Encycl. Math. Sci. 124. Springer, New York, 2002.
	
\end{thebibliography}
\end{document}